\documentclass{ps}
\usepackage [utf8]{inputenc}
\usepackage [T1] {fontenc}
\usepackage{xspace} 
\usepackage{graphicx} 
\usepackage{hyperref} 
\usepackage{enumerate}
\usepackage[ruled,vlined,linesnumbered,noresetcount]{algorithm2e} 
\usepackage[toc,page]{appendix} 
\usepackage{geometry}
\usepackage[numbers]{natbib}
\usepackage[english]{babel}
\usepackage{caption} 
\usepackage{subcaption}

\usepackage{amsthm}
\usepackage{amsmath}
\usepackage{amssymb} 
\usepackage{mathtools} 

\newcommand{\R}{\mathbb{R}}

\newcommand{\N}{\mathbb{N}}
\newcommand{\trans}{^{\top}}

\newcommand{\bs}{\boldsymbol}

\newcommand{\corr}{ \bs{ \Sigma }_{ \bs{ \theta } } }

\newcommand{\norme}{ \| \bs{\mu} \| }
\newcommand{\manqueI}{\bs{\mu}_{-i}}
\newcommand{\normeI}{ \| \bs{\mu}_{-i} \| }
\newcommand{\coIrr}{ \bs{ \Sigma }_{ \bs{ \mu } } }
\newcommand{\coI}{ \bs{ \Sigma }_{ \mu_i } }
\newcommand{\co}{ \bs{ \Sigma }_{ \mu } }
\newcommand{\coIrrI}{\bs{\Sigma}_{\bs{\mu}_{-i}}}
\newcommand{\deIrivee}{ \frac{\partial}{\partial \mu_i } \coIrr }

\newcommand{\corrNouveauAncien}{\bs{\Sigma}_{\bs{\theta},0,\cdot}}
\newcommand{\corrAncienNouveau}{\bs{\Sigma}_{\bs{\theta},\cdot,0}}

\DeclareMathOperator{\Tr}{Tr}

\DeclareMathOperator{\argtanh}{argtanh}

\numberwithin{equation}{section}
\newtheorem{thm}{Theorem}[section]
\newtheorem{cor}[thm]{Corollary}
\newtheorem{prop}[thm]{Proposition}
\newtheorem{lem}[thm]{Lemma}
\newtheorem{defn}[thm]{Definition}

\theoremstyle{remark}
\newtheorem*{rmq}{Remark}

\begin{document}


\title{Optimal compromise between incompatible conditional probability distributions, with application to Objective Bayesian Kriging}

\author{Joseph Muré}

\address{EDF R\&D, Dpt. PRISME, 6 quai Watier, 78401 Chatou, France.  E-mail: \textsf{joseph.mure@edf.fr} }

\secondaddress{ Université Paris Diderot, Laboratoire de Probabilités, Statistique et Modélisation, France.}


\begin{abstract}
Models are often defined through conditional rather than joint distributions, but it can be difficult to check whether the conditional distributions are compatible, i.e. whether there exists a joint probability distribution which generates them. When they are compatible, a Gibbs sampler can be used to sample from this joint distribution. When they are not, the Gibbs sampling algorithm may still be applied, resulting in a ``pseudo-Gibbs sampler''. We show its stationary  probability distribution to be the optimal compromise between the conditional distributions, in the sense that it minimizes a mean squared misfit between them and its own conditional distributions. This allows us to perform Objective Bayesian analysis of correlation parameters in Kriging models by using univariate conditional Jeffreys-rule posterior distributions instead of the widely used multivariate Jeffreys-rule posterior. This strategy makes the full-Bayesian procedure tractable. Numerical examples show it has near-optimal frequentist performance in terms of prediction interval coverage.
\end{abstract}

\begin{resume}
Les modèles statistiques sont souvent définis par lois de probabilité conditionnelles plutôt que jointes. Il peut cependant être difficile de contrôler la compatibilité des lois conditionnelles, i.e. l'existence d'une loi jointe qui les engendre. Quand les lois conditionnelles sont compatibles, un échantillonneur de Gibbs peut être utilisé pour échantillonner selon la loi jointe dont elles procèdent. Dans le cas contraire, l'algorithme de l'échantillonneur de Gibbs peut parfois tout de même être appliqué. Une telle procédure est appelée \flqq{} pseudo-échantillonneur de Gibbs \frqq{}. Nous montrons que sa distribution stationnaire est le compromis optimal entre les lois conditionnelles, au sens où elle minimise un écart quadratique moyen entre elles et ses propres lois conditionnelles. Cela permet une analyse bayésienne objective des paramètres de corrélation de modèles de krigeage : nous utilisons des lois \emph{a posteriori} de Jeffreys univariées conditionnelles plutôt que la très usitée loi \emph{a posteriori} de Jeffreys multivariée. Cette stratégie rend la procédure pleinement bayésienne abordable. Des exemples numériques montrent que ses performances fréquentistes en termes de taux de couverture des intervalles prédictifs sont quasi-optimales.
\end{resume}

\subjclass{Primary 62F15; secondary 62M30, 60G15}

\keywords{Incompatibility, Conditional distribution, Markov kernel, Optimal compromise, Kriging, Reference prior, Integrated likelihood, Gibbs sampling, Posterior propriety, Frequentist coverage.}

\maketitle

\selectlanguage{english}


\section{Introduction}

Generally speaking, there are two ways to create statistical models for multiple random variables. One can either consider them simultaneously and directly define their joint distribution, or one can define a system of conditional distributions. The first approach is conceptually easier and often (but not always) leads to models with well understood properties because a closed-form expression is available. The second one allows for more flexibility in modeling but makes theoretical analysis more difficult. \medskip 

The main problem with the second approach is that conditional distributions may not be compatible. In this context, it means that there exists no joint distribution from which the conditional distributions can all be derived. Other definitions of compatibility exist in the literature. For example, in the context of a model with a given prior distribution, \citet{DL01} examine the problem of eliciting a compatible prior distribution for a submodel. In the domain of Bayesian Networks, a probability distribution can be compatible or not with a given Directed Acyclic Graph (DAG) \citep{RC04}. Moreover, an abstraction (simplification) of a DAG can be compatible or not \citep{YM14}. A concept of compatibility of two prior distributions also exists in the field of Bayesian model selection. It is based on the Kullback-Leibler divergence of the corresponding marginal (predictive) distributions \citep{CMR05}. In this paper however, the notion of compatibility concerns families of conditional distributions. A family of conditional distributions is called compatible if there exists a joint distribution that agrees with them all \citep{HC98}. Uniqueness of this joint distribution is desirable but not included in the requirements of compatibility. \medskip

To illustrate this definition of compatibility, consider the following random-effect model (\citep{HC98}, cited by \citet{Rob06} page 41), which shows that ``in general, reasonable-seeming conditional models will not be compatible with any single joint distribution'' \citep{GR01}. Define $y_{ij}= \beta + u_i + \epsilon_{ij}$, $i \in [\![1,I]\!]$ and $j \in [\![1,J]\!]$ where $u_i \sim \mathcal{N}(0,\sigma^2)$ and $\epsilon_{ij} \sim \mathcal{N}(0,\tau^2)$. The parameters of the model being $(\beta,\sigma^2, \tau^2)$, let us consider the prior distribution $\pi(\beta,\sigma^2,\tau^2) \propto (\sigma^2 \tau^2 )^{-1}$: the corresponding posterior is improper. The conditional posterior distributions cannot therefore be compatible. Nevertheless, they are all proper. With $\bar{y}$ and $\bar{u}$ being the empirical means of $\bs{y} = (y_{ij})_{i \in [\![1,I]\!], j \in [\![1,J]\!]}$ and $\bs{u} = (u_i)_{i \in [\![1,I]\!]}$, $\beta |\bs{u},\bs{y},\sigma^2,\tau^2 \sim \mathcal{N}(\bar{y} - \bar{u}, \tau^2 / IJ)$, $\sigma^2 | \bs{u}, \bs{y}, \beta, \tau^2 \sim \mathcal{IG} \left( I/2, \sum_i u_i^2 / 2 \right)$, $\tau^2 | \bs{u}, \beta, \bs{y}, \sigma^2 \sim \mathcal{IG} \left( IJ/2, \sum_{ij} (y_{ij} - u_i  - \beta)^2 /2 \right)$. If the posterior distribution were presented only through these conditionals, one might miss the fact that Gibbs sampling is impossible in this case due to the corresponding Markov chain being null recurrent.  \medskip

For finite state spaces, null recurrent Markov chains are impossible, but attempting to define a joint probability distribution through its conditional distributions may still lead to incompatibility.
Accordingly, the problem of efficiently determining whether a given system of conditionals is compatible has received considerable attention over the years. \citet{KSJ17}, after listing previous attempts, provide probably the best solution to date. Their idea relies on the Structural Ratio Matrix which contains ratios between conditional distributions. \medskip

However, even if a system contains incompatible conditional probability distributions, it does not follow that it is useless. Since \citet{HCMK00}, practitioners have been using systems of conditional probability distributions without reference to compatibility. 
Indeed, providing the Markov chain is positive recurrent, it is always possible to fire up Gibbs samplers to deal with a system of conditional distributions. Some authors use the colorful acronym PIGS for ``Potentially Incompatible Gibbs Sampler'' to describe such a procedure. When the conditionals are definitely known to be incompatible, the most widely used term seems to be ``Pseudo-Gibbs Sampler'' (PGS). \medskip

Behind the practice of PIGS is the intuition that the Gibbs sampler should converge to the joint distribution that best represents the system of conditionals. \citet{KW17} provide a detailed analysis and geometrical interpretation of the behavior of Pseudo-Gibbs Samplers for discrete conditional distributions. In particular, they show how the scanning order determines its stationary distribution. In Section \ref{Sec:compromis_optimal} of the present paper, we provide some theoretical foundation for the intuition that the stationary distribution of a PGS with random scanning order is, in case of uniqueness, 
the best ``compromise'' between incompatible conditionals. Section \ref{Sec:testing_compatibility} provides further discussion of this theory by considering alterations to the main definitions and showing them to lead to undesirable results. \medskip

In Section \ref{Sec:convergence_Gibbs_algorithm}, we use the theory of optimal compromise to derive Objective Bayesian inference on correlation parameters of Kriging models. Kriging models are widely used in spatial statistics, but they are more complex than standard models. Their parameters are numerous, and their interactions complex, which makes eliciting a joint Objective prior difficult.  This makes an approach resting on conditional priors attractive despite the risk of incompatibility of the associated conditional posteriors. It is to deal with such situations that the theory of optimal compromise was developed. \medskip

The Objective Bayesian paradigm as explained by \citet{Ber06} consists in eliciting for every model a ``default'', reasonable prior distribution that could be used when no explicit prior information is available. In particular, the Berger-Bernardo reference prior \citep{Ber05}, hereafter simply named ``reference prior'', can be algorithmically computed with minimal user intervention. \medskip

For models with a single scalar parameter, the reference prior rewards parameter values that are easily discriminated by the likelihood function.  Its definition 
is related to the Kullback-Leibler divergence between posterior and prior \citep{Ber05}. 
For usual continuous models -- essentially models where the Fisher information matrix is equal to the opposite of the expectancy of the second derivative of the log-likelihood, see \citet{CB94} for an exhaustive list of conditions -- it coincides with the Jeffreys-rule prior. \medskip

For models with multiple parameters, the reference prior algorithm requires the user to specify an ordering on the parameters and then iteratively compute the reference prior on each parameter conditionally to all subsequent parameters.  The only user input is therefore this ordering, and common sense arguments often make one more sensible than others. \citet{YB96} list different reference priors obtained with different parameter orderings for a large number of statistical  models.  Of course, one could also group several parameters and treat them as one single multi-dimensional parameter, but doing so tends to produce less satisfactory inference \citep{BB92}. In particular, for usual continuous models \citet{BBS15} state `` 
  We actually know of no multivariable example in which we would recommend the
Jeffreys-rule prior. In higher dimensions, the prior always seems to be either ‘too diffuse’ [...] or ‘too concentrated’ ''. \medskip

\citet{BDOS01} were the first to derive a reference prior for the parameters of a Gaussian Process regression model. This model contained only one correlation parameter, however. When several correlation parameters are involved, there is no reasonable way to order them. Even if one were arbitrarily picked, computation of the prior would be analytically intractable. Several authors \citep{paulo05,RSH12,KP12,RSS13,GWB18} 
 have therefore resolved to treat all correlation parameters as a single multidimensional parameter. It is in order to avoid having to do this that we make use of PIGS. \medskip
 
The idea is simple: for every correlation parameter, it is possible to analytically derive the reference prior for this parameter conditionally to all others. Each of the corresponding posterior distributions can be seen as a conditional probability distribution on one correlation parameter when all others are known. These conditional distributions then serve as input to a PIGS. \medskip

Theorem \ref{Thm:existence_posterior_gibbs} is the main result with respect to the application.

First, under reasonable assumptions, the PIGS admits one single stationary probability distribution. Second, the Markov kernel defined by the PIGS is uniformly ergodic. Since this Markov kernel is defined over an uncountable state space, the latter fact is significant. The stationary distribution, which we call the Gibbs reference posterior distribution, can be used to improve prediction of the value taken by the Gaussian process at unobserved points. Sections \ref{Sec:comp_MLE_MAP} and \ref{Sec:comp_predictions} illustrate the inferential and predictive performance of the stationary distribution respectively.

\section{Optimal compromise: a general theory} \label{Sec:compromis_optimal}

\subsection{Definitions and notations} \label{Sec:compromis_optimal_def}

In this section we introduce the concepts necessary to define the optimal compromise between potentially incompatible conditional distributions. For the sake of readability, all proofs are provided in Appendix \ref{App:compromis_optimal}. \medskip

First, note that in this context, ``conditional distribution'' is really an informal way of referring to a Markov kernel.

\begin{defn}
Let $(A,\mathcal{A})$ and $(B,\mathcal{B})$ be measurable sets. A mapping $\pi:A \times \mathcal{B} \rightarrow [0,1]$ is called a Markov kernel if:
\begin{enumerate}
\item for all $x \in A$, $\pi(x,\cdot): \mathcal{B} \rightarrow [0,1]$ is a probability distribution and
\item for all $S \in \mathcal{B}$, $\pi(\cdot,S): A \rightarrow [0,1]$ is $\mathcal{A}$-measurable.
\end{enumerate}
We use the following notation: for every $(x,S) \in A \times \mathcal{B}$, $\pi(S|x):=\pi(x,S)$.
\end{defn}

Let $r$ be a positive integer and let $(\Omega_1,\mathcal{A}_1)$,...,$(\Omega_r,\mathcal{A}_r)$ be measurable sets. Define $\Omega = \bigtimes_{i=1}^r \Omega_i = \Omega_1 \times ... \times \Omega_r$ and $\mathcal{A}:= \bigotimes_{i=1}^r \mathcal{A}_r = \mathcal{A}_1 \otimes ... \otimes \mathcal{A}_r$. \medskip

For every $i \in [\!|1,r]\!]$, let $\pi_i$ be a Markov kernel $\left( \bigtimes_{j \neq i} \Omega_j \right) \times \mathcal{A}_i \rightarrow [0,1]$. \medskip

Intuitively (we formalize this below), every $\pi_i$ should be assembled with a distribution $m_{\neq i}$ on $\bigotimes_{j \neq i} \mathcal{A}_j$ to create a ``joint'' distribution, that is a probability distribution on $\mathcal{A}$. We refer to every $m_{\neq i}$ ($i \in [\![1,r]\!])$ as an $(r-1)$-dimensional distribution. If the $m_{\neq i}$ can be chosen in such a way as to make all joint distributions equal, then the Markov kernels in the sequence $(\pi_i)_{i \in [\![1,r]\!]}$ are called compatible. And if no choice of $(m_{\neq i})_{i \in [\![1,r]\!]}$ can make all joint distributions equal, we have to look for a ``compromise'' between the Markov kernels. \medskip

Consider the following example \citep{ACS01} with $r=2$ and $\Omega_1 = \Omega_2 = (0,+\infty)$ are endowed with their Borel $\sigma$-algebra: let $\pi_1$ and $\pi_2$ be Markov kernels such that for all $x, y>0$ $\pi_1(\cdot |y)$ and $\pi_2(\cdot |x)$ are absolutely continuous with respect to the Lebesgue measure on $(0,+\infty)$. Let $\lambda$ be this measure and let the densities of $\pi_1$ and $\pi_2$ respectively be

\begin{align}
p_1(x|y) := \frac{d \pi_1(\cdot|y)}{d \lambda} (x) &= (y+2) \exp(-(y+2)x); \\
p_2(y|x) := \frac{d \pi_2(\cdot|x)}{d \lambda} (y) &= (x+3) \exp(-(x+3)y).
\end{align}

Let us define $m_{\neq 1}$ and $m_{\neq 2}$ as the probability distributions with the following densities with respect to the Lebesgue measure on $(0,+\infty)$:

\begin{align}
\frac{d m_{\neq 1}}{d \lambda} (y) & = \frac{ (y+2)^{-1} \exp(-3y) }{ \int_0^{+ \infty} (t+6)^{-1} \exp(-t) dt}; \\
\frac{d m_{\neq 2}}{d \lambda} (x) & = \frac{ (x+3)^{-1} \exp(-2x) }{ \int_0^{+ \infty} (t+6)^{-1} \exp(-t) dt}.
\end{align}

Then the joint probability distributions $\pi_1 m_{\neq 1}$ and $\pi_2 m_{\neq 2}$ are equal. Now denoting by $\lambda$ the Lebesgue measure on $(0,+\infty) \times (0,+\infty)$, the density of $\pi_1 m_{\neq 1} = \pi_2 m_{\neq 2}$ is

\begin{align}
\frac{d \pi_1 m_{\neq 1}}{d \lambda} (x,y) = \frac{d \pi_2 m_{\neq 2}}{d \lambda} (x,y) = \frac{ \exp(-xy -2x -3y) }{ \int_0^{+\infty} (t+6)^{-1} \exp(-t) dt }.
\end{align}

\begin{rmq}[Producing incompatibility is easy]
Take $r=2$ and let $\Omega_1=\Omega_2$ be a Borel subset of $\R$. Assume that for all $x,y \in \Omega_1$, $\pi_1(\cdot|y)$ and $\pi_2(\cdot|x)$ are absolutely continuous with respect to $\lambda$, which denotes here the Lebesgue measure on $\Omega_1$. Let $p_1(\cdot|y)$ and $p_2(\cdot|x)$ be their respective density functions and further assume that for $\lambda$-almost all real numbers $x$ and $y$, $p_1(x|y)>0$ and $p_2(y|x)>0$. 
A necessary condition \citep{ACS01} for the compatibility of $\pi_1$ and $\pi_2$ can be derived from Bayes' rule: there must exist two mappings $u$ and $v$ defined on $\Omega_1$ such that for $\lambda$-almost all real numbers $x$ and $y$

\begin{equation} \label{Eq:compatible_conditionals}
\left. \frac{d \pi_1(\cdot|y)}{d \lambda} (x)\right/ \frac{d \pi_2(\cdot|x)}{d \lambda} (y)= u(x) v(y).
\end{equation}

In the previous example, $\Omega_1 = \Omega_2 = (0,+\infty)$ and this necessary condition is fulfilled: for all $x,y>0$,

\begin{equation}
\frac{p_1(x|y)}{p_2(y|x)}  = \frac{(x+3)^{-1} \exp(-2x)}{(y+2)^{-1} \exp(-3y)}.
\end{equation}

Taking $p_1(x|y) = (y+2) \exp(-(y+2)x)$ and $p_2(y|x)=\exp(-y)$ makes $\pi_1$ and $\pi_2$ fail the necessary condition for compatibility.
\end{rmq}

\begin{defn}
Let $\phi$ be a probability distribution on $\mathcal{A}$. For every $i \in [\![1,r]\!]$, denote by $\phi_{-i}$ the probability distribution on $\bigotimes_{j \neq i} \mathcal{A}_j$ defined as follows. For every set $S_{-i}$ that can be decomposed as $S_{-i} = \bigtimes_{j \neq i} S_j$ (with $S_j \in \mathcal{A}_j$ for every $j \neq i$),
\begin{equation}
\phi_{-i}(S_{-i}) = \phi(\bigtimes_{j<i} S_j \times \Omega_i \times \bigtimes_{k>i} S_k).
\end{equation}
$\phi_{-i}$ is called the $i$-th $(r-1)$-marginal distribution of $\phi$.
\end{defn}

\begin{rmq}
The above definition is valid because any probability distribution on $\mathcal{A}$ can be characterized by its values on ``rectangles'' $\bigtimes_{i=1}^r S_i$ (where for every $i \in [\![1,r]\!]$, $S_i \in \mathcal{A}_i$).
\end{rmq}

For every $i \in [\![1,r]\!]$ and every probability distribution $m_{\neq i}$ on $\bigotimes_{j \neq i} \mathcal{A}_j$, denote by $\pi_i m_{\neq i}$ the distribution on $\mathcal{A}$ defined as follows. For every $i \in [\![1,r]\!]$, for every set $S_{<i} \in \bigotimes_{j < i} \mathcal{A}_j$, every set $S_{>i} \in \bigotimes_{k > i} \mathcal{A}_k$ and every set $S_i \in \mathcal{A}_i$,
 
\begin{equation}
\pi_i m_{\neq i}(S_{<i} \times S_i \times S_{>i}) = \int_{S_{<i} \times S_{>i}} \pi_i(S_i|\omega_{-i}) d m_{\neq i}(\omega_{-i}).
\end{equation}

Naturally, for $i=1$ (resp. $i=r$), remove $S_{<i}$ (resp. $S_{>i}$) from the formula above. In the following, do this kind of operation when $i=1$ or $i=r$. \medskip

Notice that for every $i \in [\![1,r]\!]$, $m_{\neq i}$ is the $i$-th $(r-1)$-marginal distribution of $\pi_i m_{\neq i}$: 

\begin{equation}
\left(\pi_i m_{\neq i}\right)_{-i} = m_{\neq i}.
\end{equation}

If there exists a sequence of $(r-1)$-dimensional distributions $(m_{\neq i})_{i \in [\![1,r]\!]}$ such that all distributions $\pi_i m_{\neq i}$ are equal, then the Markov kernels $(\pi_i)_{i \in [\![1,r]\!]}$ are compatible. If no such sequence $(m_{\neq i})_{i \in [\![1,r]\!]}$ exists, then we wish to find a sequence $(m_{\neq i})_{i \in [\![1,r]\!]}$ that makes the $\pi_i m_{\neq i}$ share some ``common ground''. The following definition expresses this constraint formally.

\begin{defn} \label{Def:compatibilite}
A sequence of $(r-1)$-dimensional distributions $(m_{\neq i})_{i \in [\![1,r]\!]}$ (each $m_{\neq i}$ being a probability distribution on $\bigotimes_{j \neq i} \mathcal{A}_j$) is said to be \emph{compatible} with the sequence of Markov kernels $(\pi_i)_{i \in [\![1,r]\!]} $ if for every $i \in [\![1,r]\!]$ 

\begin{equation}
 m_{\neq i} = \frac{1}{r} \sum_{j=1}^r (\pi_j m_{\neq j})_{-i}.
\end{equation}

\end{defn}

So the ``common ground'' we require for the sequence of distributions $(\pi_i m_{\neq i})_{i \in [\![1,r]_!]}$ is that their $(r-1)$-marginal distributions should be the same on average. Other constraints would have been possible, and we discuss some of them in Section \ref{Sec:compatibility_discussion} below. \medskip

Consider the case where $r=2$, $\Omega_1 = \Omega_2 = \R$ and for all $x,y \in \R$, $\pi_1(\cdot|y) = \mathcal{N}(y/4,1/8)$ and $\pi_2(\cdot|x) = \mathcal{N}(x,1)$, the second argument of $\mathcal{N}(\cdot,\cdot)$ being the variance. Now define $m_{\neq 1} = \mathcal{N}(0,2)$ and $m_{\neq 2} = \mathcal{N}(0,1)$. We have

\begin{align}
\pi_1 m_{\neq 1} = \mathcal{N} \left(
\begin{pmatrix}
0 \\ 0
\end{pmatrix},
\begin{pmatrix}
1 & 1/2 \\ 1/2 & 2
\end{pmatrix}
\right)
\quad
\mathrm{and}
\quad
\pi_2 m_{\neq 2} = \mathcal{N} \left(
\begin{pmatrix}
0 \\ 0
\end{pmatrix},
\begin{pmatrix}
1 & 1 \\ 1 & 2
\end{pmatrix}
\right)
\end{align}

Since $(\pi_1 m_{\neq 1})_{-2} = \mathcal{N}(0,1) = m_{\neq 2}$ and $(\pi_2 m_{\neq 2})_{-1} = \mathcal{N}(0,2) = m_{\neq 1}$, we have \emph{a fortiori} $m_{\neq 1} = 1/2 (\pi_1 m_{\neq 1})_{-1} + 1/2 (\pi_2 m_{\neq 2})_{-1}$ and $m_{\neq 2} = 1/2 \; (\pi_1 m_{\neq 1})_{-2} + 1/2 \; (\pi_2 m_{\neq 2})_{-2}$, so  $(m_{\neq 1},m_{\neq 2})$ is compatible with $(\pi_1,\pi_2)$ in the sense of Definition \ref{Def:compatibilite}. \medskip

The definition of a compromise follows from this new definition of compatibility.

\begin{defn} \label{Def:compromis}

A probability distribution $P$ on $\mathcal{A}$ is called a \emph{compromise} between the Markov kernels $(\pi_i)_{i \in [\![1,r]\!]} $ if these two conditions are verified:
\begin{enumerate}
\item for every $i \in [\![1,r]\!]$, $\pi_i P_{-i}$ is absolutely continuous with respect to $P$;
\item the sequence $(P_{-i})_{i \in [\![1,r]\!]}$ of $P$'s $(r-1)$-marginal distributions 
is compatible with $(\pi_i)_{i \in [\![1,r]\!]} $.
\end{enumerate}
\end{defn}

In the definition of a compromise, the first condition exists to give meaning to the definition of an optimal compromise below. It is reasonable on its own though: a compromise should not deem events impossible if they are considered possible by the Markov kernels. \medskip

Returning to the previous example, both $\pi_1 m_{\neq 1}$ and $\pi_2 m_{\neq 2}$ are compromises, and any convex combination of these two distributions is one as well. The following definition introduces a cost function for compromises in order to determine which compromises are optimal.

\begin{defn} \label{Def:compromis_optimal}
Let $\lambda$ be a positive measure on $\mathcal{A}$. Let $P$ be a compromise between the sequence of Markov kernels $(\pi_i)_{i \in [\![1,r]\!]} $ that is absolutely continuous with respect to $\lambda$. $P$ is called an \emph{optimal compromise} with respect to $\lambda$ between the sequence of Markov kernels $(\pi_i)_{i \in [\![1,r]\!]} $ if it minimizes the functional $E_\lambda$ over all compromises between $(\pi_i)_{i \in [\![1,r]\!]} $ that are absolutely continuous with respect to $\lambda$. $E_\lambda$ is defined by:

\begin{equation} \label{Eq:energie_compromis}
E_\lambda(P) = \sum_{i=1}^r \int_{\mathcal{A}} \left[ \frac{d (\pi_i P_{-i})}{d \lambda}(\omega) - \frac{dP}{d \lambda}(\omega) \right]^2 d \lambda (\omega).
\end{equation}

\end{defn}

In the previous example, the optimal compromise with respect to the Lebesgue measure on $\R \times \R$ can be shown (cf. Section \ref{Sec:Gibbs_compromise_optimal_compromise}) to be $1/2 \; \pi_1 m_{-1} + 1/2 \; \pi_2 m_{-2}$.

\begin{prop} \label{Prop:set_compromises_convex}
The set of all compromises between $(\pi_i)_{i \in [\![1,r]\!]} $ is convex, as is the subset of all compromises absolutely continuous with respect to $\lambda$. 
\end{prop}

If the Markov kernels $(\pi_i)_{i \in [\![1,r]\!]} $ are compatible and there exists a joint distribution $\pi$ on $\mathcal{A}$ that agrees with them all, then for every positive measure $\lambda$ on $\mathcal{A}$ such that $\pi$ is absolutely continuous with respect to $\lambda$, $E_\lambda(\pi)=0$ and $\pi$ is an optimal compromise with respect to $\lambda$. \medskip

Even though Definition \ref{Def:compromis_optimal} makes it seem like the notion of optimal compromise is tied to a reference measure $\lambda$, it turns out that in many situations there exists a compromise that is optimal with respect to all possible reference measures -- cf. Theorem \ref{Thm:Gibbs_compromis_optimal} below. In the previous example, it is given by $1/2 \pi_1 m_{-1} + 1/2 \pi_2 m_{-2}$.

\subsection{Deriving the optimal compromise} \label{Sec:Gibbs_compromise_optimal_compromise}

The concepts of compromise and optimal compromise defined in Section \ref{Sec:compromis_optimal_def} are intimately linked to Gibbs sampling, or (because the Markov kernels are incompatible) pseudo-Gibbs sampling (PGS). Let us therefore recall the Gibbs sampling algorithm.

\begin{algorithm}[H]
\SetAlgoLined
\SetKwInOut{Input}{input}
\Input{Variable $x_i$ for $i \in [\![1,r]\!]$ and Markov kernels $\pi_i$ for $i \in [\![1,r]\!]$}

Initialize $x_1$,...,$x_r$ \;

\While{additional samples are desired}{
Select index $i$ from $[\![1,r]\!]$\;
Sample $x_i$ from $\pi_i(\cdot | x_1,...,x_{i-1},x_{i+1},...,x_r)$ \;
}
 \caption{Gibbs sampling}
\end{algorithm}

The method used to sample the index $i$ is called \emph{scanning order}. A \emph{systematic} scan moves through each index in turn using a deterministic pattern while an \emph{equiprobable random} scan selects for each step an index within $[\![1,r]\!]$ with probability $1/r$. \medskip

If the Markov kernels $(\pi_i)_{i \in [\![1,r]\!]}$ are compatible, then the scanning order influences only the rate of convergence of the algorithm. \citet{HDSMR16} provide theoretical results on the subject and \citet{MM17} propose a measure for the quality of any scan order for Gibbs sampling on finite state spaces. \medskip

If the Markov kernels are not compatible, then every scan order may produce a different target distribution. \citet{KW17} thoroughly study the links between all possible systematic scan orders for Markov kernels on finite state spaces. The equiprobable random scan order has yet another target distribution. \medskip

The notion of Gibbs compromise is central to this theory of compromises between incompatible Markov kernels.

\begin{defn}
A probability distribution $P_G$ on $\mathcal{A}$ is called a \emph{Gibbs compromise} between the sequence of Markov kernels $(\pi_i)_{i \in [\![1,r]\!]}$ if it satisfies:
\begin{equation}
P_G = \frac{1}{r} \sum_{i=1}^r \pi_i (P_G)_{-i}.
\end{equation}
\end{defn}

If it exists, a Gibbs compromise is a stationary distribution for the Gibbs sampler with equiprobable random scan order. This fact makes it practical from a sampling standpoint. In the following, we show how it relates to the concepts of compromise and optimal compromise in the sense of Definitions \ref{Def:compromis} and \ref{Def:compromis_optimal}.

\begin{prop} \label{Prop:Gibbs_compromise_is_compromise}
A Gibbs compromise between the sequence of Markov kernels $(\pi_i)_{i \in [\![1,r]\!]}$ is also a compromise between this sequence of Markov kernels in the sense of Definition \ref{Def:compromis}.
\end{prop}

The denomination ``Gibbs compromise'' is justified because it is a stationary distribution for the Gibbs sampler with random equiprobable scanning order. \medskip 

The proposition below shows that all compromises are tied to Gibbs compromises.

\begin{prop} \label{Prop:hypermarginales_fixes}
If a sequence of $(r-1)$-dimensional probability distributions $(m_{\neq i})_{i \in [\![1,r]\!]}$ (each $m_{\neq i}$ being a probability distribution on $\bigotimes_{j \neq i} \mathcal{A}_j$) is compatible with the sequence of Markov kernels $(\pi_i)_{i \in [\![1,r]\!]}$, then it is the sequence of $(r-1)$-dimensional distributions of a Gibbs compromise between the Markov kernels $(\pi_i)_{i \in [\![1,r]\!]}$.
\end{prop}

\begin{rmq}[Equivalence relation and convexity]
Let us say that two compromises are equivalent if they share the same sequence of $(r-1)$-marginal distributions. This is obviously an equivalence relation and every class of equivalence can be represented by a single Gibbs compromise. Moreover, each class of equivalence is a convex subset of the set of all compromises. And for any positive measure $\lambda$ on $\mathcal{A}$, its intersection with the set of all compromises absolutely continuous with respect to $\lambda$ is also convex. Finally, the functional $E_\lambda$ is convex over this intersection.
\end{rmq}

\subsection{A theoretical justification of Pseudo-Gibbs sampling}

At this stage, all necessary tools are available to derive the two most important results of this theory of compromise between incompatible Markov kernels.

\begin{thm} \label{Thm:Gibbs_compromis_optimal}
If there exists a unique Gibbs compromise $P_G$ between the sequence of Markov kernels $(\pi_i)_{i \in [\![1,r]\!]}$, then the following statements holds:
\begin{enumerate}
\item $P_G$ is absolutely continuous with respect to any compromise between the sequence of Markov kernels $(\pi_i)_{i \in [\![1,r]\!]}$;
\item for any positive measure $\lambda$ on $\mathcal{A}$ such that $P_G$ is absolutely continuous with respect to $\lambda$, $P_G$ is the unique optimal compromise with respect to $\lambda$.
\end{enumerate}
Because of these two properties, we call $P_G$ \emph{the} optimal compromise.
\end{thm}

\begin{rmq}[Equivalence class and convexity -- continued]
The arguments used in the proof of Theorem \ref{Thm:Gibbs_compromis_optimal} (cf. Appendix \ref{App:compromis_optimal}) can also be used to deal with the case where there exist several different Gibbs compromises. First, one can show that each Gibbs compromise is absolutely continuous with respect to any equivalent compromise. Now let $\mathcal{C}$ be an equivalence class and let $\pi_\mathcal{C}$ be the unique Gibbs compromise in this class of equivalence. Let $\lambda$ be a positive measure on $\mathcal{A}$ such that $\pi_\mathcal{C}$ is absolutely continuous with respect to $\lambda$. $\pi_C$ (uniquely) minimizes $E_\lambda$ over the intersection of $\mathcal{C}$ and the set of all compromises absolutely continuous with respect to $\lambda$. This implies that for any positive measure $\lambda$ on $\mathcal{A}$, any optimal compromise with respect to $\lambda$ is a Gibbs compromise. Finally, note that the set of all Gibbs compromises is convex (however there is no reason $E_\lambda$ should be convex over this set!).
\end{rmq}

Theorem \ref{Thm:Gibbs_compromis_optimal} has important practical implications. It opens the possibility of using Gibbs sampling to find the optimal compromise between incompatible Markov kernels and justifies using PIGS. \medskip

The next result shows that, under the conditions of Theorem \ref{Thm:Gibbs_compromis_optimal}, the optimal compromise remains the same for a fairly large class of reparametrizations. This result is key to the application of PIGS in an Objective Bayesian framework, where some degree of invariance by reparametrization of priors and posteriors is usually expected. \medskip

For every $i \in [\![1,r]\!]$, let $(\tilde{\Omega}_i,\tilde{\mathcal{A}}_i)$ be a measurable space and let $f_i$ be bijective measurable mapping $\Omega_i \rightarrow \tilde{\Omega}_i$ whose inverse $f_i^{-1}$ is also measurable. Define $f = (f_1,...,f_r) : \bigtimes_{i \in [\![1,r]\!]} \Omega_i \rightarrow \bigtimes_{i \in [\![1,r]\!]} \tilde{\Omega}_i$ and for every $i \in [\![1,r]\!]$ $f_{-i} = (f_1,..,f_{i-1},f_{i+1},...,f_r) : \bigtimes_{j \neq i} \Omega_j \rightarrow \bigtimes_{j \neq i} \tilde{\Omega}_j$. \medskip

Also let $\tilde{\pi}_i$ be the Markov kernel $\left( \bigtimes_{j \neq i} \tilde{\Omega}_j \right) \times \tilde{\mathcal{A}}_i \rightarrow [0,1]$ such that for every $\omega_{-i} \in \bigtimes_{j \neq i} \Omega_j$ and every $S_i \in \mathcal{A}_i$, $\tilde{\pi}_i(f_i(S_i) | f_{-i}(\omega_{-i})) = \pi_i(S_i|\omega_{-i})$.

\begin{prop} \label{Prop:invariant_reparametrisation}
 Assume there exists a unique Gibbs compromise $P_G$ between the sequence of Markov kernels $(\pi_i)_{i \in [\![1,r]\!]}$. Then the push-forward measure of $P_G$ by $f$ $\tilde{P}_G := P_G \ast f$ is the unique Gibbs compromise between the sequence of Markov kernels $(\tilde{\pi}_i)_{i \in [\![1,r]\!]}$.
\end{prop}

\section{Testing the definitions of compatibility} \label{Sec:testing_compatibility}

While the previous section presented the theory of compromise, this section aims to test its foundations, namely Definitions \ref{Def:compatibilite} and \ref{Def:compromis}. Subsection \ref{Sec:compatibility_discussion} shows that strengthening their requirements is not possible since it would in many cases threaten the very existence of a compromise. Subsection \ref{Sec:compromise_examples} on the other hand shows that these requirements cannot be weakened at a small price.

\subsection{Stronger definitions of compromises are not possible} \label{Sec:compatibility_discussion}

While the definition of the \emph{optimal compromise} is straightforward, as it involves minimizing some measure of distance between the ``targeted'' conditionals and the conditionals of the compromise, the definition of \emph{a compromise} may seem arbitrary. To motivate this definition, let us focus on the two-dimensional case. \medskip

Suppose that $r=2$ and that $\pi_1$ and $\pi_2$ are incompatible. This means there exists no joint distribution $\pi$ which agrees with both Markov kernels. This being the case, it seems sensible to weaken the definition of compatibility by applying it to the ``marginals'' instead of the ``joint'' distribution. The following definition makes this idea precise.

\begin{defn} \label{Def:compatibilite_2d}
A pair of probability distributions $m_{\neq 1}$ (resp. $m_{\neq 2}$) on $\mathcal{A}_2$ (resp. $\mathcal{A}_1$) is \emph{compatible} with the pair of Markov kernels $\pi_1$ and $\pi_2$ if the distributions $\pi_1 m_{\neq 1}$ and $\pi_2 m_{\neq 2}$ verify
\begin{align}
\left( \pi_1 m_{\neq 1} \right)_{-2} = m_{\neq 2} \emph{ and} \left( \pi_2 m_{\neq 2} \right)_{-1} = m_{\neq 1}.
\end{align}
\end{defn}

While this definition may seem more restrictive at first glance than Definition \ref{Def:compatibilite}, both definitions are in fact equivalent when applied to a pair of Markov kernels, because $(r-1)$-dimensional distributions are simply $1$-dimensional distributions in this case. Indeed, following directly from Definition \ref{Def:compatibilite}, we have this result which holds for any $r$:

\begin{prop} \label{Prop:marginales_identiques}
If a sequence of $(r-1)$-dimensional probability distributions $(m_{\neq i})_{i \in [\![1,r]\!]} $ (each $m_{\neq i}$ being a probability distribution on $\bigotimes_{j \neq i} \mathcal{A}_j$) is compatible (in the sense of Definition \ref{Def:compatibilite}) with the sequence of Markov kernels $(\pi_i)_{i \in [\![1,r]\!]} $, then all joint distributions in the sequence $(\pi_i m_{\neq i})_{i \in [\![1,r]\!]} $ share the same marginals, that is

\begin{equation} \label{Eq:weak_compatibility}
\forall i,j,k \in [\![1,r]\!],  \quad \forall S_k \in \mathcal{A}_k, \quad \pi_i m_{\neq i} \left( \bigtimes_{k'<k} \Omega_{k'} \times S_k \times \bigtimes_{k''>k} \Omega_{k''} \right) =  \pi_j m_{\neq j} \left( \bigtimes_{k'<k} \Omega_{k'} \times S_k \times \bigtimes_{k''>k} \Omega_{k''} \right).
\end{equation}
\end{prop}

Now let us consider the three-dimensional case ($r=3$). Because the aim of this section is merely to motivate the definitions of compromises and optimal compromises, there is no need for the discussion to be fully general. Let us therefore restrict the discussion to an important particular case. Assume that $\Omega_1$, $\Omega_2$ and $\Omega_3$ are finite sets and that $\mathcal{A}_1$, $\mathcal{A}_2$ and $\mathcal{A}_3$ are respectively the sets of all their subsets. This has an important consequence: any mapping from a subset of $\Omega$ to a subset of $\Omega$ is measurable. \medskip

Also assume that the Markov kernels $\pi_1$, $\pi_2$ and $\pi_3$ are positive mappings. For $\pi_1$, this means that for every $(\omega_1,\omega_2,\omega_3) \in \Omega_1 \times \Omega_2 \times \Omega_3$, $\pi_1(\{\omega_1\} | \omega_2,\omega_3) > 0$. \medskip

If we consider $\omega_3$ known, then the situation is reduced to the two-dimensional case. Because $\pi_1$ and $\pi_2$ are positive mappings and both $\Omega_1$ and $\Omega_2$ are finite sets, Markov chain theory ensures there exists a unique Gibbs compromise $P(\cdot|\omega_3)$. For every $(\omega_1,\omega_2) \in \Omega_1 \times \Omega_2$, 

\begin{equation} \label{Eq:compromis_ideal}
P(\{\omega_1\} \times \{\omega_2\} |\omega_3) = \frac{1}{2} \pi_1(\{\omega_1\}|\omega_2,\omega_3) P(\Omega_1 \times \{\omega_2\}|\omega_3) + \frac{1}{2} \pi_1(\omega_2|\omega_1,\omega_3) P(\{\omega_1\} \times \Omega_2|\omega_3).
\end{equation}

Thanks to Theorem \ref{Thm:Gibbs_compromis_optimal}, $P(\cdot|\omega_3)$ is the optimal compromise. Moreover, notice that
Equation \eqref{Eq:compromis_ideal} defines a Markov kernel on $\Omega_3 \times \left( \mathcal{A}_1 \otimes \mathcal{A}_2 \right)$. \medskip

We may similarly derive Markov kernels $Q: \Omega_1 \times \left( \mathcal{A}_2 \otimes \mathcal{A}_3 \right)$ and $R:  \Omega_2 \times \left( \mathcal{A}_1 \otimes \mathcal{A}_3 \right)$. \medskip

Once again, because $\pi_1$, $\pi_2$ and $\pi_3$ are positive mappings, it follows from Markov chain theory that $P$, $Q$ and $R$ are also positive mappings. We now show it using only elementary arguments, because these arguments will be useful again later. \medskip

Assume $P$ is not a positive mapping. Then there exists $(\omega_1^{(0)},\omega_2^{(0)},\omega_3^{(0)}) \in \Omega_1 \times \Omega_2 \times \Omega_2$ such that $P(\{\omega_1^{(0)}\} \times \{\omega_2^{(0)}\} | \omega_3^{(0)}) = 0$. Equation \eqref{Eq:compromis_ideal} then implies that $P(\Omega_1 \times \{\omega_2^{(0)}\}|\omega_3^{(0)})=0$. So for every $\omega_1 \in \Omega_1$, $P(\{\omega_1\} \times \{\omega_2^{(0)}\} | \omega_3^{(0)}) = 0$. But then Equation \eqref{Eq:compromis_ideal} implies that $P(\{\omega_1\} \times \Omega_2 | \omega_3^{(0)}) = 0$. Since this holds for every $\omega_1 \in \Omega_1$, $P(\Omega_1 \times \Omega_2 |\omega_3^{(0)}) = 0$, which is absurd since $P(\cdot | \omega_3^{(0)})$ is supposed to be a probability distribution. So $P$ is a positive mapping. \medskip

Ideally, we would wish to define the optimal compromise between $\pi_1$, $\pi_2$ and $\pi_3$ as the joint distribution $\phi$ such that for every $(\omega_1,\omega_2,\omega_3) \in \Omega_1 \times \Omega_2 \times \Omega_3$

\begin{align}
\phi(\{\omega_1\} \times \{\omega_2\} \times \{\omega_3\})
&= P(\{\omega_1\} \times \{\omega_2\} |\omega_3) \phi(\Omega_1 \times \Omega_2 \times \{\omega_3\}) \label{Eq:ideal_compromise_expression1}\\
&= Q(\{\omega_2\} \times \{\omega_3\} |\omega_1) \phi(\{\omega_1\} \times \Omega_2 \times \Omega_3) \\
&= R(\{\omega_1\} \times \{\omega_3\} |\omega_2) \phi(\Omega_2 \times \{\omega_2\} \times \Omega_3). \label{Eq:ideal_compromise_expression2}
\end{align}

Unfortunately, the existence of such an ``optimal compromise'' $\phi$ implies that $\pi_1$, $\pi_2$ and $\pi_3$ are compatible. First, one can show that for every $(\omega_1,\omega_2,\omega_3) \in \Omega_1 \times \Omega_2 \times \Omega_3$, $\phi(\{\omega_1\} \times \{\omega_2\} \times \{\omega_3\})>0$. The arguments are similar to those used above to show that $P$ is a positive mapping. Rewriting Equations \eqref{Eq:ideal_compromise_expression1} and \eqref{Eq:ideal_compromise_expression2} yields

\begin{align}
\frac{\phi(\{\omega_1\} \times \{\omega_2\} \times \{\omega_3\})}{\phi(\Omega_1 \times \{\omega_2\} \times \{\omega_3\})}
&= \frac{P(\{\omega_1\} \times \{\omega_2\} |\omega_3)}{P(\Omega_1 \times \{\omega_2\} |\omega_3)}
= \frac{1}{2} \pi_1(\{\omega_1\} | \omega_2,\omega_3) + \frac{1}{2} \pi_2(\{\omega_2\} | \omega_1,\omega_3) \frac{P(\{\omega_1\} \times \Omega_2 |\omega_3)}{P(\Omega_1 \times \{\omega_2\} |\omega_3)} \label{Eq:ideal_compromise_expression12}\\
&= \frac{R(\{\omega_1\} \times \{\omega_3\} |\omega_2)}{R(\Omega_1 \times \{\omega_2\} |\omega_3)}
= \frac{1}{2} \pi_1(\{\omega_1\} | \omega_2,\omega_3) + \frac{1}{2} \pi_3(\{\omega_3\} | \omega_1,\omega_2) \frac{R(\{\omega_1\} \times \Omega_3 |\omega_2)}{R(\Omega_1 \times \{\omega_3\} |\omega_2)}. \label{Eq:ideal_compromise_expression22}
\end{align}

Now combine Equations \eqref{Eq:ideal_compromise_expression12} and \eqref{Eq:ideal_compromise_expression22}:

\begin{equation}
\frac{1}{2} \pi_2(\{\omega_2\} | \omega_1,\omega_3) \frac{\phi(\{\omega_1\} \times \Omega_2 \times \{\omega_3\})}{\phi(\Omega_1 \times \{\omega_2\} \times \{\omega_3\})}
=
 \frac{1}{2} \pi_3(\{\omega_3\} | \omega_1,\omega_2) \frac{\phi(\{\omega_1\} \times \{\omega_2\} \times \Omega_3)}{\phi(\Omega_1 \times \{\omega_2\} \times \{\omega_3\})}.
\end{equation}

As this holds for every $(\omega_1,\omega_2,\omega_3) \in \Omega_1 \times \Omega_2 \times \Omega_3$, it implies that $\pi_2 \phi_{-2} = \pi_3 \phi_{-3}$. A similar proof then shows that $\pi_3 \phi_{-3} = \pi_1 \phi_{-1}$. This means that if an ``optimal compromise'' $\phi$ exists, then $\pi_1$, $\pi_2$ and $\pi_3$ are compatible and no compromise was needed. \medskip

Similarly to what was done in the two-dimensional case, we avoid this difficulty by weakening the compatibility requirements: we no longer require $P(\cdot|\omega_3)$ to be the optimal compromise between $\pi_1$ and $\pi_2$ \emph{for every} $\omega_3 \in \Omega_3$ (as expressed by  Equation \eqref{Eq:compromis_ideal}), but only 
\emph{on average} over $\omega_3 \in \Omega_3$. So a compromise $\phi$ should still verify Equation \eqref{Eq:ideal_compromise_expression1}, but it would only need to verify this weakened version of Equation \eqref{Eq:compromis_ideal} for every $(\omega_1,\omega_2) \in \Omega_1 \times \Omega_2$:

\begin{align}
&\sum_{\omega_3 \in \Omega_3} P(\{\omega_1\} \times \{\omega_2\} |\omega_3) \phi(\Omega_1 \times \Omega_2 \times \{\omega_3\}) \nonumber \\
=& \sum_{\omega_3 \in \Omega_3}
 \left[  \frac{1}{2} \pi_1(\{\omega_1\}|\omega_2,\omega_3) P(\Omega_1 \times \{\omega_2\}|\omega_3)+  \frac{1}{2} \pi_2(\{\omega_2\}|\omega_1,\omega_3) P(\{\omega_1\} \times \Omega_2|\omega_3) \right] \phi(\Omega_1 \times \Omega_2 \times \{\omega_3\}). \label{Eq:pragmatic_compromise}
\end{align}

Because Equation \eqref{Eq:ideal_compromise_expression1} is still expected to hold, Equation \eqref{Eq:pragmatic_compromise} is equivalent to

\begin{equation}
\phi_{-3}(\{\omega_1\} \times \{\omega_2\})
=
\frac{1}{2}  \sum_{\omega_3 \in \Omega_3} \pi_1(\{\omega_1\}|\omega_2,\omega_3) \phi_{-1} (\{\omega_2\} \times \{\omega_3\})
+
\frac{1}{2}  \sum_{\omega_3 \in \Omega_3} \pi_2(\{\omega_2\}|\omega_1,\omega_3) \phi_{-2} (\{\omega_1\} \times \{\omega_3\}).
\end{equation}

So the requirement boils down to $2 \phi_{-3} = (\pi_1 \phi_{-1})_{-3} + (\pi_2 \phi_{-2})_{-3}$. Of course, we symmetrically require $2 \phi_{-1} = (\pi_2 \phi_{-2})_{-1} + (\pi_3 \phi_{-3})_{-3}$ and $2 \phi_{-2} = (\pi_1 \phi_{-1})_{-2} + (\pi_3 \phi_{-3})_{-2}$ as well. \medskip

To sum this part of the discussion up, the necessity of weakening our ``ideal'' requirements for ``optimal compatibility'' made us downgrade from a requirement about a ``joint'' $3$-dimensional distribution $\phi$ to requirements about its $(3-1)$-marginal distributions $\phi_{-1}$, $\phi_{-2}$ and $\phi_{-3}$.  Definition \ref{Def:compatibilite} is just another formulation of these requirements, which are taken to define a compatible sequence of $(r-1)$-dimensional distributions. \medskip

Proposition \ref{Prop:hypermarginales_fixes} shows that in cases where there exists a unique Gibbs compromise, even with this weakened set of requirements for compatibility, there exists only one compatible sequence of $(r-1)$-dimensional distributions, so we may not strenghten it if there is to be any solution. Indeed, in cases where no Gibbs compromise exists, no compatible set of $(r-1)$-dimensional distributions 
exists either!

\subsection{Weaker definitions of compromises are inconvenient} \label{Sec:compromise_examples}

As was shown in the previous subsection, the requirements given by Definition \ref{Def:compatibilite} for the compatibility of a sequence of $(r-1)$-dimensional distributions with a given sequence of Markov kernels cannot be strengthened. The following shows they cannot be weakened either.

\subsubsection{Why we need some notion of compatibility: example in 2 dimensions.}

Why bother with the compatibility of $(r-1)$-dimensional distributions and not simply minimize the functional $E_\lambda$ of Definition \ref{Def:compromis_optimal} over all distributions absolutely continuous with respect to $\lambda$ ? The following 2-dimensional example shows that doing so yields unsatisfactory results. \medskip

Consider the following situation: $\Omega_1 = \Omega_2 = \{0,1\}$ and $\mathcal{A}_1=\mathcal{A}_2 = \{\emptyset,\{0\},\{1\},\{0,1\}\}$. Let $X_1$ (resp. $X_2$) be the identity function on $\Omega_1$ (resp. $\Omega_2$). Both $X_1$ and $X_2$ are measurable functions (and thus random variables when $\mathcal{A}_1 \otimes \mathcal{A}_2$ is endowed with a probability measure). Define the following Markov kernels:

\begin{align}
\pi_1(X_1 = 1|\omega_2) = \bs{1}_{\{\omega_2=0\}} + 1/2 \, \bs{1}_{\{\omega_2=1\}}; \\
\pi_2(X_2 = 1|\omega_1) = 1/2 \, \bs{1}_{\{\omega_1=0\}} + \bs{1}_{\{\omega_1=1\}}.
\end{align}

Let $\lambda$ be the counting measure.
Denote by $\pi_C$ the optimal compromise with respect to $\lambda$ (minimizing $E_\lambda$ over all compromises) and $\pi_E$ the distribution on $\mathcal{A}_1 \otimes \mathcal{A}_2$ that minimizes $E_\lambda$ over all distributions on $\mathcal{A}_1 \otimes \mathcal{A}_2$. We have $E_\lambda(\pi_C) = 2/25 > E_\lambda(\pi_E) = 1/15$ and 

\begin{align}
\pi_C(\{\omega_1\} \times \{\omega_2\}) &= 1/10 \left( \bs{1}_{\{(\omega_1,\omega_2)=(0,0)\}} + \bs{1}_{\{(\omega_1,\omega_2)=(1,0)\}} + 3 \, \bs{1}_{\{(\omega_1,\omega_2)=(0,1)\}} + 5 \, \bs{1}_{\{(\omega_1,\omega_2)=(1,1)\}} \right); \\
\pi_E(\{\omega_1\} \times \{\omega_2\}) &= 1/30 \left( 3 \, \bs{1}_{\{(\omega_1,\omega_2)=(0,0)\}} + \bs{1}_{\{(\omega_1,\omega_2)=(1,0)\}} + 11 \, \bs{1}_{\{(\omega_1,\omega_2)=(0,1)\}} + 15 \, \bs{1}_{\{(\omega_1,\omega_2)=(1,1)\}} \right).
\end{align}

As $\pi_C$ is the unique Gibbs compromise between $\pi_1$ and $\pi_2$, its marginals are the only compatible marginals: $(\pi_C)_{-1}(X_2=1) = 4/5$ and $(\pi_C)_{-2}(X_1=1) = 3/5$. The marginals of $\pi_E$ are noticeably different: $(\pi_E)_{-1}(X_2=1) = 13/15$ and $(\pi_E)_{-2}(X_1=1) = 8/15$. \medskip

Observe that according to $\pi_1$, $X_2=0$ implies $X_1=1$ but that according to $\pi_2$, $X_1=1$ implies that $X_2=1 \neq 0$. This discrepancy is a major source of incompatibility between the two Markov kernels. 
So, as $\pi_E$ makes both $X_1=1$ and $X_2=0$ less likely than $\pi_C$, it ``ignores the inconsistent parts''  of $\pi_1$ and $\pi_2$ to some extent. Therefore, if the marginals are not set in advance (say, by imposing compatibility with the conditionals in the sense of Definition \ref{Def:compatibilite}), one may ``cheat'' by having the marginals disadvantage inconvenient values for the parameters.

\subsubsection{Why the compatibility requirements can hardly be weakened: example in 3 dimensions.}

In the two-dimensional case, because of Proposition \ref{Prop:marginales_identiques}, Definitions \ref{Def:compatibilite_2d} and \ref{Def:compatibilite} give the same meaning to the concept of compatibility of marginals, so Definition \ref{Def:compatibilite} may be thought of as a generalization of Definition \ref{Def:compatibilite_2d} to cases with more than two dimensions. However, another generalization of the latter definition is possible. To avoid confusion, this other generalization will be called \emph{weak compatibility}. In the following, the sequence of Markov kernels $(\pi_i)_{i \in [\![1,r]\!]} $ is defined as in Section \ref{Sec:compromis_optimal_def}.

\begin{defn} \label{Def:compatibilite_faible}
A sequence of $(r-1)$-dimensional distributions $(m_{\neq i})_{i \in [\![1,r]\!]} $ is \emph{weakly compatible} with a sequence of Markov kernels $(\pi_i)_{i \in [\![1,r]\!]} $ if Equation \eqref{Eq:weak_compatibility} holds.
\end{defn}

Proposition \ref{Prop:marginales_identiques} means that compatibility in the sense of Definition \ref{Def:compatibilite} implies weak compatibility in the sense of Definition \ref{Def:compatibilite_faible}, hence its denomination as ``weak''. \medskip

Using the concept of weak compatibility of a sequence of $(r-1)$-marginal distributions, we define weak compromises and the optimal weak compromise as analogues to compromises and optimal compromises respectively.

\begin{defn}
A probability distribution $P$ on $\bigotimes_{i \in [\![1,r]\!]} \mathcal{A}_i$ is called a \emph{weak compromise} between the Markov kernels $(\pi_i)_{i \in [\![1,r]\!]} $ if these two conditions are verified:
\begin{enumerate}
\item for every $i \in [\![1,r]\!]$, $\pi_i P_{-i}$ is absolutely continuous with respect to $P$;
\item the sequence $(P_{-i})_{i \in [\![1,r]\!]}$ of $P$'s $(r-1)$-marginal distributions  
is weakly compatible with $(\pi_i)_{i \in [\![1,r]\!]} $.
\end{enumerate}
\end{defn}

\begin{defn}
Let $\lambda$ be a positive measure on $\mathcal{A}$. Let $P$ be a weak compromise between the sequence of Markov kernels $(\pi_i)_{i \in [\![1,r]\!]} $ that is absolutely continuous with respect to $\lambda$. $P$ is called an \emph{optimal weak compromise} with respect to $\lambda$ between the sequence of Markov kernels $(\pi_i)_{i \in [\![1,r]\!]} $ if it minimizes the functional $E_\lambda$ over all compromises between $(\pi_i)_{i \in [\![1,r]\!]} $ that are absolutely continuous with respect to $\lambda$. $E_\lambda$ is defined by Equation \eqref{Eq:energie_compromis}.
\end{defn}

Because, for any positive measure $\lambda$ on $\mathcal{A}$, the set of all weak compromises absolutely continuous with respect to $\lambda$ includes the set of all compromises absolutely continuous with respect to $\lambda$, an optimal weak compromise with respect to $\lambda$ makes the 
functional $E_\lambda$ 
no greater than an optimal compromise with respect to $\lambda$. 
However, as shown in the following example with $r=3$, optimal weak compromises may have undesirable behavior. \medskip

Assume $\Omega_1 = \Omega_2 = \Omega_3 = \{0,1\}$ and $\mathcal{A}_1 = \mathcal{A}_2 = \mathcal{A}_3 =  \{\emptyset,\{0\},\{1\},\{0,1\}\}$. Let $X_1$ (resp. $X_2$, $X_3$) be the identity function on $\Omega_1$ (resp. $\Omega_2$, $\Omega_3$).
 
Consider the following Markov kernels. For every $(\omega_1,\omega_2,\omega_3) \in \Omega_1 \times \Omega_2 \times \Omega_3$

\begin{align}
\pi_1(X_1 = 1| \omega_2, \omega_3) &= 1/2; \\
\pi_2(X_2 = 1| \omega_1, \omega_3) &= 1/2; \\
\pi_3(X_3 = 1| \omega_1, \omega_2) &=  \bs{1}_{\{\omega_1=\omega_2\}} + 1/2 \,\bs{1}_{\{\omega_1\neq \omega_2\}}.
\end{align}

Notice that, provided $\omega_3$ is known, $\pi_1$ and $\pi_2$ are compatible Markov kernels. The unique probability distribution on $\mathcal{A}_1 \otimes \mathcal{A}_2$ that fits both $\pi_1$ and $\pi_2$ (conditional to $\omega_3$) verifies for every $(\omega_1,\omega_2) \in \Omega_1 \times \Omega_2$ 

\begin{equation}
P(\{\omega_1\} \times \{\omega_2\} | \omega_3) = 1/4.
\end{equation}

Thus, any joint distribution fitting the Markov kernel $P$ would make $X_1$, $X_2$ and $X_3$ mutually independent. Unfortunately, no such joint distribution could fit $\pi_3$, but we may expect compromises between $\pi_1$, $\pi_2
$ and $\pi_3$ to retain the independence of $X_1$ and $X_2$. \medskip

Let $\lambda$ be the counting measure on $\mathcal{A}$. \medskip

Denote by $\pi_C$ the (unique) optimal compromise between $\pi_1$, $\pi_2$ and $\pi_3$ with respect to $\lambda$: we have $E_\lambda(\pi_C) = 1/48 \approx 0.021$. For every $(\omega_1,\omega_2,\omega_3) \in \Omega_1 \times \Omega_2 \times \Omega_3$

\begin{equation}
\pi_C(\{\omega_1\} \times \{\omega_2\} \times \{\omega_3\}) = \frac{1}{24} \bs{1}_{\{\omega_1=\omega_2,\omega_3=0\}} +
\frac{1}{12} \bs{1}_{\{\omega_1 \neq \omega_2,\omega_3=0\}} +
\frac{5}{24} \bs{1}_{\{\omega_1=\omega_2,\omega_3=1\}} +
\frac{1}{6} \bs{1}_{\{\omega_1 \neq \omega_2,\omega_3=1\}}.
\end{equation}

Notably, its third 2-marginal distribution $(\pi_C)_{-3}$ verifies for every $(\omega_1,\omega_2) \in \Omega_1 \times \Omega_2$

\begin{equation}
(\pi_C)_{-3}(\{\omega_1\} \times \{\omega_2\}) = 1/4.
\end{equation}

So, as expected, $\pi_C$ retains the independence of $X_1$ and $X_2$. Because $\pi_C$ is the unique Gibbs compromise between $\pi_1$, $\pi_2$ and $\pi_3$, Proposition \ref{Prop:hypermarginales_fixes} implies any other compromise between $\pi_1$, $\pi_2$ and $\pi_3$ also retains this property.

Let us now consider an optimal weak compromise $\pi_W$ between $\pi_1$, $\pi_2$ and $\pi_3$ with respect to $\lambda$. 
Numerical computation gives us the following approximation, with $E_\lambda(\pi_W) \approx 0.019$.
For every $(\omega_1,\omega_2,\omega_3) \in \Omega_1 \times \Omega_2 \times \Omega_3$,

\begin{equation}
\pi_W(\{\omega_1\} \times \{\omega_2\} \times \{\omega_3\}) \approx 0.04 \, \bs{1}_{\{\omega_1=\omega_2,\omega_3=0\}} +
0.10 \, \bs{1}_{\{\omega_1 \neq \omega_2,\omega_3=0\}} +
0.19 \, \bs{1}_{\{\omega_1=\omega_2,\omega_3=1\}} +
0.17 \, \bs{1}_{\{\omega_1 \neq \omega_2,\omega_3=1\}}.
\end{equation}

Its third 2-marginal distribution $(\pi_W)_{-3}$ is approximately

\begin{equation}
(\pi_W)_{-3}(\{\omega_1\} \times \{\omega_2\}) \approx 0.23 \, \bs{1}_{\{\omega_1 = \omega_2\}} + 0.27 \, \bs{1}_{\{\omega_1 \neq \omega_2\}}.
\end{equation}

Thus the independence of $X_1$ and $X_2$ is lost. Therefore, weak compatibility is no adequate notion of compatibility. As a matter of fact, although $(\pi_W)_{-3}$ and $(\pi_C)_{-3}$ share the same marginals, that is

\begin{align}
(\pi_W)_{-3}(\{\omega_1\} \times \Omega_2 ) = 1/2, \\
(\pi_W)_{-3}(\Omega_1 \times \{\omega_2\}) = 1/2,
\end{align}

$(\pi_W)_{-3}$ slightly disadvantages the event $X_1=X_2$, which is where the incompatibility between $\pi_3$ and the pair $(\pi_1,\pi_2)$ is most obvious: according to $\pi_3$, $X_1=X_2$ implies $X_3=1$, so conversely, $X_3=0$ should imply $X_1 \neq X_2$, when in fact $\pi_1$ and $\pi_2$ state that even given $\omega_3=0$, $\{X_1 \neq X_2 \}$ only happens with probability 1/2. On the other hand, according to $\pi_3$, if $X_1 \neq X_2$, then $X_3$ can with equal probability be 0 or 1, which matches $\pi_1$ and $\pi_2$ better.

\section{Optimal compromise between Objective Posterior conditional distributions in Gaussian Process regression} \label{Sec:convergence_Gibbs_algorithm}

Kriging is a surrogate model used to emulate a real-valued function on a spatial domain $\mathcal{D}$ when said function can only be evaluated on a finite subset of $\mathcal{D}$ called ``design set''. 
The ``Kriging prediction'' is the mean function of the process taken conditionally to all known values of the emulated function, i.e. the values at the points in the design set. The main advantage of the framework is its natural way of 
representing uncertainty about the value of the function at unobserved points \cite{SWN}. Prediction does not consist of a single value but of a complete Normal distribution. ``Kriging'' is the name given to the framework in the geostatistical literature \cite{JH78}, but is also frequently used in the context of computer experiments and machine learning under the label ``Gaussian Process regression'' \cite{Ras06}.  In this work, we focus on Simple Kriging, where the  Gaussian Process is assumed to be stationary with known mean, as opposed to Universal Kriging, which incorporates an unknown mean function. \medskip

The probability distribution of a stationary Gaussian Process 
is characterized by a variance parameter and a correlation function (also known as ``correlation kernel'') which itself depends on parameters. So 
one should deal with uncertainty about model parameters.

The problem is ``notoriously difficult'', as highlighted by \citet{KOH01}, because the likelihood function may often be quite flat \cite{LS05}. In a Bayesian framework, this uncertainty is represented by a prior distribution on the parameters. \medskip

\subsection{Issues raised by objective prior elicitation for Gaussian processes}

Let $Y(\bs{x})$, $\bs{x} \in \mathcal{D} $ be a real-valued random field on a bounded subset $\mathcal{D}$ of $\R^r$. 
We assume $Y$ is Gaussian with zero mean (or known mean) and with covariance
of the form ${\rm Cov}( Y(\bs{x}),Y(\bs{x}')) = \sigma^2 K_{\bs{\theta}}(\bs{x}-\bs{x}')$.
$\sigma^2$ thus denotes the variance of the Gaussian Process
 and $\bs{\theta} \in (0,+\infty)^r$, hereafter named the ``vector of correlation lengths'', is the vector of scaling parameters used by the chosen class of correlation kernels $K_{\bs{\theta}}$. \medskip

Consider a set of $n \in \N$ points $(\bs{x}^{(i)})_{i \in [\![1,n]\!]}$ belonging to the domain $\mathcal{D}$. This set is called the design set and $Y$ is observed at all points of this set. Let $\bs{Y}$ be the Gaussian vector $(Y(\bs{x}^{(i)}))_{i \in [\![1,n]\!]}$ and let $\corr$ be its correlation matrix: the distribution of $\bs{Y}$ is therefore $\mathcal{N}(\bs{0}_n, \sigma^2 \corr)$.\medskip

Let $\bs{y}$ be the vector of observations. When applied to a matrix, $ | \cdot | $ refers to its determinant. \medskip

With these notations, 
the likelihood of the parameters $\sigma^2$ and  $\bs{\theta}$ is

\begin{equation} \label{Eq:vraisemblance}
L( \bs{y} \; | \;  \sigma^2 , \bs{ \theta } ) = 
\left( \frac{ 1 }{ 2 \pi \sigma^2 } \right) ^ { \frac{n}{2} } | \bs{ \Sigma }_{ \bs{ \theta } } | ^ {- \frac{1}{2} } \exp \left\{ - \frac{ 1 }{ 2 \sigma^2 } \bs{y} \trans \bs{ \Sigma }_{ \bs{ \theta } }^{-1} \bs{y} \right\} \; .
\end{equation}

The reference prior with parameter ordering $\sigma^2 \prec \bs{\theta}$ is given by \citet{BDOS01}:

\begin{align}
\pi(\sigma^2,\bs{\theta}) =  \pi(\sigma^2 | \bs{\theta}) \pi(\bs{\theta}) \quad \mathrm{with} \quad
\pi(\sigma^2 | \bs{\theta}) \propto 1 / \sigma^2. 
\end{align}

The distribution $\pi(\sigma^2 | \bs{\theta})$ has infinite mass: it is an improper prior. \medskip

Let us integrate $\sigma^2$ out of the likelihood (\ref{Eq:vraisemblance}):

\begin{equation} \label{Eq:vraisemblance_integree}
L^{1} ( \bs{y} \; | \; \bs{ \theta } ) \propto
\int_0^{ \infty } L( \bs{y} \; | \; \sigma^2 , \bs{ \theta } ) \pi ( \sigma^2 | \; \bs{ \theta } ) \; d( \sigma^2 ) =
\left( \frac{ 2 \pi^{\frac{n}{2}} }{ \Gamma \left(\frac{n}{2}\right) } \right)^{-1}   | \bs{ \Sigma }_{ \bs{ \theta } } | ^ {- \frac{1}{2} } 
 \left( \bs{y} \trans \corr^{-1} \bs{y}  \right)^{ - \frac{ n }{2}  }.  
\end{equation}

\citet{RSH12} provide the reference prior $\pi(\bs{\theta})$, where $\bs{\theta}$ is regarded as a single multidimensional parameter. It is proportional to the square root of the determinant of the $r \times r$ matrix with $(i,j)$-th entry

\begin{equation} \label{Eq:Fisher_info}
\Tr \left[ \left( \frac{ \partial }{\partial \theta_i } \left( \corr \right) \corr^{-1} \right) 
\left( \frac{ \partial }{\partial \theta_j } \left( \corr \right) \corr^{-1} \right)  \right] 
- \frac{1}{ n } \Tr \left[ \frac{ \partial }{\partial \theta_i } \left( \corr \right) \corr^{-1} \right] 
\Tr \left[ \frac{ \partial }{\partial \theta_j } \left( \corr \right) \corr^{-1} \right].
\end{equation}

However, this method has the disadvantage of requiring the use of a multidimensional Jeffreys-rule prior distribution, which may show the sort of undesirable behavior mentioned in the introduction.

Alternatively, we could draw inspiration from the one-dimensional case in the following way. Suppose that we know every entry of $ \bs{ \theta } $ except one, $ \theta_i $. Then, according to Equation \eqref{Eq:Fisher_info}, the prior density on $ \theta_i $ knowing all entries $ \theta_j $ ($ j \neq i $) would be

\begin{equation} \label{PriorCond}
\pi_i( \theta_i \; | \; \theta_j \; \forall j \neq i ) \propto \sqrt{ \Tr \left[ \left( \frac{ \partial }{\partial \theta_i } (\corr) \corr^{-1} \right)^2 \right] - \frac{1}{ n } \Tr \left[ \frac{ \partial }{\partial \theta_i } (\corr) \corr^{-1} \right]^2 }.
\end{equation}

The density functions $\pi_i( \theta_i \; | \; \theta_j \; \forall j \neq i )$ ($i \in [\![1,r]\!]$) define Markov kernels. Indeed, they are continuous with respect to the $\theta_j$ ($j \neq i$) and are probability densities with respect to the Lebesgue measure. They are unfortunately likely to violate the necessary condition for compatibility given by Equation \eqref{Eq:compatible_conditionals}. \medskip

Let us now consider the corresponding posterior conditional densities $\pi_i( \theta_i \; |  \; \bs{y}, \; \theta_j \; \forall j \neq i )$ ($i \in [\![1,r]\!]$). Just like their prior counterparts, they are likely to violate the necessary condition for compatibility. However, each of them represents our opinion about one parameter if all others were known. This is a setting where the results of Section \ref{Sec:compromis_optimal} can be applied in order to find the optimal compromise between the Markov kernels $\R^{r-1} \times \mathcal{B}(\R)$ they define. This optimal compromise will then be taken as posterior probability of the vector $\bs{\theta}$. In the following, we describe settings in which there exists a single Gibbs compromise between these Markov kernels. Theorem \ref{Thm:Gibbs_compromis_optimal} then asserts it is the optimal compromise. We call this compromise the Gibbs reference posterior distribution because of its link to the reference posterior distribution in settings with a one-dimensional parameter $\bs{\theta}$. \medskip

However, even though we call it a ``posterior'' distribution, it is unclear whether there exists a prior distribution from which it could be derived using Bayes' rule. Denote by $\pi_G(\bs{\theta} | \bs{y})$ the probability density with respect to the Lebesgue measure of the Gibbs reference posterior distribution. Bayes' rule requires that in case a (proper or improper) prior density $\pi_G(\bs{\theta})$ exists, there should also exist a function $\tilde{L}(\bs{y})$ such that, for almost every $\bs{\theta} \in \R^r$ in the sense of the Lebesgue measure,

\begin{equation} \label{Eq:condition_exists_Gibbs_prior}
\frac{\pi_G(\bs{\theta} | \bs{y})}{L^1(\bs{y} | \bs{\theta})} = \frac{\pi_G(\bs{\theta})}{\tilde{L}(\bs{y})}.
\end{equation}

As we have no explicit expression of $\pi_G(\bs{\theta} | \bs{y})$, we have no way to check whether Equation \eqref{Eq:condition_exists_Gibbs_prior} holds or not. \medskip

In this section, we establish that, whenever Matérn anisotropic geometric or tensorized kernels with known smoothness parameter $\nu$ are used, under certain conditions to be detailed later, there exists a unique Gibbs compromise between the reference posterior conditionals, which thanks to Theorem \ref{Thm:Gibbs_compromis_optimal} is the optimal compromise. Henceforth, it will be called ``Gibbs reference posterior distribution'', even though this ``posterior'' has not been derived from a prior distribution using the Bayes rule. \medskip

All proofs for this section can be found in Appendix \ref{App:noyaux_Matérn}.

\subsection{Definitions}

In this work, we 
use the following convention for the Fourier transform: the Fourier transform $\widehat{g}$ of a smooth function $g : \R^r \rightarrow \R$ verifies $ g(\bs{x}) = \int_{\R^r} \widehat{g} (\bs{\omega}) e^{i \langle \bs{\omega} | \bs{x} \rangle} d \bs{\omega} $ and $ \widehat{g} (\bs{\omega}) = (2 \pi)^{-r} \int_{\R^r} g(\bs{x}) e^{-i \langle \bs{\omega} | \bs{x} \rangle} d \bs{x} $. \medskip

Let us set up a few notations.

\begin{enumerate}[(a)]
\item $\mathcal{K}_\nu$ is the modified Bessel function of second kind with parameter $\nu$ ;
\item $K_{r,\nu}$ is the $r$-dimensional Matérn isotropic covariance kernel with variance 1, correlation length 1 and smoothness $\nu \in (0,+\infty)$ and $\widehat{K}_{r,\nu}$ is its Fourier transform:

\begin{enumerate}[(i)]
\item $\forall \bs{x} \in \R^r$,
\begin{equation} \label{Eq:Matern_anis_geom}
K_{r,\nu}(\bs{x}) = \frac{1}{\Gamma(\nu) 2^{\nu - 1} } \left( 2 \sqrt{\nu} \|\bs{x}\| \right) ^{\nu} \mathcal{K}_\nu \left( 2 \sqrt{\nu} \|\bs{x}\| \right) \; ;
\end{equation}
\item 
$
\forall \bs{\omega} \in \R^r,
$
\begin{equation}
\widehat{K}_{r,\nu} (\bs{\omega}) 
= \frac{M_r(\nu)}{( \|\bs{\omega}\|^2 + 4 \nu  )^{  \nu + \frac{r}{2}  }}
\text{ with }
M_r(\nu) = \frac{\Gamma ( \nu + \frac{r}{2} ) (2 \sqrt{\nu} )^{2 \nu } }{ \pi^{\frac{r}{2} } \Gamma(\nu) }.
\end{equation}
\end{enumerate}

\item $K_{r,\nu}^{tens}$ is the $r$-dimensional Matérn tensorized covariance kernel with variance 1, correlation length 1 and smoothness $\nu \in \R_+$ and $\widehat{K}_{r,\nu}^{tens}$ is its Fourier transform:

\begin{enumerate}[(i)]
\item $\forall \bs{x} \in \R^r$,
\begin{equation} \label{Eq:Matern_tens}
K_{r,\nu}^{tens}(\bs{x}) = \prod_{j=1}^r K_{1,\nu}(\bs{x}_j) \; ;
\end{equation}
\item $\forall \bs{\omega} \in \R^r$,
\begin{equation}
\widehat{K}_{r,\nu}^{tens} (\bs{\omega}) =  \prod_{j=1}^r \widehat{K_{1,\nu}}(\bs{\omega}_j).
\end{equation}
\end{enumerate}
\item if $\bs{t} \in \R^r$, 
  $ \frac{\bs{t}}{\bs{\theta}} = \left( \frac{t_1}{\theta_1},...,\frac{t_r}{\theta_r} \right)$ and 
  $ \bs{t} \, \bs{\mu} = \left( t_1 \mu_1,..., t_r \mu_r \right)$.
\end{enumerate}

We define the Matérn geometric anisotropic covariance kernel with variance parameter $\sigma^2$, correlation lengths $\bs{\theta}$ (resp. inverse correlation lengths $\bs{\mu}$) and smoothness $\nu$ as the function $\bs{x} \mapsto \sigma^2 K_{r,\nu}\left(\frac{\bs{x}}{\bs{\theta}}\right)$ (resp. $\bs{x} \mapsto \sigma^2 K_{r,\nu}\left(\bs{x}\bs{\mu}\right)$).

Similarly, we define the Matérn tensorized covariance kernel with variance parameter $\sigma^2$, correlation lengths $\bs{\theta}$ (resp. inverse correlation lengths $\bs{\mu}$) and smoothness $\nu$ as the function $\bs{x} \mapsto \sigma^2 K_{r,\nu}^{tens}\left(\frac{\bs{x}}{\bs{\theta}}\right)$ (resp. $\bs{x} \mapsto \sigma^2 K_{r,\nu}^{tens} \left(\bs{x}\bs{\mu}\right)$). \medskip

Thanks to Proposition \ref{Prop:invariant_reparametrisation}, we may choose any parametrization we wish for the Matérn correlation kernels. We have found that the parametrization involving inverse correlation lengths makes proofs easier. \medskip

Several key passages in the proofs (to be found in Appendix \ref{App:noyaux_Matérn}) involve a technical assumption on the design set: 
\begin{defn}
A design set is said to have coordinate-distinct points, or simply to be coordinate-distinct, if 
for any distinct points in the set $\bs{x}$ and $\bs{x'}$, every component of the vector $\bs{x} - \bs{x'}$ differs from 0.
\end{defn}

Most randomly sampled design sets almost surely have coordinate-distinct points -- for instance Latin Hypercube Sampling. Cartesian product design sets, however, do not. 

\subsection{Main result}

The result is valid for Simple Kriging models with the following characteristics:
\begin{enumerate}[(a)]
\item the design set contains $n$ coordinate-distinct points in $\R^r$ ($n$ and $r$ are positive integers);
\item the covariance function is Matérn anisostropic geometric or tensorized with variance parameter $\sigma^2>0$, smoothness parameter $\nu$ and vector of correlation lengths (resp. inverse correlation lengths) $\bs{\theta}\in (0,+\infty)^r$ (resp. $\bs{\mu} \in (0,+\infty)^r$) ;
\item one of the following conditions is verified:
\begin{enumerate}[(i)]
\item $\nu \in (0,1)$ and $n>1$ and the Matérn kernel is tensorized ;
\item $\nu \in (1,2)$ and $n>r+2$ ;
\item $\nu \in (2,3)$ and $n>r(r+1)/2 + 2r + 3$.
\end{enumerate}
\end{enumerate}

\begin{thm} \label{Thm:existence_posterior_gibbs}

In a Simple Kriging model with the characteristics described above, there exists a hyperplane $\mathcal{H}$ of $\R^n$ such that, for any $\bs{y} \in \R^n \setminus \mathcal{H}$,
there exists a unique Gibbs compromise $\pi_G(\bs{\theta} | \bs{y})$ (resp. $\pi_G(\bs{\mu} | \bs{y})$) between the reference posterior conditionals $\pi_i(\theta_i | \bs{y}, \bs{\theta}_{-i})$ (resp. $\pi_i(\mu_i | \bs{y}, \manqueI)$). It is the unique stationary distribution of the Markov kernel $ P_{\bs{y}} : (0,+\infty)^r \times \mathcal{B}\left((0,+\infty)^r\right) \rightarrow [0,1]$ defined by
\begin{align}
P_{\bs{y}}(\bs{\theta}^{(0)}, d \bs{\theta}) &= \frac{1}{r} \sum_{i=1}^r \pi_i(\theta_i | \bs{y}, \bs{\theta}_{-i}^{(0)}) d \theta_i \;  \delta_{\bs{\theta}_{-i}^{(0)}} (d \bs{\theta}_{-i}) \nonumber \\
\text{(resp.} \; P_{\bs{y}}(\bs{\mu}^{(0)}, d \bs{\mu}) &= \frac{1}{r} \sum_{i=1}^r \pi_i(\mu_i | \bs{y}, \manqueI^{(0)}) d \mu_i \;  \delta_{\manqueI^{(0)}} (d \manqueI) \text{).} \nonumber
\end{align}
The Markov kernel $ P_{\bs{y}} $ is uniformly ergodic: 
$ \lim_{n \to \infty} \sup_{\bs{\theta}^{(0)} \in (0,+\infty)^r } \| P_{\bs{y}}^n (\bs{\theta}^{(0)},\cdot) - \pi_G(\cdot | \bs{y}) \|_{TV} = 0$
(resp. $ \lim_{n \to \infty} \sup_{\bs{\mu}^{(0)} \in (0,+\infty)^r } \| P_{\bs{y}}^n (\bs{\mu}^{(0)},\cdot) - \pi_G(\cdot | \bs{y}) \|_{TV} = 0$), where $\| \cdot \|_{TV}$ is the total variation norm.
\end{thm}

\begin{rmq}
The reference posterior conditionals are invariant by reparametrization, so the Markov kernel $ P_{\bs{y}} $ does not depend on whether the chosen parametrization uses correlation lengths $\bs{\theta}$ or inverse correlation lengths $\bs{\mu}$. Due to Proposition \ref{Prop:invariant_reparametrisation}, the Gibbs compromise does not either. The parametrization using inverse correlation lenghts $\bs{\mu}$ is more convenient for proving this theorem, however.
\end{rmq}

Notice that in such a Kriging model, the vector of observations $\bs{y}$ almost surely belongs to $\R^n \setminus \mathcal{H}$, so this assumption is of no practical consequence. Theorem \ref{Thm:existence_posterior_gibbs} therefore asserts that the Gibbs compromise between the incompatible conditionals $\pi_i(\mu_i | \bs{y}, \manqueI)$ exists, is unique, and can be sampled from using Potentially Incompatible Gibbs Sampling (PIGS). In the following, it is called ``Gibbs reference posterior distribution''. 

\subsection{Using the Gibbs reference posterior distribution} \label{Sec:usage_posterior}

Let $\bs{x}_0$ be a point in the domain $\mathcal{D}$ that does not belong to the design set. Denote by $\corrNouveauAncien$ the correlation matrix between $Y(\bs{x}_0)$ and $\bs{Y}$, and by $\corrAncienNouveau$ its transpose the correlation matrix between $\bs{Y}$ and $Y(\bs{x}_0)$. \medskip

Theorem 4.1.2. (case 4) of \citet{SWN} provides this useful result for prediction:

\begin{prop}
Conditionally to $\bs{Y} = \bs{y}$ and assuming $\bs{\theta}$ is known, 
the random variable $Z_0$ defined below follows the Student t-distribution with $n$ degrees of freedom.
$$ Z_0 := \sqrt{ \frac{n}{ \bs{y} \trans \corr^{-1} \bs{y} } } \frac{ Y(\bs{x}_0) - \corrNouveauAncien \corr^{-1} \bs{y} } { \sqrt{1 - \corrNouveauAncien \corr^{-1} \corrAncienNouveau } }. $$
\end{prop}

\begin{rmq}
If $n$ exceeds 30, it is usually accepted that the Student t-distribution with $n$ degrees of freedom can be approximated by the standard Normal distribution. As this threshold should be exceeded in practical cases, we would recommend performing all computations as though the Student t-distribution were Normal. 
\end{rmq}

This Proposition implicitely gives the distribution of $y_0 := Y(\bs{x}_0)$ conditionally to $\bs{Y} = \bs{y}$ and $\bs{\theta}$. For later reference, denote it by $L^1(y_0 | \bs{y}, \bs{\theta})$. 
In practice, 
when $\bs{\theta}$ is unknown, the distribution of $y_0 = Y(\bs{x}_0)$ conditionally to $\bs{Y}=\bs{y}$ can be obtained once $\bs{\theta}$ has been sampled from the Gibbs reference posterior distribution:

$$
P(y_0 | \bs{y}) := \int L^1(y_0 | \bs{y}, \bs{\theta}) \pi_G(\bs{\theta} | \bs{y}) d \bs{\theta}.
$$

 Its cdf can be approximated by averaging the cdfs of the Student t-distributions (or their Normal approximations) corresponding to every point in the sample.

\section{Comparisons between the MLE and MAP estimators} \label{Sec:comp_MLE_MAP}

To illustrate the inferential performance of the Gibbs reference posterior distribution, let us introduce the Maximum A Posteriori estimator (MAP). It takes the value $\hat{\bs{\theta}}_{MAP}$ of $\bs{\theta}$ where the density with respect to the Lebesgue measure of the Gibbs reference posterior distribution is largest. We contrast it with the Maximum Likelihood Estimator (MLE) $\hat{\bs{\theta}}_{MLE}$ which does the same with the likelihood function.

\subsection{Methodology}
In this section, we compare the MLE and MAP estimators for accuracy and robustness. \medskip

Our test cases are 3-dimensional Gaussian Processes with Matérn anisotropic geometric correlation kernels with smoothness 5/2. Their mean is the null function, which only leaves us with the matter of estimating their correlation length for each dimension. \medskip

We use uniform designs: our observation points are randomly generated according to the uniform distribution on a cube with side length 1. \medskip

In order to measure the performance of our estimators, we define a suitable distance between two vectors of correlation lengths. Then the error of an estimator is defined as its distance to the ``true'' vector of correlation lengths. \medskip

Let $g$ be the function such that for any $t$ in $(-1,1)$, $g(t) = \argtanh(t)$ and $g(-1)=g(1)=0$. We use the convention that, for any matrix $\bs{M}$ with elements in $[0,1]$, $g(\bs{M})$ is the matrix resulting from applying $g$ to every element of $\bs{M}$.

\begin{defn} \label{Def:distance_fisher}
For a given design set, the distance between two vectors of correlation lengths $\bs{\theta}^1$ and $\bs{\theta}^2$ is
 $\| g(\bs{\Sigma}_{\bs{\theta}^1}) - g(\bs{\Sigma}_{\bs{\theta}^2}) \|$, where $\| \cdot \|$ denotes the Frobenius norm.
\end{defn}

This distance involves applying the Fisher transformation \cite{Hot53} (that is, the inverse hyperbolic tangent function) to every (non-unitary) correlation coefficient in both associated correlation matrices. This is a variance-stabilizing transformation. For any random variables $U$ and $V$ following the normal distribution with mean $0$ and variance $1$, let $\rho$ denote the correlation coefficient between $U$ and $V$ ($-1<\rho<1$). If $(U_i,V_i)$ $(1 \leqslant i \leqslant N)$ are independent copies of $(U,V)$, then  $\hat{\rho} =  \sum_{i=1}^N U_i V_i / n$ is a random variable and $\argtanh(\hat{\rho})$ follows the normal distribution with mean $\argtanh(\rho)$ and variance $1/(N-3)$. So the variance of $\argtanh(\hat{\rho})$ does not depend on $\rho$, whereas the variance of $\hat{\rho}$ does and goes to zero for $|\rho| \rightarrow 1$. Involving the Fisher transformation in the definition of the distance between two vectors of correlation lengths is therefore a way to assert that vectors of correlation lengths can be far apart even if they both lead to highly correlated observations.

This allows us to make sure errors made when estimating near-1 correlation coefficients are no less taken into account than errors made when estimating near-0 correlation coefficients.
\medskip

Let us choose a ``true'' vector of correlation lengths (and also a variance parameter, but this parameter has no effect on either the MLE or the MAP). Then we need to:

\sf

\begin{enumerate}
\item Sample $n$ points randomly according to the uniform distribution on the unit cube (in the following, $n=30$).
\item Generate the observations of the Gaussian Process at the sampled points according to the selected ``true'' variance and correlation lengths.
\item Sample the vector of correlation lengths according to the Gibbs reference posterior distribution $\pi_G(\bs{\theta}|\bs{y})$ through PIGS.
\item Compute the MLE and the MAP of the vector of correlation lengths and their errors.
\item Repeat steps 1 to 4 $m-1$ times (in the following, $m=500$).
\end{enumerate}

\rm

This method allows us to derive an approximate distribution of the errors of both estimators when both the realization of the Gaussian Process and the design set vary. Thus we get to test the robustness of both estimators versus the variability of both the Gaussian Process and the choice of design set.

\subsection{Results}

This subsection provides results obtained on 3-dimensional Gaussian Processes with null mean function and Matérn anisotropic geometric correlation kernels with smoothness 5/2. The results are divided by ``true'' vectors of correlation lengths. In each case, we give in Table \ref{Tab:RMSE_distFisher} the empirical Root Mean Square Errors (RMSEs) of both MLE and MAP estimators 
as functions of  varying instances of the Gaussian Process and uniform design sets on the unit cube. \medskip

Most of the ``true'' vectors of correlation lengths featured in Table \ref{Tab:RMSE_distFisher} were selected in a way to showcase the behavior of both estimators in strongly anisotropic cases, but one (0.5 - 0.5 - 0.5) also showcases their behavior if the true kernel is actually isotropic. And the final one (0.8 - 1 - 0.9) is used to illustrate the performance in the case of a strongly correlated Gaussian Process: this case is fundamentally different from all others, because the Matérn anisotropic geometric family of correlation kernels is designed in such a way that the correlation length with greatest influence is the lowest. Informally speaking, it is enough for one correlation length to be near zero to make the whole process very uncorrelated, even should all other correlation lengths be very high. \medskip

In all studied cases, the MAP estimator was more robust than the MLE estimator: its RMSE (Root Mean Square Error) was between 9 and 15\% lower, as showcased in Table \ref{Tab:RMSE_distFisher}. \medskip

\begin{table}[!ht]
\begin{center}
\begin{tabular}{|c c c c|}
\hline 
\textbf{Corr. lengths} & \textbf{MLE} & \textbf{MAP} & \textbf{ – (\%) } \\ 
\hline 
0.4 – 0.8 – 0.2 & 3.49 & 2.97 & 15 \\

0.5 – 0.5 – 0.5 & 4.00 & 3.46 & 13 \\
 
0.7 – 1.3 – 0.4 & 4.02 & 3.64 & 9  \\

0.8 – 0.3 – 0.6 & 3.75 & 3.26 & 13 \\

0.8 – 1.0 – 0.9 & 4.65 & 4.18 & 10 \\
\hline
\end{tabular} 
\caption{RMSE (where the error is measured in terms of the distance in Definition \ref{Def:distance_fisher} ) of the MLE and MAP estimators for several ``true'' vectors of correlation lengths. The last column displays in percents the decrease of the RMSE of the MAP estimator with respect to the MLE.}
\label{Tab:RMSE_distFisher}
\end{center}
\end{table}

To get a better sense of the distribution of the error when the design set and the realization of the Gaussian Process vary, we give in Figure \ref{fig:erreurFisher} violin plots of the errors in the two most extreme case: very low correlation (0.4 – 0.8 – 0.2) and very high correlation (0.8 – 1.0 – 0.9)

\begin{figure}[!ht]
\begin{subfigure}{0.47\textwidth}
\includegraphics[angle=0,width=1\linewidth, height=7cm]{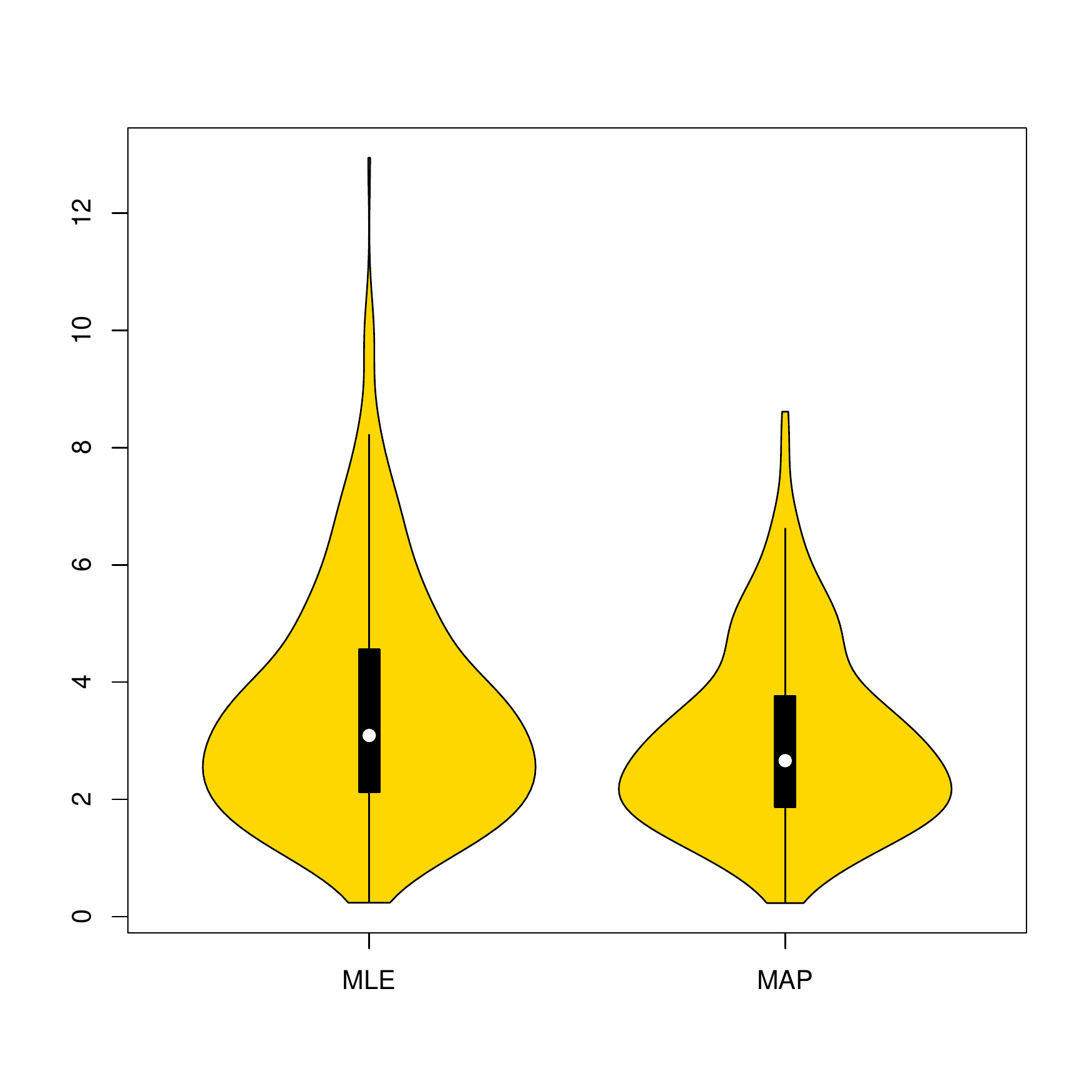}
\end{subfigure}
\begin{subfigure}{0.47\textwidth}
\includegraphics[angle=0,width=1\linewidth, height=7cm]{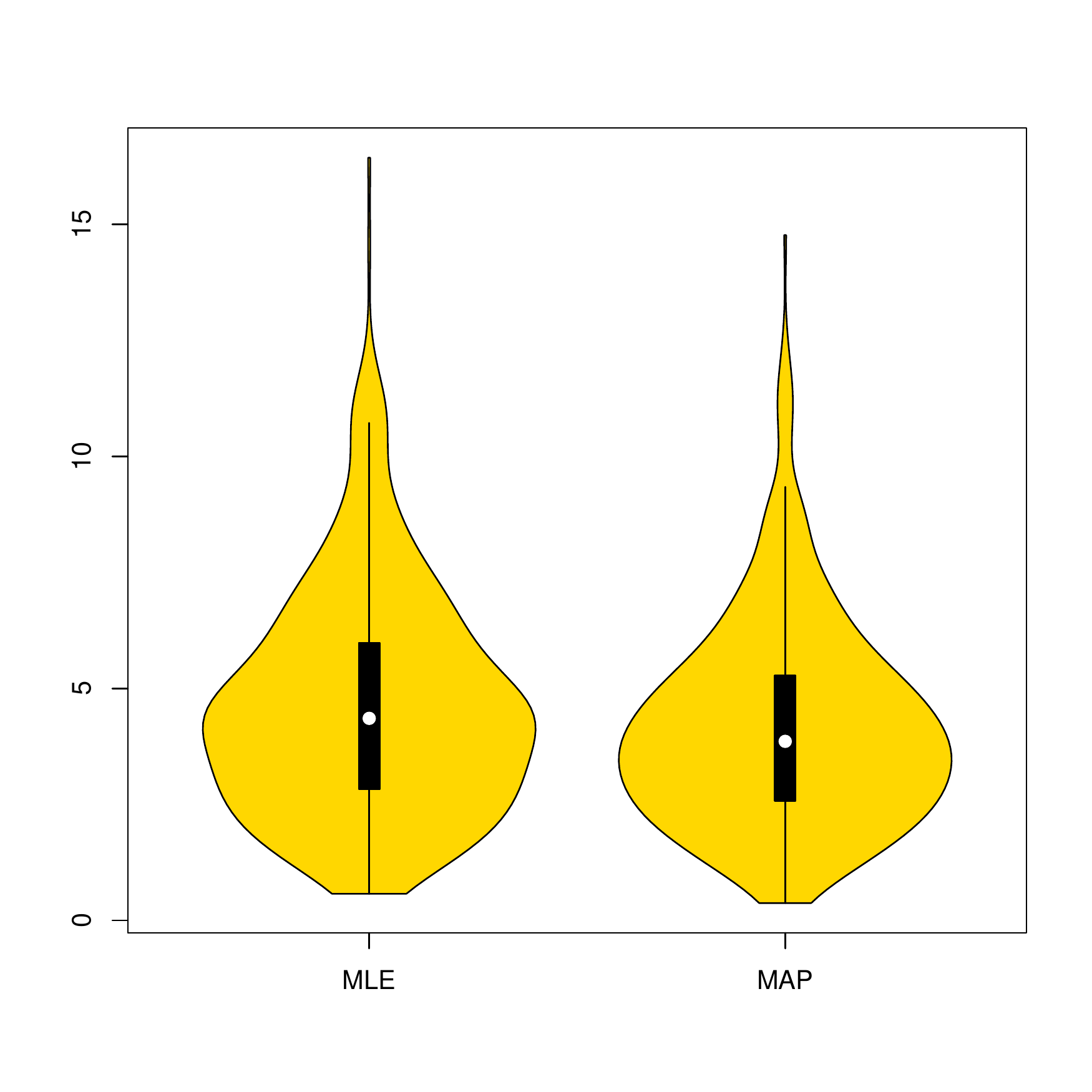}
\end{subfigure}
\caption{Violin plots of the error of the MLE and MAP estimators with respect to a design set following the uniform distribution and a Gaussian Process with correlation lengths 0.4 – 0.8 – 0.2 (left) and 0.8 – 1.0 – 0.9 (right).}
\label{fig:erreurFisher}
\end{figure}

\section{Comparison of the MLE and MAP plug-in distributions and the full posterior predictive distribution}
\label{Sec:comp_predictions}

\subsection{Methodology}

We use the same test cases as before. In this section, our goal is to assess the accuracy of prediction intervals associated with both estimators and with the full posterior distribution. We consider 95\% intervals: the lower bound is the 2.5\% quantile and the upper bound the 97.5\% quantile of plug-in distributions $\hat{P}_{MLE}(y_0  |  \bs{y}) = L^1(y_0  |  \bs{y}, \hat{\bs{\theta}}_{MLE})$, $\hat{P}_{MAP}(\bs{y}_0  |  \bs{y}) = L^1(y_0  |  \bs{y}, \hat{\bs{\theta}}_{MAP})$ and $P(\bs{y}_0  |  \bs{y}) = \int L^1(y_0  |  \bs{y}, \bs{\theta}) \; \pi_G(\bs{\theta}  |  \bs{y}) \; d \bs{\theta}$. For the sake of comprehensiveness, we also consider predictive intervals associated with the ``true'' predictive distribution $L(\bs{y}_0 \; | \; \bs{y}, \sigma^2, \bs{\theta})$, which is the predictive distribution we would use if we knew the correct values of the parameters $\sigma^2$ and $\bs{\theta}$. \medskip

Let us choose  a ``true'' vector of correlation lengths $\bs{\theta}$ (and also a variance parameter $\sigma^2$, but this parameter has no effect on predictive accuracy). Then we do the following:

\sf

\begin{enumerate}
\item Sample $n$ observation points randomly according to the uniform distribution 
on the unit cube (in the following, $n=30$).
\item Generate the observations of the Gaussian Process at the sampled points according to the selected ``true'' variance and correlation lengths.
\item Sample the vector of correlation lengths according to the Gibbs reference posterior distribution $\pi_G(\bs{\theta}|\bs{y})$ through PIGS.
\item Compute the MLE and the MAP of the vector of correlation lengths.
\item Sample $n_0$ test points randomly according to the uniform distribution on the unit cube (in the following, $n_0=100$).
\item At each point, determine the 95\% prediction intervals derived from $L(\bs{y}_0 \; | \; \bs{y}, \sigma^2, \bs{\theta})$ ($\sigma^2$ and $\bs{\theta}$ being the ``true'' parameters), $\hat{P}_{MLE}(\bs{y}_0 \; | \; \bs{y})$, $\hat{P}_{MAP}(\bs{y}_0 \; | \; \bs{y})$ and $P(\bs{y}_0 \; | \; \bs{y})$.
\item Generate the values of the Gaussian Process at the newly sampled points (naturally, do this conditionally to the previously generated observations).
\item Count the number of points within the prediction intervals derived from each of the four distributions.
Divide the counts by $n_0$:  this yields four \textit{coverages} corresponding to each type of predictive intervals.
Also compute the \textit{mean length} of every type of prediction interval.
\item Repeat steps 1 to 8 $m-1$ times (in the following, $m=500$).
\end{enumerate}

\rm

\subsection{Results}

There is no reason for individual coverages of 95\% predictive intervals given by the predictive distribution to be equal to 95\%. Recall that any coverage is given for a unique realization of the Gaussian Process, and that the values of this process at different points are correlated. If the predictive interval at some point fails to cover the true value at this point, it is likely that predictive intervals at neighboring points will also fail to cover the true values at those points, even though the nominal value is 95\% everywhere. Conversely, if it actually covers the true value, then prediction intervals at neighboring points are more than 95\% likely to cover their true values.

In short, prediction intervals give information that is only valid if understood to refer to what can be guessed on the sole basis of the observations made at the design points, which is why coverages for individual realizations of the Gaussian Process are not necessarily 95\% \textit{even if the predictive distribution is perfectly accurate} (\textit{i.e.} based on the true values of $\sigma^2$ and $\bs{\theta}$).

However, \textit{if the predictive distribution is perfectly accurate}, then the average of the coverages is the nominal value: 95\%. It is thus interesting to compute the average of the coverages for all distributions, whether they are plug-in distributions based on the MLE or MAP estimator, or the predictive distribution based on the full posterior distribution (hereafter noted FPD).
In the above described methodology, the average was taken over the realizations of the Gaussian Process with the chosen true parameters and over all design sets with $n$ design points. The results below are obtained in this way. \medskip

The results given in Table \ref{Tab:IP_taux_couverture} show that using the full posterior distribution (FPD) to derive the predictive distribution is the best possible choice from a frequentist point of view as the nominal value is nearly matched by the average coverage. Predictive intervals derived from the MAP estimator do not perform as well, and predictive intervals derived from the MLE perform even worse. \medskip

\begin{table}[!ht]
\begin{center}
\begin{tabular}{|c c c c c|}
\hline 
\textbf{Corr. lengths} & \textbf{True} & \textbf{MLE} & \textbf{MAP} & \textbf{FPD} \\ 
\hline 
0.4 – 0.8 – 0.2 & 0.95 & 0.88 & 0.91 & 0.95 \\

0.5 – 0.5 – 0.5 & 0.95 & 0.89 & 0.90 & 0.94 \\
 
0.7 – 1.3 – 0.4 & 0.95 & 0.90 & 0.92 & 0.95 \\

0.8 – 0.3 – 0.6 & 0.95 & 0.89 & 0.91 & 0.95 \\

0.8 – 1.0 – 0.9 & 0.95 & 0.90 & 0.92 & 0.94 \\
\hline
\end{tabular} 
\caption{Average with respect to randomly sampled design sets and realizations of the Gaussian Process (with variance parameter 1 and smoothness parameter 5/2) of the coverage of 95\% Prediction Intervals across the sample space. ``True'' stands for the prediction based on the knowledge of the true variance parameter and the true vector of correlation lengths.}
\label{Tab:IP_taux_couverture}
\end{center}
\end{table}

Let us now focus on the average (with respect to the 
uniform design sets and realizations of the Gaussian Process) of the mean (over the test set for a given realization of the Gaussian Process and a given uniform design set) length of prediction intervals. The results are given in Table \ref{Tab:IP_longueur_moyenne}, where the figures between parentheses give the increase or decrease (in percents) of the average mean length when compared to the average mean length of prediction intervals obtained using the true values of the parameters. \medskip

\begin{table}[!ht]
\begin{center}
\begin{tabular}{|c c c c c|}
\hline 
\textbf{Corr. lengths} & \textbf{True} & \textbf{MLE} & \textbf{MAP} & \textbf{FPD} \\ 
\hline 
0.4 – 0.8 – 0.2 & 2.23 & 2.05 (-8) & 2.13 (-4) & 2.59 (+16) \\

0.5 – 0.5 – 0.5 & 1.69 & 1.55 (-8) & 1.58 (-6) & 1.84 \; (+9) \\

0.7 – 1.3 – 0.4 & 1.09 & 1.02 (-6) & 1.07 (-2) & 1.21 (+11) \\

0.8 – 0.3 – 0.6 & 1.63 & 1.51 (-7) & 1.56 (-4) & 1.82 (+12) \\

0.8 – 1.0 – 0.9 & 0.71 & 0.66 (-7) & 0.69 (-3) & 0.76 \; (+8)\\
\hline
\end{tabular} 
\caption{Average with respect to randomly sampled design sets and realizations of the Gaussian Process (with variance parameter 1 and smoothness parameter 5/2) of the mean length of 95\% Prediction Intervals across the sample space. The numbers in parentheses represent in percents the increase when using the MLE/MAP/FPD instead of the ``true'' vector of correlation lengths and variance parameter.}
\label{Tab:IP_longueur_moyenne}
\end{center}
\end{table}

Predictive intervals derived from the full posterior distribution (FPD) are on average the largest, but not much larger than predictive intervals derived using the true parameters. In the tests we conducted, they seemed on average to be larger by about one fifth at worst. Predictive intervals derived from the MLE and MAP estimators are on average shorter than those derived from the true parameters. 
This can be interpreted as an under-estimation of the uncertainty of the prediction when fixing the vector
of correlation lengths to the most likely value given the observations, and this can explain the low observed coverage 
in Table~\ref{Tab:IP_taux_couverture}. \medskip

In Figure \ref{fig:IP_violons}, we give violin plots of coverage and mean length of Prediction Intervals in the two most extreme cases: 
correlation lengths 0.4 – 0.8 – 0.2 (very low correlation) and 0.8 – 1.0 – 0.9 (very high correlation).
The results are similar and illustrate the fact that the FPD gives larger intervals in order to reach the derived coverage value.

\begin{figure}[!ht]
\begin{subfigure}{0.45\textwidth}
\includegraphics[angle=0,width=1\linewidth, height=7cm]{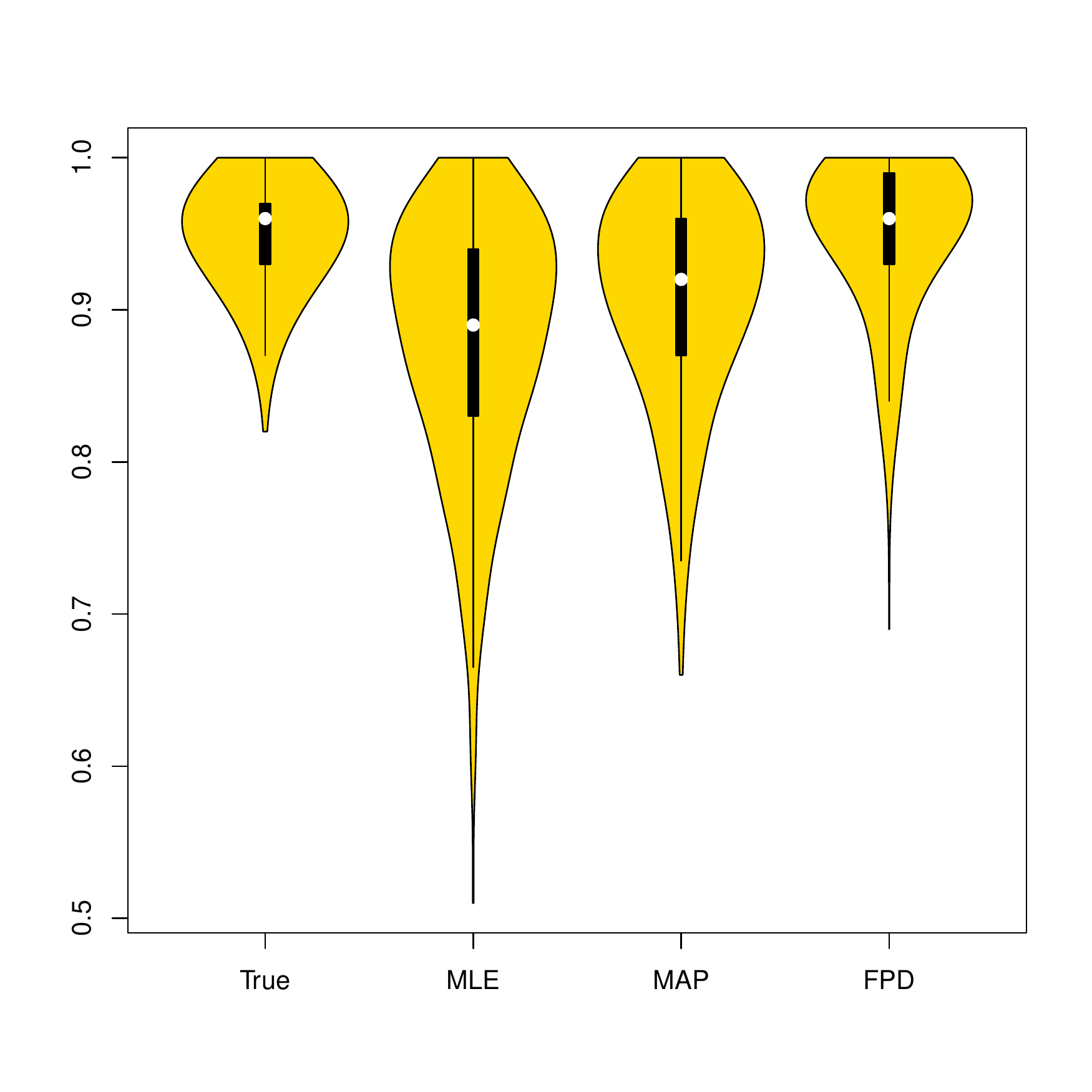}
\end{subfigure}
\begin{subfigure}{0.45\textwidth}
\includegraphics[angle=0,width=1\linewidth, height=7cm]{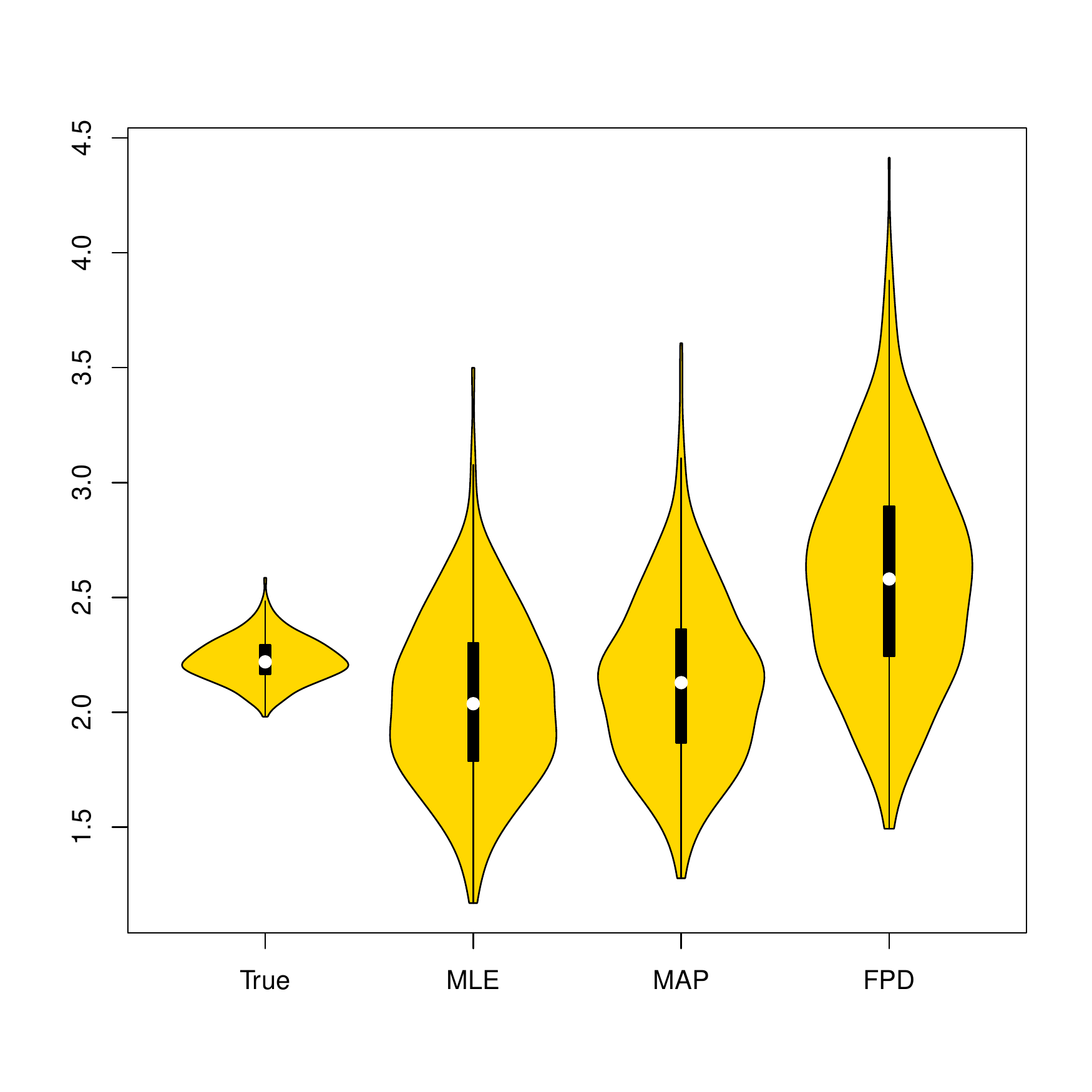}
\end{subfigure}

\begin{subfigure}{0.45\textwidth}
\includegraphics[angle=0,width=1\linewidth, height=7cm]{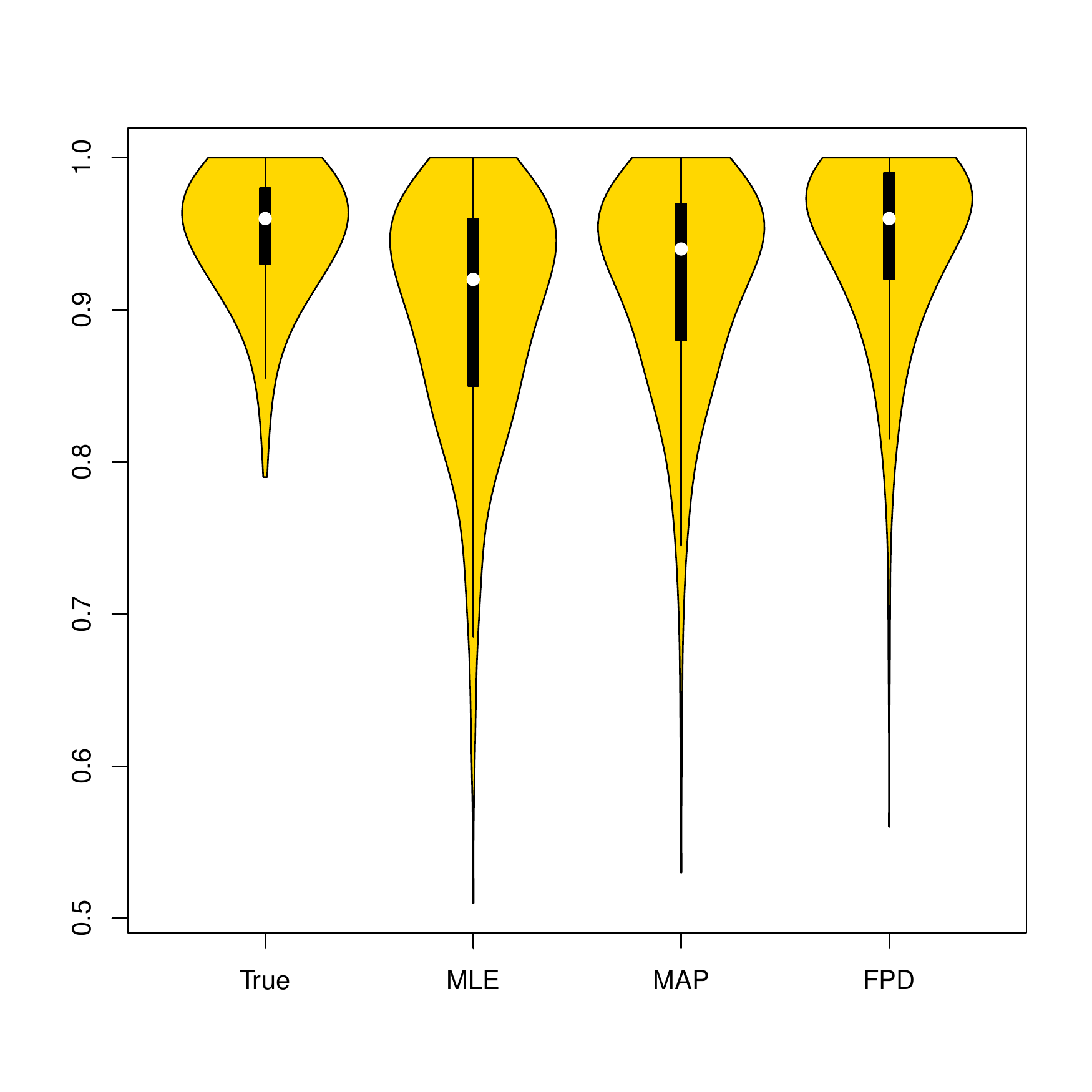}
\end{subfigure}
\begin{subfigure}{0.45\textwidth}
\includegraphics[angle=0,width=1\linewidth, height=7cm]{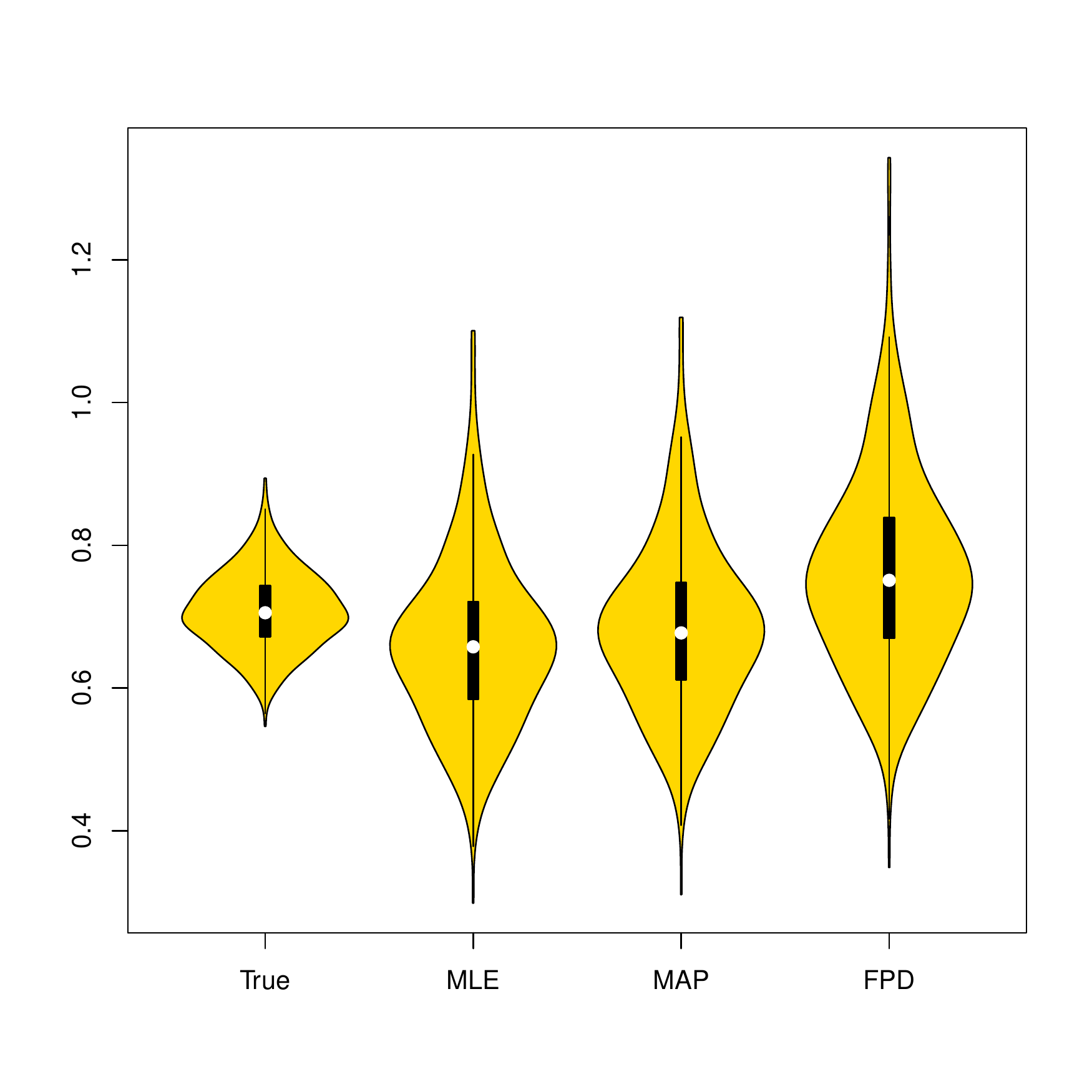}
\end{subfigure}
\caption{Violin plots of the coverage (left) and mean length (right) of Prediction Intervals with respect to a design set following the uniform distribution and a Gaussian Process with correlation lengths 0.4 – 0.8 – 0.2 (top) and 0.8 – 1.0 – 0.9 (bottom).}
\label{fig:IP_violons}
\end{figure}

\subsection{A higher-dimensional case} \label{Sec:Ackley}

In this subsection, we emulate using Simple Kriging the 10-dimensional Ackley function:

\begin{equation}
A(\bs{x}) = 20 + \exp(1) - 20 \exp \left( - 0.2 \sqrt{ \frac{1}{10} \sum_{i=1}^{10} x_i^2} \right) - \exp \left( \frac{1}{10} \sum_{i=1}^{10} \cos(2 \pi x_i) \right).
\end{equation}

The goal in this section is to emulate the Ackley function on the unit hypercube $[0,1]^{10}$ using design sets with 100 observation points. Although the impact of the design set type is not the focus of this study, we present the results with a randomly chosen design according to the Uniform distribution on the domain $[0,1]^{10}$, a design obtained through Latin Hypercube Sampling (LHS), and a design obtained through LHS and subsequently optimized to maximize the minimum distance between two points. The Simple Kriging model uses the null function as mean function and the Matérn anisotropic geometric covariance kernel family with smoothness parameter 5/2. The Gibbs reference posterior distribution is accessed through a sample of 1000 points. The conditional densities are sampled using the Metropolis algorithm with normal instrumental density with standard deviation $0.4$ and a 100-step burn-in period. \medskip

To evaluate the performance of prediction intervals, we follow steps 3, 4, 5, 6 and 8 of the method presented in this section (step 7 is skipped as the ``values of the Gaussian process'' are naturally the values of the Ackley function) with $n_0=1000$. The results are presented in Tables \ref{Tab:Ackley_couverture} and \ref{Tab:Ackley_longueur}. 

\begin{table}[!ht]
\begin{center}
\begin{tabular}{|c c c c|}
\hline 
\textbf{Design set type} & \textbf{MLE} & \textbf{MAP} & \textbf{FPD} \\ 
\hline 
Unoptimized LHS &  0.89 & 0.92 & 0.93 \\
Optimized LHS &  0.74 & 0.76 & 0.80 \\
Random design & 0.87 & 0.88 & 0.91 \\
\hline
\end{tabular} 
\caption{Coverage of 95 \% prediction intervals when emulating the Ackley function on the unit hypercube using a Gaussian Process with null mean function and a Matérn anisotropic geometric covariance kernel with smoothness 5/2, unknown variance parameter and unknown vector of correlation lengths. The design sets contain 100 points. }
\label{Tab:Ackley_couverture}
\end{center}
\end{table}

As is shown in Table \ref{Tab:Ackley_couverture}, prediction intervals derived using the Full Posterior Distribution perform better than those derived from the MAP, which themselves perform better than those derived from the MLE. This order of performance is the same regardless of the type of design set, although the optimized design set leads to much worse performances on average for prediction intervals than unoptimized designs. The latter fact is not surprising since space-filling designs ensure than no two points can be very close to each other, which makes it harder to determine the correlation lengths.

\begin{table}[!ht]
\begin{center}
\begin{tabular}{|c c c c|}
\hline 
\textbf{Design set type} & \textbf{MLE} & \textbf{MAP} & \textbf{FPD} \\ 
\hline
Unoptimized LHS  & 0.31 & 0.33 & 0.35 \\
Optimized LHS & 0.24 & 0.24 & 0.28   \\
Random design & 0.28 & 0.29 & 0.32 \\
\hline
\end{tabular}
\caption{Mean length of 95 \% prediction intervals when emulating the Ackley function on the unit hypercube using a Gaussian Process with null mean function and a Matérn anisotropic geometric covariance kernel with smoothness 5/2, unknown variance parameter and unknown vector of correlation lengths. The design sets contain 100 points. }
\label{Tab:Ackley_longueur}
\end{center}
\end{table}

As expected, prediction intervals derived from the Full Posterior Distribution are on average longer than those derived from the MAP and a fortiori the MLE. Notice that prediction intervals are on average shorter with the optimized design set, which explains the poorer performances in terms of coverage.

\section{Conclusion and Perspectives}

We provided theoretical foundation to the claim that the stationary distribution of the Markov chain underlying PIGS with random scanning order is the optimal compromise between the potentially incompatible conditional distributions. \medskip

This theory is mainly derived from intuitive conceptions of what a compromise should
be. In places where such conceptions were inconclusive, we relied on concrete examples to
precisely determine what was acceptable and what was not in a compromise and used it to
complete the definition.
One strength of this theory is that it can be applied to continuous as well as discrete probability
distributions, whereas previous studies focused on the discrete, or even finite, case. \medskip

A question that remains open outside the finite-state case is how compatibility of conditional distributions is to be checked in practice. Although compromises are useful, not needing them is better. \medskip

Further investigation is needed to fully understand the properties of the optimal compromise. Nevertheless, its invariance by reparametrization and its respect of pairwise independence show that it preserves important features of the  conditional distributions. \medskip

The theory of optimal compromise suggests a framework for deriving a new objective posterior distribution based on the conditionals yielded by the reference prior theory on Simple Kriging parameters. Applying this framework to Matérn anisotropic kernels, we showed prediction to have good frequentist properties. \medskip

Future work should investigate whether this posterior distribution formally corresponds to some joint prior distribution. And if it does, how the joint prior distribution could be accessed, and how it relates to the conditional prior distributions. \medskip

Regarding the specific Kriging application presented in this paper, the next step is to extend the framework to Universal Kriging, where instead of being known, the mean function is only assumed to be a linear combination of known  functions $f_1,...,f_p$. The linear coefficients $\beta_1,...,\beta_p$ are then considered parameters of the model. This extension is of practical relevance, because the mean function can rarely be considered known. It can probably be done in the same way \citet{BDOS01} extended the reference prior from the Simple Kriging to the Universal Kriging framework: they used the flat improper prior as joint prior on $\beta_1,...,\beta_p$ conditional to $\sigma^2$ and $\bs{\theta}$ and used it to integrate $\beta_1,...,\beta_p$ out of the likelihood function, and then proceeded to derive the reference prior on $\sigma^2$ and $\bs{\theta}$ with respect to the integrated likelihood.  \medskip

A further extension would involve deriving an objective prior on the smoothness parameter $\nu$. 
In this endeavor, one should take into account the relation between correlation length $\bs{\theta}$ and smoothness $\nu$. Unfortunately, asymptotic theory is not of much help in this regard, as \citet{And10} shows that provided the spatial domain $\mathcal{D}$ is of dimension at least 5, then all parameters of the Matérn anisotropic geometric kernel are microergodic (\citet{Zha04} shows this to be untrue for spatial domains of dimension 1, 2 or 3, but the non-microergodic parameters are $\sigma^2$ and $\bs{\theta}$, not $\nu$). This means that the Gaussian measures on $\mathcal{D}$ corresponding to Gaussian Processes with two different smoothness parameters are orthogonal, which suggests that there exists a consistent estimator (the MLE possibly). \citet{Ste99} (section 6.6) considers the Fisher information on $\bs{\theta}$ and $\nu$, and gives examples (with a one-dimensional sample space $\mathcal{D}$) showing that the Fisher information on these parameters depends a lot on the design set. 
Fisher information relative to the smoothness parameter $\nu$ increases when design points are chosen to be close to one another (relative to the "true" correlation length $\bs{\theta}$), whereas Fisher information relative to correlation length $\bs{\theta}$ is maximized for design points that are farther apart. This, according to him, is coherent with the fact that $\bs{\theta}$ has greater influence on the low frequency behavior of the Matérn kernel while $\nu$ has greater influence on its high frequency behavior. This also suggests to us that the smoothness parameter $\nu$, like the variance parameter $\sigma^2$, can only be meaningfully estimated if the vector of correlation lengths $\bs{\theta}$ is known. Otherwise, the estimator could hardly tell which design points are close to each other, which intuitively seems a prerequisite to evaluating the smoothness of the process. If we wish to apply the reference prior algorithm to the case where $\nu$ is unknown, we should thus probably derive the reference prior on $\nu$ conditional to $\bs{\theta}$. 

\section*{Acknowledgements}
The author would like to thank his PhD advisor Professor Josselin Garnier (École Polytechnique, Centre de Mathématiques Appliquées) for his guidance, Loic Le Gratiet (EDF R\&D, Chatou) and Anne Dutfoy (EDF R\&D, Saclay) for their advice and helpful suggestions.
I also thank the Associate Editor and both anonymous reviewers for their constructive criticism which considerably improved this paper.

The author acknowledges the support of the French Agence Nationale de la Recherche (ANR), under grant ANR-13-MONU-0005 (project CHORUS).

\pagebreak

\begin{appendices}

\section{Proofs of Section \ref{Sec:compromis_optimal}} \label{App:compromis_optimal}

\begin{proof}[Proof of Proposition \ref{Prop:set_compromises_convex}]
Let $P^{(0)}$ and $P^{(1)}$ be two compromises between $(\pi_i)_{i \in [\![1,r]\!]} $ and let $t \in (0,1)$. Let us check that $(1-t) P^{(0)} + t P^{(1)}$ verifies Conditions (1) and (2) from Definition \ref{Def:compromis}. \medskip

For every $i \in [\![1,r]\!]$, $\pi_i \left( (1-t) P^{(0)} + t P^{(1)} \right)_{-i} = (1-t) \pi_i P_{-i}^{(0)} + t \pi_i  P_{-i}^{(1)}$. \medskip

Let $A$ be a measurable set such that $\left( (1-t) P^{(0)} + t P^{(1)} \right) (A) = 0$. Then $P^{(0)}(A) = 0$ and $P^{(1)}(A) = 0$. Because both $P^{(0)}$ and $P^{(1))}$ are compromises, for every $i \in [\![1,r]\!]$, $ \pi_i  P_{-i}^{(0)}  (A) = 0$ and $ \pi_i  P_{-i}^{(1)}  (A) = 0$. So $\left( \pi_i \left( (1-t) P^{(0)} + t P^{(1)} \right)_{-i} \right) (A) = 0$ and Condition (1) is verified. \medskip

Now, for every $i \in [\![1,r]\!]$,

\begin{align}
\frac{1}{r} \sum_{j=1}^r \left( \pi_j \left( (1-t) P^{(0)} + t P^{(1)} \right)_{-j} \right)_{-i} &= \frac{1-t}{r} \sum_{j=1}^r \left( \pi_j P_{-j}^{(0)} \right)_{-i} + \frac{t}{r} \sum_{j=1}^r \left( \pi_j P_{-j}^{(1)} \right)_{-i} \nonumber \\
&= (1-t) P_{-i}^{(0)} + t P_{-i}^{(1)} \nonumber \\
&= \left( (1-t) P^{(0)} + t P^{(1)} \right)_{-i}.
\end{align}

So Condition (2) is also verified.
\end{proof}

\begin{proof}[Proof of Proposition \ref{Prop:Gibbs_compromise_is_compromise}]
Let $P_G$ be a Gibbs compromise between the sequence of Markov kernels $(\pi_i)_{i \in [\![1,r]\!]}$. 
Then, for any measurable set $A$ such that $P_G(A)=0$,
\begin{equation}
P_G(A) = \frac{1}{r} \sum_{i=1}^r \pi_i (P_G)_{-i} (A) = 0.
\end{equation}
So for every integer $i \in [\![1,r]\!]$, $\pi_i (P_G)_{-i} (A) = 0$. This fulfills the first condition of Definition \ref{Def:compromis}. \medskip
 
Its second condition is also fulfilled because for every integer $i \in [\![1,r]\!]$,
\begin{equation}
(P_G)_{-i} = \left( \frac{1}{r} \sum_{j=1}^r \pi_j (P_G)_{-j} \right)_{-i} = \frac{1}{r} \sum_{j=1}^r \left( \pi_j (P_G)_{-j} \right)_{-i}.
\end{equation}
\end{proof}

\begin{proof}[Proof of Proposition \ref{Prop:hypermarginales_fixes}]

Define the probability distribution $P$ on $\mathcal{A}$ as follows:

\begin{equation} \label{Eq:melange}
P = \frac{1}{r} \sum_{i=1}^r \pi_i m_{\neq i}.
\end{equation}

Then for every $i \in [\![1,r]\!]$ the $i$-th $(r-1)$-marginal distribution $P_{-i}$ of $P$ is given by

\begin{equation} \label{Eq:hypermarginales}
P_{-i} = \frac{1}{r} \sum_{j=1}^r \left( \pi_j m_{\neq j} \right)_{-i} = m_{\neq i},
\end{equation}

where the last equality is due to $(m_{\neq i})_{i \in [\![1,r]\!]}$ being compatible with $(\pi_i)_{i \in [\![1,r]\!]}$. Plugging this into Equation \eqref{Eq:melange}, we obtain that $P$ is a Gibbs compromise. 
Equation \eqref{Eq:hypermarginales} then yields the result.

\end{proof}

\begin{proof}[Proof of Theorem \ref{Thm:Gibbs_compromis_optimal}]
If $P_G$ is the unique Gibbs compromise between the sequence of Markov kernels $(\pi_i)_{i \in [\![1,r]\!]}$, then Proposition \ref{Prop:hypermarginales_fixes} asserts that $\left( (P_G)_{-i} \right)_{i \in [\![1,r]\!]}$ is the only compatible sequence of $(r-1)$-dimensional distributions. So any compromise has the same sequence of $(r-1)$-marginal distributions. Let $P$ be such a compromise.\medskip

For every $i \in [\![1,r]\!]$, $\pi_i (P_G)_{-i} = \pi_i P_{-i}$ is absolutely continuous with respect to $P$
. So the average $P_G$ is also absolutely continuous with respect to $P$. \medskip

Let $\lambda$ be a positive measure on $\mathcal{A}$ such that $P$ is absolutely continuous with respect to $\lambda$.
Because $P_G$ and $P$ share the same sequence of $(r-1)$-marginal distributions, for $\lambda$-almost any $\omega \in \Omega$, 
\begin{equation} \label{Eq:same_marginal_derivative}
\frac{d (\pi_i P_{-i})}{d \lambda} (\omega) = \frac{d (\pi_i (P_G)_{-i})}{d \lambda} (\omega).
\end{equation}

Moreover, Equation \eqref{Eq:melange} implies that for $\lambda$-almost any $\omega \in \Omega$, $ \frac{dP_G}{d \lambda}(\omega)$ is the arithmetic average between the $\frac{d (\pi_i (P_G)_{-i})}{d \lambda}(\omega) $ ($i \in [\![1,r]\!]$), so it minimizes the mean squared error.
Together with Equation \eqref{Eq:same_marginal_derivative}, this implies that for $\lambda$-almost any $\omega \in \Omega$, 
\begin{equation} \label{Equation:mean_minimzes_squared_error}
\sum_{i=1}^r  \left[ \frac{d (\pi_i (P_G)_{-i})}{d \lambda}(\omega) - \frac{dP_G}{d \lambda}(\omega) \right]^2
\leqslant \sum_{i=1}^r  \left[ \frac{d (\pi_i P_{-i})}{d \lambda}(\omega) - \frac{dP}{d \lambda}(\omega) \right]^2.
\end{equation}

Consequently, $E_\lambda(P_G) \leqslant E_\lambda(P)$. Moreover, if $P \neq P_G$, then there exists $S \in \mathcal{A}$ such that $\lambda(S)>0$ and for every $\omega \in S$, 

\begin{equation}
\forall i \in [\![1,r]\!] \quad \frac{d (\pi_i P_{-i})}{d \lambda} (\omega) = \frac{d (\pi_i (P_G)_{-i})}{d \lambda} (\omega) \quad \mathrm{and} \quad \frac{dP}{d \lambda} (\omega) \neq \frac{d (P_G)}{d \lambda} (\omega).
\end{equation}

So for every $\omega \in S$, Equation \eqref{Equation:mean_minimzes_squared_error} is a strict inequality and thus $E_\lambda(P_G) < E_\lambda(P)$. $P_G$ is therefore the unique optimal compromise with respect to $\lambda$.
\end{proof}

\begin{proof}[Proof of Proposition \ref{Prop:invariant_reparametrisation}]
For every $i \in [\![1,r]\!]$, for every $\tilde{S}_i \in \tilde{\mathcal{A}}_i$,

\begin{align}
\tilde{P}_G \left( \bigtimes_{i \in [\![1,r]\!]} \tilde{S}_i \right)
&= P_G \left( f^{-1} \left(  \bigtimes_{i \in [\![1,r]\!]} \tilde{S}_i \right) \right) \nonumber \\
&= P_G \left( \bigtimes_{i \in [\![1,r]\!]} f_i^{-1} (\tilde{S}_i) \right)  \nonumber \\
&= \frac{1}{r} \sum_{i=1}^r \int_{ \bigtimes_{j \neq i} f_j^{-1} (\tilde{S}_j) } \pi_i(f_i^{-1} (\tilde{S}_i) | \omega_{-i})  d \left\{ (P_G)_{-i} \right\} (\omega_{-i}) \nonumber \\
&= \frac{1}{r} \sum_{i=1}^r \int_{ \bigtimes_{j \neq i} f_j^{-1} (\tilde{S}_j) } \tilde{\pi}_i(\tilde{S}_i | f_{-i}(\omega_{-i}) )  d \left\{ (P_G)_{-i} \right\} (\omega_{-i}) \nonumber \\
&= \frac{1}{r} \sum_{i=1}^r \int_{ \bigtimes_{j \neq i} \tilde{S}_j } \tilde{\pi}_i(\tilde{S}_i | \tilde{\omega}_{-i} )  d \left\{ (P_G)_{-i} \ast f_{-i} \right\} (\tilde{\omega}_{-i}) . \label{Eq:pushforward_Gibbs_compromise}
\end{align}

Now, for every $i \in [\![1,r]\!]$, for every $\tilde{T}_i \in \tilde{\mathcal{A}}_i$,

\begin{align}
(P_G)_{-i} \ast f_{-i} \left( \bigtimes_{j \neq i} \tilde{T}_j \right)
&= (P_G)_{-i} \left( \bigtimes_{j \neq i} f_j^{-1} (\tilde{T}_j)  \right) \nonumber \\
&= P_G \left( \bigtimes_{j<i} f_j^{-1} (\tilde{T}_j) \times f_i^{-1} (\tilde{\Omega}_i) \times \bigtimes_{k>i} f_k^{-1} (\tilde{T}_k) \right) \nonumber \\
&= P_G \ast f \left( \bigtimes_{j<i} \tilde{T}_j \times \tilde{\Omega}_i \bigtimes_{k>i} \tilde{T}_k \right) \nonumber \\
&= (P_G \ast f)_{-i} \left( \bigtimes_{j \neq i} \tilde{T}_j  \right).
\end{align}

So $(P_G)_{-i} \ast f_{-i} = (P_G \ast f)_{-i} = (\tilde{P}_G)_{-i}$. Then, returning to Equation \eqref{Eq:pushforward_Gibbs_compromise},

\begin{equation}
\tilde{P}_G \left( \bigtimes_{i \in [\![1,r]\!]} \tilde{S}_i \right)
= \frac{1}{r} \sum_{i=1}^r \int_{ \bigtimes_{j \neq i} \tilde{S}_j } \tilde{\pi}_i(\tilde{S}_i | \tilde{\omega}_{-i} )  d \left\{ (\tilde{P}_G)_{-i} \right\} (\tilde{\omega}_{-i}) . 
\end{equation}

Finally, we obtain

\begin{equation}
 \tilde{P}_G = \frac{1}{r} \sum_{i=1}^r \tilde{\pi}_i (\tilde{P}_G)_{-i}.
\end{equation}

$\tilde{P}_G$ is therefore a Gibbs compromise between the sequence of Markov kernels $(\tilde{\pi}_i)_{i \in [\![1,r]\!]}$. We now prove its uniqueness. For every $i \in [\![1,r]\!]$, $f_i$ is bijective and $f_i^{-1}$ is measurable. So for any Gibbs compromise $\tilde{Q}_G$ between the sequence of Markov kernels $(\tilde{\pi}_i)_{i \in [\![1,r]\!]}$, $\tilde{Q}_G \ast f^{-1}$ is a Gibbs compromise between the sequence of Markov kernels $(\pi_i)_{i \in [\![1,r]\!]}$. Given $P_G$ is the unique Gibbs compromise between the Markov kernels in this sequence, $\tilde{Q}_G \ast f^{-1} = P_G$, so $\tilde{Q}_G = \tilde{Q}_G \ast f^{-1} \ast f = P_G \ast f = \tilde{P}_G$.

\end{proof}

\section{Proofs of Section \ref{Sec:convergence_Gibbs_algorithm}} \label{App:noyaux_Matérn}

The following holds where there is no mention of the contrary.
When applied to a vector, $\| \cdot \|$ denotes the Euclidean norm and when applied to a matrix, it denotes the Frobenius norm.
The choice of norm does not matter much because in finite-dimensional vector spaces, all norms are equivalent.

\subsection{Differentiating the Matérn correlation kernel}

\begin{lem} \label{Lem:derivee_matern_tens}
The partial derivative with respect to $ \mu_i $ of the Matérn tensorized kernel of variance $ \sigma^2 $, smoothness $ \nu $ and inverse correlation length vector $ \bs{ \mu } $ is:
\begin{equation} \label{Eq:deriv_tens}
\frac{ \partial }{ \partial \mu_i } \left( \sigma^2 K_{r,\nu}^{tens} \left(  \bs{x} \bs{ \mu } \right)  \right) = - \frac{ \sigma^2 (2 \sqrt{\nu} )^2 }{ \Gamma(\nu) 2^{\nu-1} } 
|x_i|^2 \mu_i
\left( 2\sqrt{\nu} |x_i| \mu_i \right)^{\nu-1}
\mathcal{K}_{\nu-1} \left( 2 \sqrt{\nu} |x_i| \mu_i \right)
\prod_{j \neq i} K_{1,\nu} \left( |x_j| \mu_j \right).
\end{equation}
This can be rewritten as:
\begin{equation} \label{Eq:deriv_tens_3cas}
\frac{ \partial }{ \partial \mu_i } \left( \sigma^2 K_{r,\nu}^{tens} \left(  \bs{x} \bs{ \mu } \right)  \right) =  \left\{
    \begin{array}{lll}
         \sigma^2 \frac{2\nu}{\nu-1} |x_i|^2 \mu_i K_{1,\nu-1} \left( |x_i| \mu_i \right)  \prod_{j \neq i} K_{1,\nu} \left( |x_j| \mu_j \right)  &\mbox{if } \nu > 1\\
         \sigma^2 4  |x_i|^2 \mu_i
        \mathcal{K}_{0} \left( 2 |x_i| \mu_i \right)
		 \prod_{j \neq i} K_{1,\nu} \left( |x_j| \mu_j \right) & \mbox{if } \nu = 1 \\
		 \sigma^2 2 \nu^{\nu}  \frac{ \Gamma(1-\nu)}{\Gamma(\nu)} |x_i|^{2\nu} \mu_i^{2\nu-1}
		K_{1,1-\nu} \left( |x_i| \mu_i \right)
		 \prod_{j \neq i} K_{1,\nu} \left( |x_j| \mu_j \right) & \mbox{if } \nu < 1. \\
    \end{array}
\right.
\end{equation}
\end{lem}

\begin{proof}
The first assertion is a simple matter of differentiating Equation \eqref{Eq:Matern_tens}.
In the following calculation, the fourth line is given by formula 9.6.28 (page 376) in \citet{AS64}.

\begin{equation}
\begin{split}
\frac{ \partial }{ \partial \mu_i } \left( \sigma^2 K_{r,\nu}^{tens} \left(  \bs{x} \bs{ \mu } \right)  \right) &=
\sigma^2 \frac{ \partial }{ \partial \mu_i } \left(  K_{1,\nu} \left( x_i \mu_i \right)  \right)
\prod_{j \neq i} K_{1,\nu} \left( |x_j| \mu_j \right) \\
&= \sigma^2 x_i  \left(  K_{1,\nu}' \left( x_i  \mu_i \right)  \right)
\prod_{j \neq i} K_{1,\nu} \left( |x_j| \mu_j \right) \\
&= \sigma^2 x_i
 \left(  
\frac{ 2 \sqrt{\nu} }{ \Gamma(\nu) 2^{\nu-1} }
\left. \frac{d}{dy} \right|_{ y = 2 \sqrt{\nu} x_i \mu_i }
[ y^{\nu} \mathcal{K}_\nu (y) ]
  \right)
\prod_{j \neq i} K_{1,\nu} \left( |x_j| \mu_j \right) \\
&= \sigma^2 x_i
 \left(  
\frac{ 2 \sqrt{\nu} }{ \Gamma(\nu) 2^{\nu-1} }
[ -y \cdot y^{\nu-1} \mathcal{K}_{\nu-1} (y) ]_{ y = 2 \sqrt{\nu} x_i \mu_i }
  \right)
\prod_{j \neq i} K_{1,\nu} \left( |x_j| \mu_j \right). \\
\end{split}
\end{equation}

From there, Equation \eqref{Eq:deriv_tens} follows immediately. Rewriting it in the form given in (\ref{Eq:deriv_tens_3cas}) only requires us to recall $ \Gamma(\nu) = (\nu-1) \Gamma(\nu-1) $ (case $\nu>1$), $ \Gamma(1)=1 $ (case $\nu=1$) and $ \mathcal{K}_{\nu-1} = \mathcal{K}_{1-\nu} $ (case $\nu<1$).
\end{proof}

\begin{lem} \label{Lem:derivee_matern_anisgeom}
The partial derivative with respect to $ \mu_i $ of the Matérn geometric anisotropic kernel of variance $ \sigma^2 $, smoothness $ \nu $ and inverse correlation length vector $ \bs{ \mu } $ is:
\begin{equation} \label{Eq:deriv_anis_geom}
\frac{ \partial }{ \partial \mu_i } \left( \sigma^2 K_{r,\nu} \left( \bs{x} \bs{ \mu } \right)  \right) =  \frac{ \sigma^2 (2 \sqrt{\nu} )^2 }{ \Gamma(\nu) 2^{\nu-1} } 
|x_i|^2 \mu_i
\left( 2\sqrt{\nu} \left\| \bs{x} \bs{\mu} \right\| \right)^{\nu-1}
\mathcal{K}_{\nu-1} \left( 2 \sqrt{\nu} \left\| \bs{x} \bs{\mu} \right\| \right).
\end{equation}
This can be rewritten as:
\begin{equation} \label{Eq:deriv_anis_geom_3cas}
\frac{ \partial }{ \partial \mu_i } \left( \sigma^2 K_{r,\nu} \left( \bs{x} \bs{ \mu } \right)  \right) = \left\{
    \begin{array}{lll}
         \sigma^2 \frac{2\nu}{\nu-1} |x_i|^2 \mu_i K_{1,\nu-1} \left( \left\| \bs{x} \bs{\mu}  \right\| \right) & \mbox{if } \nu > 1\\
         \sigma^2 4 |x_i|^2 \mu_i
        \mathcal{K}_{0} \left( 2 \left\| \bs{x} \bs{\mu}  \right\| \right) & \mbox{if } \nu = 1 \\
		 \sigma^2 2 \nu^{\nu} \frac{ \Gamma(1-\nu)}{\Gamma(\nu)} 
		\frac{1}{\mu_i}
		\left( \frac{ |x_i| \mu_i }{ \| \bs{x}  \bs{\mu} \|^{1-\nu} } \right)^2
		K_{1,1-\nu} \left( \left\| \bs{x} \bs{\mu} \right\| \right) &\mbox{if } \nu < 1. \\
    \end{array}
\right.
\end{equation}
\end{lem}

\begin{proof}
The first assertion is a simple matter of differentiating Equation \eqref{Eq:Matern_anis_geom}.
In the following calculation, the fourth line is given by formula 9.6.28 (page 376) in \citet{AS64}.

\begin{equation}
\begin{split}
\frac{ \partial }{ \partial \mu_i } \left( \sigma^2 K_{r,\nu} \left( \bs{x} \bs{ \mu } \right)  \right) 
&= \sigma^2 \frac{ \partial }{ \partial \mu_i } \left( K_{1,\nu} \left( \left\| \bs{x} \bs{ \mu } \right\|  \right) \right) \\
&= \sigma^2 x_i^2 \mu_i \left\| \bs{x} \bs{ \mu } \right\|^{-1}
K_{1,\nu}' \left( \left\| \bs{x} \bs{ \mu } \right\|  \right) \\
&= \sigma^2 x_i^2 \mu_i \left\| \bs{x} \bs{ \mu } \right\|^{-1}
 \left(  
\frac{ 2 \sqrt{\nu} }{ \Gamma(\nu) 2^{\nu-1} }
\left. \frac{d}{dy} \right|_{ y = 2 \sqrt{\nu} \left\| \bs{x} \bs{ \mu } \right\| }
[ y^{\nu} \mathcal{K}_\nu (y) ]
  \right) \\
&= \sigma^2 x_i^2 \mu_i \left\| \bs{x} \bs{ \mu } \right\|^{-1}
 \left(  
\frac{ 2 \sqrt{\nu} }{ \Gamma(\nu) 2^{\nu-1} }
[ -y \cdot y^{\nu-1} \mathcal{K}_{\nu-1} (y) ]_{ y = 2 \sqrt{\nu} \left\| \bs{x} \bs{ \mu } \right\| }
  \right). \\
\end{split}
\end{equation}

From there, Equation \eqref{Eq:deriv_anis_geom} follows immediately. Rewriting it in the form given in (\ref{Eq:deriv_anis_geom_3cas}) only requires us to recall $ \Gamma(\nu) = (\nu-1) \Gamma(\nu-1) $ (case $\nu>1$), $ \Gamma(1)=1 $ (case $\nu=1$) and $ \mathcal{K}_{\nu-1} = \mathcal{K}_{1-\nu} $ (case $\nu<1$).
\end{proof}

\subsection{Accounting for low correlation: \texorpdfstring{$\norme \to \infty$}{Lg}}

In this subsection, we consider a fixed design set of $n$ coordinate-distinct points $\bs{x}^{(k)}$ ($k \in [\![1,n]\!]$) in $\R^r$. 

\begin{lem} \label{Lem:Matern_respecte_hyp}

For any Matérn anisotropic geometric or tensorized correlation kernel with smoothness $\nu>0$, 
 for all $b<2\min(1,\nu)-1$ 
 and $c>1$ (and if $\nu \neq 1$, for all $b \leqslant 2\min(1,\nu)-1$),
 \begin{enumerate}[(a)]
\item $ \forall \bs{ \mu } \in (\R_+)^r $, $ \| \deIrivee \| \leqslant M_{i,1} \; \mu_i^{-c}$.
\item $ \forall \bs{ \mu } \in (\R_+)^r $, $ \| \deIrivee \| \leqslant M_{i,2} \; \mu_i^b $.
\end{enumerate}
\end{lem}

\begin{proof}
This can be gathered from Lemma \ref{Lem:derivee_matern_tens} or \ref{Lem:derivee_matern_anisgeom} after recalling that 1) a Matérn kernel is a bounded function, 2) $\forall \nu \geqslant 0$, as $z \to +\infty$, $\mathcal{K}_\nu(z) \sim \sqrt{\pi} \exp (-z) / \sqrt{2z} $ (\cite{AS64} 9.7.2) and 3) as $z \to 0$, $\mathcal{K}_0(z) \sim - \log(z)$ (\cite{AS64} 9.6.8).
\end{proof}

Let us define 
\begin{align}
f_i( \mu_i \; | \; \manqueI ) &:=  \sqrt{[\bs{\mathcal{I}}( \bs{ \mu } )]_{ii} } \, ;\\
\pi_i( \mu_i \; | \; \manqueI )  &:= f_i( \mu_i \; | \; \manqueI ) / \int_0^\infty f_i( \mu_i=t \; | \; \manqueI ) dt.
\end{align}

\begin{prop} \label{Prop:densite_continue}
For any Matérn anisotropic geometric or tensorized correlation kernel with smoothness $\nu>0$, 
for all $\mu_i \in (0,+\infty)$, $\pi(\mu_i | \manqueI)$, 
seen as a function of $\bs{\mu}$, is well defined and continuous over $\{ \bs{\mu} \in [0,+\infty)^r : \mu_i \neq 0 , \, \manqueI \neq \bs{0}_{r-1} \}$.
\end{prop}

\begin{proof}
For any given $\tilde{\bs{\mu}} \in [0,+\infty)^r$ such that $\tilde{\mu}_i \neq 0$ and $\tilde{\bs{\mu}}_{-i} \neq \bs{0}_{r-1}$, we prove that $\pi(\mu_i | \manqueI)$, seen as a function of $\bs{\mu}$, is well defined and continuous at $\bs{\mu} = \tilde{\bs{\mu}}$.

For a start, notice that if $\bs{\mu}$ is confined to a sufficiently small neighborhood of $\tilde{\bs{\mu}}$, then $\| \coIrr^{-1} \|$ remains bounded.
Therefore, Lemma \ref{Lem:Matern_respecte_hyp} implies that $  \int_0^\infty f_i( \mu_i=t \; | \; \manqueI ) dt$ is finite and, thanks to the dominated convergence theorem, that it is continuous at $\bs{\mu}_{-i} = \tilde{\bs{\mu}}_{-i}$.

\end{proof}

\begin{defn}
An anisotropic geometric or tensorized correlation kernel is said to be ``well-behaved'' if its one-dimensional version is, for any set of parameters, a positive decreasing function on $[0,+\infty)$ that vanishes in the neighborhood of $+\infty$.
\end{defn}

\begin{lem} \label{Lem:bien_eleve}
Provided a coordinate-distinct design set is used, a well-behaved anisotropic geometric or tensorized correlation kernel parametrized by $\bs{\mu}$ has the following properties:
\begin{enumerate}[(a)]
\item for any fixed $\manqueI \in \left[0,+\infty\right)^{r-1}$, it is a decreasing function of $\mu_i$ ;
\item as $\| \bs{\mu} \| \rightarrow \infty$, $\| \coIrr - \bs{I}_n \| \rightarrow 0$.
\end{enumerate}
\end{lem}

\begin{lem} \label{Lem:dens_asymptotique}
For any well-behaved correlation kernel, as $\| \bs{\mu} \| \rightarrow \infty$, $\Tr \left[ \deIrivee \coIrr^{-1} \right] = o \left( \left\| \deIrivee \right\|  \right)$.
\end{lem}

\begin{proof}
This result is due to the fact that all $\deIrivee$'s diagonal coefficients are null and $\coIrr$ goes to the identity matrix as $\| \bs{\mu} \| \rightarrow \infty$.
\end{proof}

Let us now define

\begin{equation}
h_i( \mu_i \; | \; \manqueI ) :=  \sqrt{\Tr \left[ \left( \deIrivee \right)^2 \right] } = \left\| \deIrivee \right\|.
\end{equation}

\begin{lem} \label{Lem:dens_equivalente}
For any well-behaved correlation kernel, as $\| \bs{\mu} \| \rightarrow \infty$, $ f_i( \mu_i \; | \; \manqueI ) \sim h_i( \mu_i \; | \; \manqueI )$.
\end{lem}

\begin{proof}
Because $\coIrr$ goes to the identity matrix, this is a direct consequence of Lemma \ref{Lem:dens_asymptotique}.
\end{proof}

\begin{cor} \label{Cor:dens_encadrement}
For any well-behaved correlation kernel, there exist $S>0$, and $0<a<b$ such that, whenever $\| \bs{\mu} \| \geqslant S$, 
\begin{equation}
 a \; h_i( \mu_i \; | \; \manqueI ) \leqslant
f_i( \mu_i \; | \; \manqueI ) \leqslant
b \; h_i( \mu_i \; | \; \manqueI ).
\end{equation}
\end{cor}

In the following, $\coIrrI$ is the correlation matrix that would be obtained if $\mu_i$ were replaced by 0. Moreover, if $\bs{M}$ is a matrix, $\bs{M}^{(kl)}$ is its element in the $k$-th row and $l$-th column.

\begin{lem} \label{Lem:densite_minoree}
If a well-behaved correlation kernel is used, there exist real constants $S>0$ and $c>0$ such that, for all $\mu_i \in (0,+\infty)$ and whenever $\normeI \geqslant S$,
\begin{equation}
\pi_i(\mu_i | \manqueI) \geqslant c \frac{\| \deIrivee \|}{\sum_{k \neq l} \coIrrI^{(kl)}}.
\end{equation}

\end{lem}

\begin{proof}
If a well-behaved correlation kernel is used, then for any for any $\epsilon>0$, Corollary \ref{Cor:dens_encadrement} implies that

\begin{equation}
\int_0^{+\infty} f_i(\mu_i = t | \manqueI) dt \leqslant  
 b\int_0^{+\infty} h_i(\mu_i = t | \manqueI) dt \leqslant
- b \sum_{k \neq l} \int_0^{+\infty} \deIrivee^{(kl)} dt.
\end{equation}

The last inequality holds because the Frobenius norm of any matrix is smaller than or equal to the sum of the absolute values of its elements and the correlation kernel is a decreasing function of $\mu_i$.
Now, for all $k \neq l$, when $\mu_i \rightarrow +\infty$, $\coIrr^{(kl)} \rightarrow 0$ and when $\mu_i=0$, $\coIrr^{(kl)} = \coIrrI^{(kl)}$. From this, we gather that

\begin{equation}
\int_0^{+\infty} f_i(\mu_i = t | \manqueI) dt \leqslant
 b \sum_{k \neq l} \coIrrI^{(kl)}.
\end{equation}

From this, we deduce that

\begin{equation}
\pi_i(\mu_i | \manqueI) = \frac{f_i(\mu_i | \manqueI)}{\int_0^{+\infty} f_i(\mu_i = t | \manqueI) dt}
\geqslant \frac{a}{b} \frac{\| \deIrivee \|}{\sum_{k \neq l} \coIrrI^{(kl)}}.
\end{equation}
\end{proof}

This Lemma has the following immediate consequence:

\begin{prop} \label{Prop:grandmu_tens}
If a well-behaved tensorized kernel is used, there exists $S>0$ and for every $i \in [\![1,r]\!]$, there exists a function $M_i : (0,+\infty) \rightarrow (0,+\infty)$  such that for all $\normeI \geqslant S$, $\pi_i(\mu_i | \manqueI) \geqslant M_i(\mu_i)$.
\end{prop}

\begin{proof}
If a tensorized correlation kernel is used, for every pair of integers $(k,l) \in [\![1,r]\!]^2$ such that $k \neq l$, define the function $M_i^{(kl)} :(0,+\infty) \rightarrow (0,+\infty)$ ; $t \mapsto \left| \frac{d}{dt} \bs{\Sigma}_{\mu_i=t, \manqueI=\bs{0}_{r-1}}^{(kl)} \right|$. 
\begin{equation}
\left\| \deIrivee \right\| \geqslant \frac{1}{n} \sum_{k \neq l} \left| \deIrivee^{(kl)} \right| = \frac{1}{n} \sum_{k \neq l} M_i^{(kl)}(\mu_i) \coIrrI^{(kl)} \geqslant 
 \frac{1}{n} \min_{ k \neq l} M_i^{(kl)} (\mu_i)  \sum_{k \neq l}  \coIrrI^{(kl)}.
\end{equation}

This fact, joined with Lemma \ref{Lem:densite_minoree}, yields the result.
\end{proof}

\begin{prop} \label{Prop:grandmu_anis_geom}
Assume a well-behaved anisotropic geometric correlation kernel is used. If the corresponding one-dimensional kernel $K$ has the properties (P1) and (P2), then for every $i \in [\![1,r]\!]$, there exist positive functions $s_i$ and $m_i$ defined on $(0,+\infty)$ 
such that, for all $\normeI \geqslant s_i(\mu_i)$, $\pi_i(\mu_i | \manqueI) \geqslant m_i(\mu_i)$. \medskip

(P1) : There exist $S_1>0$ and $M_1>0$ such that, for all $t \geqslant S_1$, $|K'(t)| \geqslant M_1 t K(t)$. \medskip

(P2) : For any $a>0$, there exist $S_2(a)>0$ and $M_2(a)>0$ such that, whenever $t \geqslant S_2(a)$, $K(t+a) \geqslant M_2(a) K(t)$.
\end{prop}

\begin{proof}

From (P1), we gather that for all $a>0$ and $t \geqslant S_1$, $|K'(\sqrt{t^2 + a^2})| \geqslant M_1 \sqrt{t^2 + a^2} K(\sqrt{t^2+a^2})$. Now, because the correlation kernel is well-behaved, $K$ is a decreasing function. As $\sqrt{t^2 + a^2} \leqslant t+a$, $K(\sqrt{t^2+a^2}) \geqslant K(t+a)$. 

Plugging this into the previous inequality, we get $|K'(\sqrt{t^2 + a^2})| \geqslant M_1 \sqrt{t^2 + a^2} K(t+a)$.

If $t \geqslant \max(S_1,S_2(a))$, we can then use (P2) to obtain 
\begin{equation} \label{Eq:minoration_noyau_derive}
|K'(\sqrt{t^2 + a^2})| \geqslant M_1 M_2(a) \sqrt{t^2 + a^2} K(t).
\end{equation}
Independently from this, we have the following algebraic fact:

\begin{equation}
\left\| \deIrivee \right\| \geqslant \frac{1}{n} \sum_{k \neq l} \left| \deIrivee^{(kl)} \right|.
\end{equation}

Because we use a well-behaved anisotropic geometric kernel, defining the function $M_i^{(kl)} :(0,+\infty) \rightarrow (0,+\infty)$ ; $t \mapsto \left( x_j^{(k)} - x_j^{(l)} \right)^2$, we can write:

\begin{equation}
\left| \deIrivee^{(kl)} \right| = -  \deIrivee^{(kl)}
= \left( x_i^{(k)} - x_i^{(l)} \right)^2 \mu_i \frac{K'\left(\|(\bs{x}^{(k)} - \bs{x}^{(l)}) \bs{\mu} \|\right)}{\| (\bs{x}^{(k)} - \bs{x}^{(l)}) \bs{\mu} \| }.
\end{equation}

Setting $a_{kl}:=|x_i^{(k)} - x_i^{(l)}| \mu_i$ and $t_{kl}:= \|(\bs{x}_{-i}^{(k)} - \bs{x}_{-i}^{(l)}) \bs{\mu}_{-i}\|$ (and thus, naturally, $\sqrt{t_{kl}^2 + a_{kl}^2} = \|(\bs{x}^{(k)} - \bs{x}^{(l)}) \bs{\mu}\|$), and provided $\normeI$ is sufficiently large to make all $t_{kl}$s meet the conditions  necessary to apply (P1) and (P2) (that depend in the case of (P2) on the $a_{kl}$s), Equation \eqref{Eq:minoration_noyau_derive} yields the existence of some number $m_i^{(kl)}(\mu_i)>0$ such that 

\begin{equation}
\frac{K'\left(\|(\bs{x}^{(k)} - \bs{x}^{(l)}) \bs{\mu} \|\right)}{\| (\bs{x}^{(k)} - \bs{x}^{(l)}) \bs{\mu} \| }
\geqslant m_i^{(kl)}(\mu_i) K(\|(\bs{x}_{-i}^{(k)} - \bs{x}_{-i}^{(l)}) \bs{\mu}_{-i}\|)
= m_i^{(kl)}(\mu_i) \coIrrI^{(kl)}.
\end{equation}

Finally, setting $m_i(\mu_i) :=  \mu_i \min_{k \neq l} \left[ \left( x_i^{(k)} - x_i^{(l)} \right)^2 m_i^{(kl)}(\mu_i) \right] $, we get 

\begin{equation}
\left\| \deIrivee \right\| \geqslant \frac{m_i(\mu_i)}{n} \sum_{k \neq l} \coIrrI^{(kl)} .
\end{equation}

Then, applying Lemma \ref{Lem:densite_minoree} yields the result.
\end{proof}

\begin{prop} \label{Prop:grandmu_anis_geom_Matern}
Matérn one-dimensional kernels with smoothness parameter $\nu>1$ have the properties (P1) and (P2) of Proposition \ref{Prop:grandmu_anis_geom}.
\end{prop}
\begin{proof}
(P1) is given by Lemma \ref{Lem:derivee_matern_anisgeom}, after noticing that, denoting by $K_\nu$ the Matérn one-dimensional kernel of smoothness $\nu>1$, provided $t$ is sufficiently large, $K_\nu(t) \leqslant K_{\nu-1}(t)$. 

This inequality ensues from the fact that $\forall \nu \geqslant 0$, as $t \to +\infty$, $\mathcal{K}_\nu(t) \sim \sqrt{\pi} \exp (-t) / \sqrt{2t} $ (\cite{AS64} 9.7.2) and thus $K_\nu(t) \sim 2 / \Gamma(\nu) (\sqrt{\nu} t)^\nu \sqrt{ \pi / (4 \sqrt{\nu} t) } \exp ( - 2 \sqrt{\nu} t )$. Moreover, this last equivalence relation also implies (P2).
\end{proof}

\begin{prop} \label{Prop:minoration_prior}
For Matérn anisotropic geometric kernels with smoothness $\nu>1$ and Matérn tensorized correlation kernels with smoothness $\nu>0$,

for any $\delta>0$, $i\in [\![1,r]\!]$ and $\mu_i \in (0,+\infty)$, there exists $b_{i,\delta}(\mu_i)>0$ such that, if $\normeI \geqslant \delta$, then $\pi_i(\mu_i | \manqueI) \geqslant b_{i,\delta}(\mu_i)$.
\end{prop}

\begin{proof}
Matérn correlation kernels with such smoothness parameters make Proposition \ref{Prop:grandmu_tens} or \ref{Prop:grandmu_anis_geom} applicable. Therefore, there exist $s_i(\mu_i)>0$ and $m_i(\mu_i)>0$ such that, if $\normeI \geqslant s_i(\mu_i)$, $\pi_i(\mu_i | \manqueI) \geqslant m_i(\mu_i)$. Besides, we know from Proposition \ref{Prop:densite_continue} that $\pi_i(\mu_i | \manqueI)$, seen as a function of $\manqueI$, is continuous and positive over the compact set $\{ \manqueI : \delta \leqslant \normeI \leqslant s_i(\mu_i)\}$. Thus its minimum $\tilde{m}_{i,\delta}(\mu_i)$ on this set is positive and we obtain the result by setting $b_{i,\delta}(\mu_i) := \min(m_i(\mu_i),\tilde{m}_{i,\delta}(\mu_i))$.
\end{proof}

\begin{prop} \label{Prop:minoration_proba}
For Matérn anisotropic geometric correlation kernels with smoothness $\nu>1$ and for Matérn tensorized correlation kernels with smoothness $\nu>0$,
for any $\bs{y} \in \R^n \setminus \{0\}^n$, any $\delta>0$ and any $\mu_i \in (0,+\infty)$, there exists $b_{i,\delta,\bs{y}}(\mu_i)>0$ such that, if $\normeI \geqslant \delta$, then $\pi_i(\mu_i | \bs{y}, \manqueI) \geqslant b_{i,\delta, \bs{y}}(\mu_i)$.
\end{prop}
\begin{proof}
Set $\delta>0$ and $\bs{y} \in \R^n \setminus \{0\}^n$.
There exist $m_\delta>0$ and $M_\delta>0$ s.t. $\forall \bs{\mu} \in (0,+\infty)^r$, $\normeI \geqslant \delta \Rightarrow m_\delta \leqslant \| \coIrr^{-1} \| \leqslant M_\delta$, so there also exist $m_{\delta,\bs{y}}>0$ and $M_{\delta,\bs{y}}>0$ s.t. $m_{\delta,\bs{y}} \leqslant L(\bs{y}|\bs{\mu}) \leqslant M_{\delta,\bs{y}}$. This, combined with Proposition \ref{Prop:minoration_prior}, yields the result.
\end{proof}

\subsection{Accounting for high correlation: \texorpdfstring{$\norme \to 0$}{Lg}}

This part of the proof relies on the combination of some spectral study of the Matérn kernels and on the study of the matrices that are part of the series expansion of the correlation matrix $\coIrr$ when $\norme \to 0$ for three types of Matérn kernels: isotropic, tensorized and anisotropic geometric.

\begin{lem} \label{lemMinoreNoyauIso}
There exists a covariance kernel $\widetilde{K}_{r,\nu}$ such that for any design set $\bs{x}^{(1)},...,\bs{x}^{(n)}$, for all $\bs{\mu} \in (\R_+)^r$ and $\bs{\xi}=(\xi_1,...,\xi_n) \in \R^n$,
\begin{equation}
\sum_{j,k=1}^n \xi_j \xi_k  K_{r,\nu} \left(\left(\bs{x}^{(j)} - \bs{x}^{(k)}\right) \bs{\mu} \right) \geqslant 2^{-\frac{r}{2} - \nu} M_r( \nu ) f_{ r,\nu} ( \norme_{\infty} ) \sum_{j,k=1}^n \xi_j \xi_k  \widetilde{K}_{r,\nu} \left(\frac{ \bs{\mu} }{\norme_{\infty}} \left( \bs{x}^{(j)} - \bs{x}^{(k)} \right) \right)
\end{equation}
where $f_{r,\nu} (t) = (2 \sqrt{\nu})^{-r-2\nu} t^{-r} $ if $t \geqslant \left(2 \sqrt{\nu}\right)^{-1} $ and $f_{r,\nu} (t) = t^{2\nu} $ if $t \leqslant \left(2 \sqrt{\nu}\right)^{-1} $.
\end{lem}

\begin{proof}
For all $\bs{x},\bs{y} \in \R^r$, $K_{r,\nu} (\bs{x}-\bs{y}) = \int_{\R^r} \widehat{K}_{r,\nu} ( \bs{\omega} ) e^{ i \langle \bs{\omega} | \bs{x}-\bs{y} \rangle } d \bs{\omega} $.

\begin{equation}
\begin{split}
\sum_{j,k=1}^n \xi_j \xi_k  K_{r,\nu} \left(\left(\bs{x}^{(j)} - \bs{x}^{(k)}\right) \bs{\mu} \right)
&= \int_{\R^r} \widehat{K_{r,\nu}} (\bs{\omega}) \left| \sum_{j=1}^n \xi_j e^{i \langle \bs{\omega} \left| \bs{x}^{(j)}\bs{\mu} \right. \rangle }  \right|^2 d\bs{\omega} \\
&= M_r(\nu) \norme_{\infty}^{-r} \int_{\R^r} \left( 4 \nu + \norme_{\infty}^{-2} \| \bs{s} \|^2 \right)^{-\frac{r}{2} - \nu} \left| \sum_{j=1}^n \xi_j e^{i \langle \left. \frac{\bs{\mu}}{\norme_{\infty}} \bs{s} \right| \bs{x}^{(j)} \rangle } \right|^2 d \bs{s} \\
& \geqslant 2^{-\frac{r}{2} - \nu} M_r(\nu) f_{r,\nu} (\norme_{\infty}) \int_{\R^r \setminus B(0,1) }  \|\bs{s}\|^{-r-2\nu} \left| \sum_{j=1}^n \xi_j e^{i \langle   \bs{s} \left| \frac{\bs{\mu}}{\norme_{\infty}} \bs{x}^{(j)} \right. \rangle } \right|^2 d \bs{s}. \\
\end{split}
\end{equation}

Now, let $\widetilde{K}_{r,\nu}$ be the function with Fourier transform $\widehat{\widetilde{K}}_{r,\nu} (\bs{\omega}) = \mathbf{1}_{ \{\|\bs{\omega}\| \geqslant 1 \} } \| \bs{\omega} \|^{-r-2\nu}$. According to Bochner's theorem, $\widetilde{K}_{r,\nu}$ is a correlation kernel, which leads to the conclusion.
\end{proof}

\begin{lem} \label{lemMajor6Iso}
For every design set with coordinate-distinct points $\bs{x}^{(1)},...,\bs{x}^{(n)}$, there exists a constant $c_{\bs{x}}>0$ such that for all  $\bs{\mu} \in (\R_+)^r$,
\begin{equation}
\forall \bs{\xi}=(\xi_1,...,\xi_n) \in \R^n, \; \sum_{j,k=1}^n \xi_j \xi_k  K_{r,\nu} \left(\left(\bs{x}^{(j)} - \bs{x}^{(k)}\right) \bs{\mu} \right) \geqslant c_{\bs{x}} \| \bs{\xi} \|^2  2^{-\frac{r}{2} - \nu} M_r( \nu ) f_{ r,\nu} ( \norme_{\infty} ) 
\end{equation}
where $f_{r,\nu} (t) = (2 \sqrt{\nu})^{-r-2\nu} t^{-r} $ if $t \geqslant \left(2 \sqrt{\nu}\right)^{-1} $ and $f_{r,\nu} (t) = t^{2\nu} $ if $t \leqslant \left(2 \sqrt{\nu}\right)^{-1} $.
\end{lem}

\begin{proof}
For every design set $\bs{x}^{(1)},...,\bs{x}^{(n)}$, the set of all design sets that can be written $\frac{\bs{\mu}}{\norme_{\infty}} \bs{x}^{(1)},...,\frac{\bs{\mu}}{\norme_{\infty}} \bs{x}^{(n)}$ ($\bs{\mu} \in (\R_+)^r$) is compact. If the design set $\bs{x}^{(1)},...,\bs{x}^{(n)}$ has coordinate-distinct points, then every design set in the aforementioned compact set has no overlapping points. Thus the conclusion follows from Lemma \ref{lemMinoreNoyauIso}.
\end{proof}

\begin{prop} \label{Prop:vitesse_convergence_matern_anisotrope}
With Matérn anisotropic geometric or tensorized kernels, for every design set with coordinate-distinct points $\bs{x}^{(1)},...,\bs{x}^{(n)}$, as $\norme \to 0$, $\left\| \coIrr^{-1} \right\| = O(\norme^{-2\nu})$.
\end{prop}

\begin{proof}
For Matérn anisotropic geometric kernels, we need only apply Lemma \ref{lemMajor6Iso}. In the case of tensorized Matérn kernels, analoguous results to Lemma \ref{lemMinoreNoyauIso} and then Lemma \ref{lemMajor6Iso} may be used.
\end{proof}

\citet{AS64} give the following results on the modified Bessel function of second kind (usually noted $K_\nu$ and which we note $\mathcal{K}_\nu$ in order to avoid confusion with the Matérn correlation kernel). If $I_\nu$ is the modified Bessel function of first kind and $\psi$ is the function defined in (6.3.2) by $\psi : \mathbb{N} \setminus \{0\} \rightarrow \R$ ; $k \mapsto - \gamma + \sum_{i=1}^{k-1} i^{-1}$ :

\begin{align}
I_\nu(z) &= \left(\frac{1}{2} z\right)^{\nu} \sum_{k=0}^\infty \frac{\left(\frac{1}{4} z^2\right)^k}{k! \Gamma(\nu+k+1)} & \text{(9.6.10  in   \citep{AS64})} \nonumber \\
\mathcal{K}_\nu(z) &= \frac{1}{2} \pi \frac{I_{-\nu}(z) - I_\nu(z)}{\sin(\nu z)} \quad \text{if $\nu \notin \mathbb{Z}$}. & \text{(9.6.2  in   \citep{AS64})} \nonumber
\end{align}

This gives us the series expansion of $K_{1,\nu}(z)$ ($\nu \in [0,+\infty) \setminus \N$) when $z \rightarrow 0$:

\begin{equation}
\begin{split}
K_{1,\nu}(z) &= \frac{\pi}{\Gamma(\nu) \sin(\nu \pi)} \left( \sum_{0 \leqslant k < \nu} \frac{\nu^k z^{2k}}{k! \Gamma(-\nu+k+1)}
- \frac{\nu^\nu z^{2\nu}}{\Gamma(\nu+1)} + o\left(z^{2\nu}\right) \right) \\
&= \frac{\pi}{\Gamma(\nu) \sin(\nu \pi) \Gamma(-\nu+1)} \left( \sum_{0 \leqslant k < \nu} \frac{\Gamma(-\nu+1)}{k! \Gamma(-\nu+k+1)} \nu^k z^{2k}
- \frac{\Gamma(-\nu+1)}{\Gamma(\nu+1)} \nu^\nu z^{2\nu} + o\left(z^{2\nu}\right) \right) \\
&= \sum_{0 \leqslant k < \nu} \frac{\Gamma(-\nu+1)}{k! \Gamma(-\nu+k+1)} \nu^k z^{2k}
+ \frac{\Gamma(-\nu)}{\Gamma(\nu)} \nu^\nu z^{2\nu} + o\left(z^{2\nu}\right) \\
&= \sum_{0 \leqslant k < \nu} (-1)^k \frac{\Gamma(\nu-k)}{k! \Gamma(\nu)} \nu^k z^{2k}
+ \frac{\Gamma(-\nu)}{\Gamma(\nu)} \nu^\nu z^{2\nu} + o\left(z^{2\nu}\right).
\end{split}
\end{equation}

In the remainder of this subsection, we consider a fixed design set with $n$ coordinate-distinct points $\bs{x}^{(k)}$ ($k \in [\![1,n]\!]$) in $\R^r$. Moreover, all Matérn kernels we consider are assumed to have non-integer smoothness parameter $\nu$. \medskip

Let us now define, for every nonnegative integer $k<\nu$ the matrix $\bs{D}^k$ whose $(i,j)$ element is

\begin{equation}
\bs{D}^k(i,j) := (-1)^k \frac{\Gamma(\nu-k)}{k! \Gamma(\nu)} \nu^k \left\| \bs{x}^{(j)} - \bs{x}^{(k)} \right\|^{2k}.
\end{equation}

Let us also define the matrix $\bs{D}^\nu$ whose $(i,j)$ element is

\begin{align}
\bs{D}^\nu(i,j) &:= \frac{\Gamma(-\nu)}{\Gamma(\nu)} \nu^\nu \left\| \bs{x}^{(j)} - \bs{x}^{(k)} \right\|^{2\nu} &\mathrm{if} \; \nu \in [0,+\infty) \setminus \mathbb{N}. 
\end{align}

If the correlation kernel is Matérn isotropic, $\co$ has the following series expansion if $\nu$ is not an integer when $\mu \rightarrow 0+$:

\begin{equation} \label{Eq:devel_limite_co}
\co = \sum_{0 \leqslant k < \nu} \mu^{2k} \bs{D}^k + \mu^{2\nu} \bs{D}^\nu + \bs{R}_{\mu}.
\end{equation}

In this expansion, $\mu^{-2\nu} \|\bs{R}_{\mu}\| \rightarrow 0$.

For any integer $i \in [\![1,r]\!]$ and any nonnegative integer $k < \nu$ define 
the matrix $\bs{D}_i^k$ whose $(m,p)$ element is

\begin{equation}
\bs{D}_i^k(m,p) := (-1)^k \frac{\Gamma(\nu-k)}{k! \Gamma(\nu)} \nu^k \left| x_i^{(m)} - x_i^{(p)} \right|^{2k}
\end{equation}

and also the matrix $\bs{D}_i^{\nu}$ whose $(m,p)$ element is

\begin{align}
\bs{D}_i^\nu(m,p) &:= \frac{\Gamma(-\nu)}{\Gamma(\nu)} \nu^\nu \left| x_i^{(m)} - x_i^{(p)} \right|^{2\nu}. 
\end{align}

For every $i \in [\![1,r]\!]$, if the points in the design set differed only through their $i$-th coordinate, the series expansion of the correlation matrix (using a Matérn anisotropic geometric or tensorized kernel) when $\norme \to 0$ (and thus when $\mu_i \to 0$) would be 

\begin{equation}
\coI = \sum_{0 \leqslant k < \nu} \mu_i^{2k} \bs{D}_i^k + \mu_i^{2\nu} 
 \bs{D}_i^\nu + \bs{R}_{\mu_i}
\end{equation}

where $\mu_i^{-2\nu}$ 
$ \|\bs{R}_{\mu_i}\| \rightarrow 0$ as $\mu_i \to 0$.  \medskip 

Note the following identities:

\begin{align} \label{Eq:1D_matrices_DL}
\bs{D}_i^0 &= \bs{11} \trans; \\
\bs{D}_i^1 &= - \frac{\Gamma(\nu-1)}{\Gamma(\nu)} \nu 
\left\{ \bs{1} \left(\bs{X}_i^{\circ 2}\right)  \trans + \left(\bs{X}_i^{\circ 2}\right) \bs{1} \trans -2 \bs{X}_i \bs{X}_i \trans \right\}; \\
\bs{D}_i^2 &= \frac{\Gamma(\nu-2)}{\Gamma(\nu)} \nu^2 
\left\{ \bs{1} \left(\bs{X}_i^{\circ 4}\right)  \trans + \left(\bs{X}_i^{\circ 4}\right) \bs{1} \trans -4 \bs{X}_i \left( \bs{X}_i^{\circ 3} \right) \trans - 4 \left( \bs{X}_i^{\circ 3} \right) \bs{X}_i \trans + 6 \left( \bs{X}_i^{\circ 2} \right) \left( \bs{X}_i^{\circ 2} \right) \trans \right\} .
\end{align}

If a tensorized correlation kernel is used, the correlation matrix $\coIrr^{tens}$ may be written

\begin{equation}
\coIrr^{tens} = \prod_{i \in [\![1,r]\!]}^\circ \coI
\end{equation}

where the subscript $\circ$ above the symbol $\prod$ serves to denote the Hadamard product of matrices. \medskip

In case a Matérn anisotropic geometric kernel is used, then define for any nonnegative interger $k < \nu$ the matrix $\bs{D}^k (\bs{\mu})$ whose $(m,p)$ element is

\begin{equation}
\bs{D}^k(\bs{\mu}) (m,p) := (-1)^k \frac{\Gamma(\nu-k)}{k! \Gamma(\nu)} \nu^k d_{m,p} (\bs{\mu})^{2k}
\end{equation} 

where $d_{m,p}(\bs{\mu}) = \left\| \left(\bs{x}^{(m)} - \bs{x}^{(p)} \right) \bs{\mu} \right\|$. \medskip

And, similarly, we may define 
 the matrix $\bs{D}^{\nu} (\bs{\mu})$ whose $(m,p)$ element is
\begin{align}
\bs{D}^{\nu} (\bs{\mu}) (m,p) &:= \frac{\Gamma(-\nu)}{\Gamma(\nu)} \nu^\nu d_{m,p} (\bs{\mu})^{2\nu} &\mathrm{if} \; \nu \in [0,+\infty) \setminus \mathbb{N}. 
\end{align}

We thus have (if $\nu \in [0,+\infty) \setminus \mathbb{N}$)

\begin{equation} \label{Eq:anis_geom_DL}
\coIrr^{geom} = \sum_{0 \leqslant k < \nu} \bs{D}^k (\bs{\mu}) + \bs{D}^\nu (\bs{\mu}) + \bs{R}_{\bs{\mu}}^{geom}
\end{equation}

where $\norme^{-2\nu}$ 
$\|\bs{R}_{\bs{\mu}}^{geom}\| \rightarrow 0$ as $\norme \to 0$. \medskip

Similar identities to those of Equation \eqref{Eq:1D_matrices_DL} can be derived to make Equation \eqref{Eq:anis_geom_DL} more explicit for small values of $\nu$.

\begin{align} \label{Eq:multiD_matrices_DL}
\bs{D}^0 (\bs{\mu}) &= \bs{11} \trans; \\
\bs{D}^1 (\bs{\mu}) &= - \frac{\nu}{\nu-1}
\left\{ \sum_{i=1}^r \mu_i^2 \left( \bs{1} \left(\bs{X}_i^{\circ 2}\right)  \trans + \left(\bs{X}_i^{\circ 2}\right) \bs{1} \trans -2 \bs{X}_i \bs{X}_i \trans \right) \right\}; \\
\bs{D}^2 (\bs{\mu}) &= \frac{\nu^2}{(\nu-1)(\nu-2)} 
\left\{ \sum_{i,j \in [\![1,r]\!]} \mu_i^2 \mu_j^2 \left( \bs{1} \left(\bs{X}_i^{\circ 2} \circ \bs{X}_j^{\circ 2}\right)  \trans + \left(\bs{X}_i^{\circ 2} \circ \bs{X}_j^{\circ 2}\right) \bs{1} \trans \right. \right. \nonumber \\
& \qquad -2 \bs{X}_i \left( \bs{X}_i \circ \bs{X}_j^{\circ 2} \right) \trans - 2 \left( \bs{X}_i \circ \bs{X}_j^{\circ 2} \right) \bs{X}_i \trans -2 \bs{X}_j \left( \bs{X}_j \circ \bs{X}_i^{\circ 2} \right) \trans - 2 \left( \bs{X}_j \circ \bs{X}_i^{\circ 2} \right) \bs{X}_j \trans \nonumber \\
& \qquad \left. \left.
 + \left( \bs{X}_i^{\circ 2} \right) \left( \bs{X}_j^{\circ 2} \right) \trans + \left( \bs{X}_j^{\circ 2} \right) \left( \bs{X}_i^{\circ 2} \right) \trans + 4 \left( \bs{X}_i \circ \bs{X}_j \right) \left( \bs{X}_i \circ \bs{X}_j \right) \trans \right) \right\} .
\end{align}

Fortunately, for small values of $\nu$, $\coIrr^{tens}$ can also be simply written.

\begin{align} \label{Eq:tens_DL}
\mathrm{For} \; \nu \in (0,1): \quad \coIrr^{tens}&= \bs{11} \trans + \sum_{i=1}^r \mu_i^{2\nu} \bs{D}_i^{\nu} + \bs{R}_{\bs{\mu}}^{tens}. \\
\mathrm{For} \; \nu \in (1,2): \quad \coIrr^{tens}&= \bs{11} \trans + \bs{D}^1 (\bs{\mu}) + \sum_{i=1}^r \mu_i^{2\nu} \bs{D}_i^{\nu} + \bs{R}_{\bs{\mu}}^{tens} .\\
\mathrm{For} \; \nu \in (2,3): \quad \coIrr^{tens}&= \bs{11} \trans + \bs{D}^1 (\bs{\mu}) + \frac{\nu-2}{\nu-1} \bs{D}^2 (\bs{\mu}) + \sum_{i=1}^r \mu_i^4 \bs{D}_i^2 + \sum_{i=1}^r \mu_i^{2\nu} \bs{D}_i^{\nu} + \bs{R}_{\bs{\mu}}^{tens} .
\end{align}

In the three expressions above, $\norme^{-2\nu}$ 
$\|\bs{R}_{\bs{\mu}}^{tens}\| \rightarrow 0$ as $\norme \to 0$. \medskip 

Define $k_\nu$ as the orthogonal complement in $\R^n$ of the vector space spanned by:

\begin{enumerate}
\item if $\nu \in (0,1)$: $\bs{1}$;
\item if $\nu \in (1,2)$: $\bs{1}$ and $\bs{X}_i$ ($i \in [\![1,r]\!]$);
\item if $\nu \in (2,3)$: $\bs{1}$ and $\bs{X}_i$ ($i \in [\![1,r]\!]$) and $\bs{X}_i \circ \bs{X}_j$ ($i,j \in [\![1,r]\!]$).
\end{enumerate}

Clearly, for any $\nu \in (0,1) \cup (1,2) \cup (2,3)$, for any vector $\bs{v} \in k_\nu$,

\begin{align}
\bs{v} \trans \coIrr^{geom} \bs{v} &= \bs{v} \trans \bs{D}^\nu (\bs{\mu}) \bs{v} + \bs{v} \trans \bs{R}_{\bs{\mu}}^{geom} \bs{v} , \\
\bs{v} \trans \coIrr^{tens} \bs{v} &= \sum_{i=1}^r \mu_i^{2\nu} \bs{v} \bs{D}_i^\nu \bs{v} + \bs{v} \trans \bs{R}_{\bs{\mu}}^{tens} \bs{v}.
\end{align}

Since when $\mu \to 0$ $\| \bs{D}^\nu (\bs{\mu}) \| = O( \| \bs{\mu} \|^{2\nu} )$, $\| \bs{R}_{\bs{\mu}}^{geom} \| = o( \| \bs{\mu} \|^{2\nu} )$
and $\| \bs{R}_{\bs{\mu}}^{tens} \| = o( \| \bs{\mu} \|^{2\nu} )$, for any $\bs{\mu} \in (0,+\infty)^r$ such that $\| \bs{\mu} \|$ is small enough, there exists $c>0$ such that for any $\bs{v} \in k_\nu$, 

\begin{align} \label{Eq:norm_corr_restricted_knu}
\max(\bs{v} \trans \coIrr^{geom} \bs{v},\bs{v} \trans \coIrr^{tens} \bs{v})   \leqslant c \| \bs{\mu} \|^{2\nu} \bs{v} \trans \bs{v}.
\end{align}

\begin{prop} \label{Prop:norm_corr_bounded_by_ySigmay}
For a Matérn anisotropic geometric or tensorized correlation kernel with smoothness parameter $\nu \in (0,1) \cup (1,2) \cup (2,3)$, for any vector $\bs{y} \in \R^n$ not orthogonal to $k_\nu$, when $\| \bs{\mu} \| \to 0$,  $\| \bs{\mu} \|^{-2\nu} = O\left(\bs{y} \trans \coIrr^{-1} \bs{y} \right)$.
\end{prop}

\begin{proof}
Let $d_\nu$ be the dimension of $k_\nu$ and let $\bs{O}_\nu$ be an orthogonal $n \times n$ matrix whose first $n-d_\nu$ columns form an orthonormal basis of $k_\nu^{\perp}$ and whose last $d_\nu$ columns form an orthonormal basis of $k_\nu$. \medskip

Then $\coIrr = \bs{O}_\nu \bs{O}_\nu \trans \coIrr \bs{O}_\nu \bs{O}_\nu \trans$. Consider the following decomposition of $\bs{O}_\nu \trans \coIrr \bs{O}_\nu$:

\begin{equation}
\bs{O}_\nu \trans \coIrr \bs{O}_\nu = \begin{pmatrix} \bs{A}_{\bs{\mu}} & \bs{B}_{\bs{\mu}} \\ \bs{B}_{\bs{\mu}} \trans & \bs{C}_{\bs{\mu}} \\ \end{pmatrix}
\end{equation}

where the blocks $\bs{A}_{\bs{\mu}}$, $\bs{B}_{\bs{\mu}}$ and $\bs{C}_{\bs{\mu}}$ are respectively $(n-d_\nu) \times (n-d_\mu)$, $(n-d_\mu) \times d_\mu$ and $d_\mu \times d_\mu$ matrices. Note that $ \bs{A}_{\bs{\mu}} $ and $ \bs{C}_{\bs{\mu}} $ represent the restriction of the scalar product defined by $\coIrr$ to $k_\nu^{\perp}$ and $k_\nu$ respectively. When $\| \bs{\mu} \|$ is small enough, defining $c>0$ as in Equation \eqref{Eq:norm_corr_restricted_knu}, $\| \bs{C}_{\bs{\mu}} \| \leqslant c \| \bs{\mu} \|^{2\nu} $.

\begin{equation}
\bs{O}_\nu \trans \coIrr^{-1} \bs{O}_\nu = 
 \begin{pmatrix} \bs{I}_{n-d_\nu} & \bs{0} \\ - \bs{B}_{\bs{\mu}} \bs{C}_{\bs{\mu}}^{-1} & \bs{I}_{d_\nu} \end{pmatrix}
 \begin{pmatrix} \left(\bs{A} - \bs{B}_{\bs{\mu}} \bs{C}_{\bs{\mu}}^{-1} \bs{B}_{\bs{\mu}} \trans \right)^{-1} & \bs{0} \\ \bs{0} & \bs{C}_{\bs{\mu}}^{-1} \end{pmatrix}
 \begin{pmatrix} \bs{I}_{n-d_\nu} &  - \bs{C}_{\bs{\mu}}^{-1} \bs{B}_{\bs{\mu}} \trans \\ \bs{0} & \bs{I}_{d_\nu} \end{pmatrix}.
\end{equation}

For any vector $\bs{y} \in \R^n$, there exist $\bs{y}_1 \in \R^{n-d_\nu}$ and $\bs{y}_2 \in \R^{d_\nu}$ such that 

\begin{equation}
\bs{O}_\nu \trans \bs{y} = \begin{pmatrix} \bs{y}_1 \\ \bs{y}_2 \end{pmatrix},
\end{equation}

\begin{equation}
\begin{split}
\bs{y} \trans \coIrr^{-1} \bs{y} &= 
\begin{pmatrix} \bs{y}_1 - \bs{C}_{\bs{\mu}}^{-1} \bs{B}_{\bs{\mu}} \bs{y}_2 \\ \bs{y}_2 \end{pmatrix} \trans 
\begin{pmatrix} \left(\bs{A} - \bs{B}_{\bs{\mu}} \bs{C}_{\bs{\mu}}^{-1} \bs{B}_{\bs{\mu}} \trans \right)^{-1} & \bs{0} \\ \bs{0} & \bs{C}_{\bs{\mu}}^{-1} \end{pmatrix}
\begin{pmatrix} \bs{y}_1 - \bs{C}_{\bs{\mu}}^{-1} \bs{B}_{\bs{\mu}} \bs{y}_2 \\ \bs{y}_2 \end{pmatrix} .
\end{split}
\end{equation}

Given $\coIrr^{-1}$ is positive definite, the diagonal block $\left(\bs{A} - \bs{B}_{\bs{\mu}} \bs{C}_{\bs{\mu}}^{-1} \bs{B}_{\bs{\mu}} \trans \right)^{-1}$ is positive definite too. This implies
$
\bs{y} \trans \coIrr^{-1} \bs{y} \geqslant \bs{y}_2 \trans \bs{C}_{\bs{\mu}}^{-1} \bs{y}_2.
$
When $\| \bs{\mu} \|$ is small enough, $\bs{y}_2 \trans \bs{C}_{\bs{\mu}}^{-1} \bs{y}_2 \geqslant c^{-1} \| \bs{\mu} \|^{-2\nu} \| \bs{y}_2 \|^2$. \medskip

If $\bs{y}$ is not orthogonal to $k_\nu$, then $\| \bs{y}_2 \| \neq 0$ and thus $\| \bs{\mu} \|^{-2\nu} = O(\bs{y} \trans \coIrr^{-1} \bs{y})$.

\end{proof}

\begin{prop} \label{Prop:majoration_posterior_petitmu_grandnu}
Assume $\nu \in (1,2) \cup (2,3)$. For every $\bs{y} \in \R^n$ that is not orthogonal to the vector subspace
$k_\nu$,
$L(\bs{y} | \bs{\mu})f_i(\mu_i | \manqueI)$ is a bounded function of $\bs{\mu}$.
\end{prop}

\begin{proof}
Let $v_1(\bs{\mu}) \geqslant v_2(\bs{\mu}) \geqslant ... \geqslant v_n(\bs{\mu})$ be the ordered eigenvalues of $\coIrr$. We can now rewrite $L(\bs{y} | \bs{\mu})$ as

\begin{equation}
L(\bs{y} | \bs{\mu})^2 \propto \prod_{k=1}^n \left[ v_k(\bs{\mu})^{-1} \left( \bs{y} \trans \coIrr^{-1} \bs{y} \right)^{-1} \right] .
\end{equation}

Proposition \ref{Prop:norm_corr_bounded_by_ySigmay} asserts that for any $\bs{y} \in \R^n$ that is not orthogonal to $k_\nu$, $ \left( \bs{y} \trans \coIrr^{-1} \bs{y} \right)^{-1} = O(\norme^{2\nu})$ for $\norme \to 0$. 
Besides, Proposition \ref{Prop:vitesse_convergence_matern_anisotrope} asserts that $ \left\| \coIrr^{-1} \right\| = O( \norme^{-2\nu} )$, so $ \left( \bs{y} \trans \coIrr^{-1} \bs{y} \right)^{-1} = O \left( \left\| \coIrr^{-1} \right\|^{-1} \right)$. This implies that for every integer $i \in [\![1,r]\!]$, $v_k(\bs{\mu})^{-1} \left( \bs{y} \trans \coIrr^{-1} \bs{y} \right)^{-1} = O(1)$. \medskip

Clearly, $\lim_{\|\bs{\mu}\| \to 0} | \bs{1} \trans \bs{v}_1(\bs{\mu}) | = \| \bs{1} \|$ and $\lim_{\|\bs{\mu}\| \to 0} v_1(\bs{\mu}) = n$. The latter implies $\lim_{\|\bs{\mu}\| \to 0} v_1(\bs{\mu})^{-1} = n^{-1}$ and $v_1(\bs{\mu})^{-1} \left( \bs{y} \trans \coIrr^{-1} \bs{y} \right)^{-1} = O \left(\|\bs{\mu}\|^{2\nu} \right)$. Now, for every $\bs{\mu} \in (0,+\infty)^r$, we may (thanks to the axiom of choice) choose a unit eigenvector $\bs{v}_1(\bs{\mu})$ corresponding to the largest eigenvalue $v_1(\bs{\mu})$ and $\bs{v}_2(\bs{\mu})$ corresponding to the second largest eigenvalue $v_2(\bs{\mu})$. Because $\coIrr$ is symmetric, $\bs{v}_1(\bs{\mu}) \trans \bs{v}_2(\bs{\mu}) = 0$ for all $\bs{\mu} \in (0,+\infty)^r$, so $\lim_{\|\bs{\mu}\| \to 0} \bs{1} \trans \bs{v}_2(\bs{\mu}) = 0$. 

\begin{align}
v_2(\bs{\mu}) &= ( \bs{1} \trans \bs{v}_2(\bs{\mu}) )^2 + 2\nu (\nu-1)^{-1} \sum_{i=1}^r \mu_i^2 \left(\bs{X}_i \trans \bs{v}_2(\bs{\mu}) \right)^2 - 2 \mu_i^2 \left(\bs{1} \trans \bs{v}_2(\bs{\mu}) \right) \left(\bs{X}_i^{\circ 2 \top} \bs{v}_2(\bs{\mu}) \right)+ O \left(\|\bs{\mu}\|^4 \right) \nonumber \\
&\geqslant 2\nu (\nu-1)^{-1} \sum_{i=1}^r \mu_i^2 (\bs{X}_i \trans \bs{v}_2(\bs{\mu}))^2 + o \left(\|\bs{\mu}\|^2 \right).
\end{align}

For all $\bs{\mu} \in (0,+\infty)^r$, let $i(\bs{\mu})$ be the smallest integer $i \in [\![1,r]\!]$ such that $\mu_i = \max_{j=1}^r \mu_j$. Now for every integer $i \in [\![1,r]\!]$ let $\bs{w}_i(\bs{\mu})$ be the unit vector that belongs to the space spanned by $\bs{v}_1(\bs{\mu})$ and $\bs{X}_i$ that verifies $\bs{v}_1(\bs{\mu}) \trans \bs{w}_i(\bs{\mu})=0$ and $\bs{X}_i \trans \bs{w}_i(\bs{\mu}) > 0$.

\begin{equation}
\bs{w}_{i(\bs{\mu})}(\bs{\mu}) \coIrr \bs{w}_{i(\bs{\mu})}(\bs{\mu}) \geqslant 2 \nu (\nu-1)^{-1} r^{-1} \norme^2 (\bs{X}_{i(\bs{\mu})} \trans \bs{w}_{i(\bs{\mu})}(\bs{\mu}))^2 +  o \left(\|\bs{\mu}\|^2 \right).
\end{equation}

Because $\lim_{\|\bs{\mu}\| \to 0} | \bs{1} \trans \bs{v}_1(\bs{\mu}) | = \| \bs{1} \|$, $\liminf_{\norme \to 0} \bs{X}_{i(\bs{\mu})} \trans \bs{w}_{i(\bs{\mu})}(\bs{\mu}) \geqslant \min_{i=1}^r \lim_{\norme \to 0} \bs{X}_i \trans \bs{w}_i(\bs{\mu}) > 0$, so there exists a constant $c_2>0$ such that when $\norme$ is small enough

\begin{equation}
\bs{w}_{i(\bs{\mu})}(\bs{\mu}) \coIrr \bs{w}_{i(\bs{\mu})}(\bs{\mu}) \geqslant c_2 \norme^2.
\end{equation}

Recall $v_2(\bs{\mu}) = \max \{ \bs{\xi} \trans \coIrr \bs{\xi} | \bs{\xi} \in S^{n-1} \; \mathrm{and} \; \bs{\xi} \trans \bs{v}_1(\bs{\mu}) \}$, so a fortiori $v_2(\bs{\mu}) \geqslant c_2 \norme^2$. \medskip

This implies $v_2(\bs{\mu})^{-1} = O(\|\bs{\mu} \|^{-2})$ and therefore $v_2(\bs{\mu})^{-1} \left( \bs{y} \trans \coIrr^{-1} \bs{y} \right)^{-1}= O(\|\bs{\mu} \|^{2(\nu-1)})$ . \medskip

Finally, $L(\bs{y} | \bs{\mu}) = O \left( \| \bs{\mu} \|^{\nu} \right) O \left( \| \bs{\mu} \|^{\nu-1} \right) = O \left( \| \bs{\mu} \|^{2\nu-1} \right)$.
Given that $f_i(\mu_i | \manqueI) = O(\norme^{1 - 2\nu})$, the product $L(\bs{y} | \bs{\mu})f_i(\mu_i | \manqueI)$ is bounded when $\norme \to 0$.
\end{proof}

\begin{prop} \label{Prop:majoration_posterior_petitmu_petitnu}
Assume $\nu \in (0,1)$. For every $\bs{y} \in \R^n$ that is not collinear to $\bs{1}$, 
when $\norme \to 0$,
$L(\bs{y} | \bs{\mu})f_i(\mu_i | \manqueI) = O(\mu_i^{-1+\nu})$.
\end{prop}

\begin{proof}
This proof is similar to the previous one, so we use the same notations. $v_1(\bs{\mu})^{-1} = O(1)$, so $ v_1(\bs{\mu})^{-1} \left( \bs{y} \trans \coIrr^{-1} \bs{y} \right)^{-1} = O(\norme^{2\nu})$, which yields that 
$L(\bs{y} | \bs{\mu}) = O(\norme^{\nu})$. 

This implies that $L(\bs{y} | \bs{\mu}) \left\| \coIrr^{-1} \right\| = O(\norme^{-\nu}) = O(\mu_i^{-\nu})$.

By Lemma \ref{Lem:Matern_respecte_hyp},
$\left\| \deIrivee \right\| = O( \mu_i^{-1+2\nu} )$. 
Putting all this together, $L(\bs{y} | \bs{\mu})f_i(\mu_i | \manqueI) = O(\mu_i^{-1+\nu})$.
\end{proof}

\begin{prop} \label{Prop:densite_continue_0}

For Matérn anisotropic geometric or tensorized kernels with smoothness parameter $\nu \in (0,1) \cup (1,2) \cup (2,3)$, if $\bs{y} \in \R^n$ is not orthogonal to $k_\nu$, then
the conditional posterior distribution $\pi_i(\mu_i | \bs{y}, \bs{\mu}_{-i})$, seen as a function of $\bs{\mu}$, is continuous over $\{ \bs{\mu} \in [0,+\infty)^r : \mu_i \neq 0 \}$.

Moreover,
$$
\forall \mu_i > 0, \quad
\pi_i(\mu_i | \bs{y}, \bs{\mu}_{-i} = \bs{0}_{r-1}) = \frac{L(\bs{y} | \mu_i, \bs{\mu}_{-i} = \bs{0}_{r-1}) 
f_i(\mu_i | \bs{\mu}_{-i} = \bs{0}_{r-1})}
{\int_0^\infty L(\bs{y} | \mu_i=t, \bs{\mu}_{-i} = \bs{0}_{r-1}) f_i(\mu_i=t | \bs{\mu}_{-i} = \bs{0}_{r-1} ) dt } 
> 0.
$$
\end{prop}

\begin{proof}
Given Proposition \ref{Prop:densite_continue} and the fact that $\forall \bs{y} \in \R^n$, $\forall i \in [\![1,r]\!]$ and $\forall \mu_i \in (0,+\infty)$,
as $ \normeI \to 0$, $L(\bs{y} | \bs{\mu}) f_i(\mu_i | \bs{\mu}_{-i})$ converges pointwise to 
$L(\bs{y} | \mu_i, \bs{\mu}_{-1} = \bs{0}_{r-1}) f_i(\mu_i | \bs{\mu}_{-i} = \bs{0}_{r-1} )$, we only need to show that 

\begin{equation} \label{Eq:integrale_limite}
\int_0^\infty L(\bs{y} | \mu_i=t, \bs{\mu}_{-i}) f_i(\mu_i=t | \bs{\mu}_{-i}) dt \underset{\normeI \to 0}{\longrightarrow}
\int_0^\infty L(\bs{y} | \mu_i=t, \bs{\mu}_{-i} = \bs{0}_{r-1}) f_i(\mu_i=t | \bs{\mu}_{-i} = \bs{0}_{r-1} ) dt
< + \infty.
\end{equation}

Lemma \ref{Lem:Matern_respecte_hyp} implies that there exists $M_i>0$ such that 

\begin{equation}
\begin{split}
L(\bs{y} | \bs{\mu}) f_i(\mu_i | \bs{\mu}_{-i})
&= L(\bs{y} | \bs{\mu}) \left\| \coIrr^{-1} \right\| \left\| \coIrr^{-1} \right\|^{-1} f_i(\mu_i | \bs{\mu}_{-i}) \\
&\leqslant L(\bs{y} | \bs{\mu}) \left\| \coIrr^{-1} \right\| M_i \mu_i^{-2}.
\end{split}
\end{equation}

Lemma \ref{Lem:bien_eleve} then ensures $\left\| \coIrr^{-1} - \bs{I}_n \right\| \to 0$ as $\mu_i \to +\infty$, so the right member of the inequality is integrable in the neighborhood of $+\infty$. Let us now focus on the neighborhood of 0. \medskip

If $\nu>1$, Proposition \ref{Prop:majoration_posterior_petitmu_grandnu} asserts that $L(\bs{y} | \bs{\mu}) f_i(\mu_i | \bs{\mu}_{-i})$ is bounded in the neighborhood of 0.

If $\nu<1$, Proposition \ref{Prop:majoration_posterior_petitmu_petitnu} asserts that $L(\bs{y} | \bs{\mu}) f_i(\mu_i | \bs{\mu}_{-i}) \mu_i^{1-\nu}$ is bounded in the neighborhood of 0. \medskip

Therefore, there exists a function independent of $\manqueI$ that is both greater than $L(\bs{y} | \bs{\mu}) f_i(\mu_i | \bs{\mu}_{-i})$ and integrable over $\mu_i \in (0,+\infty)$, so the dominated convergence theorem is applicable.

\end{proof}

\subsection{Lower bound for conditional reference posterior densities}

The following Lemma provides the key to proving Theorem \ref{Thm:existence_posterior_gibbs}:

\begin{lem} \label{Lem:implique_theoreme}
In a Simple Kriging model with the characteristics described above, there exists a hyperplane $\mathcal{H}$ of $\R^n$ such that, for any $\bs{y} \in \R^n \setminus \mathcal{H}$ and any $i \in [\![1,r]\!]$,
there exists a measurable function $m_{i,\bs{y}}:(0,+\infty) \rightarrow (0,+\infty)$ such that, for all $\bs{\mu}_{-i} \in (0,+\infty)^{r-1}$, the conditional reference posterior density verifies:

\begin{equation}
\pi_i(\mu_i | \bs{y}, \bs{\mu}_{-i}) \geqslant m_{i,\bs{y}} (\mu_i) > 0.
\end{equation}

\end{lem}

\begin{proof}
This proof consists in combining Proposition \ref{Prop:minoration_proba} and Proposition \ref{Prop:densite_continue_0}, which respectively deal with large and small values of $\normeI$. \medskip

Proposition \ref{Prop:minoration_proba} implies that for any $\bs{y} \in \R^n \setminus \{0\}^n$, for any $i \in [\![1,r]\!]$ and any $\mu_i \in (0,+\infty)$, there exists a compact neighborhood $N_i(\mu_i)$ of $\bs{0}_{r-1}$ within $[0,+\infty)^r$ such that 

\begin{equation}
\inf \{\pi(\mu_i | \bs{y}, \manqueI) : \manqueI \in [0,+\infty)^{r-1} \setminus N_i(\mu_i)  \} > 0.
\end{equation}

The vector space $k_\nu \subset \R^n$ has dimension greater or equal to 

\begin{enumerate}[(a)]
\item $n-1$ if $\nu \in (0,1)$; 
\item $n-(r+1)$ if $\nu \in (1,2)$;
\item $n-(r+1)(r+2)/2$ if $\nu \in (2,3)$. 
\end{enumerate}

For all Simple Kriging models tackled by this lemma, the dimension of $k_\nu$ is therefore greater or equal to 1.
 Its orthogonal complement $k_\nu^{\perp}$ is then included within a hyperplane $\mathcal{H}$ of $\R^n$. Assuming $\bs{y} \in \R^n \setminus \mathcal{H} $, Proposition \ref{Prop:densite_continue_0} ensures that
$\pi(\mu_i | \bs{y}, \manqueI)$ is a continuous and positive function of $\bs{\mu}$ on $\{\bs{\mu} \in [0,+\infty)^r : \mu_i \neq 0 \}$. In particular, this implies that for any $\mu_i \in (0,+\infty)$ and any compact neighborhood $N_i(\mu_i)$ of $\bs{0}_{r-1}$ within $[0,+\infty)^r$, 

\begin{equation}
\inf \{\pi(\mu_i | \bs{y}, \manqueI) : \manqueI \in N_i(\mu_i)  \} > 0.
\end{equation}

Putting this together, if $\bs{y} \in \R^n \setminus \mathcal{H}$, for any $i \in [\![1,r]\!]$ and any $\mu_i \in (0,+\infty)$,

\begin{equation}
\inf \{\pi(\mu_i | \bs{y}, \manqueI) : \manqueI \in [0,+\infty)^{r-1}  \} > 0.
\end{equation}

The mapping $m_{i,\bs{y}}:\mu_i \mapsto \inf \{\pi(\mu_i | \bs{y}, \manqueI) : \manqueI \in [0,+\infty)^{r-1}  \}$, which is measurable on $(0,+\infty)$ is therefore also positive on $(0,+\infty)$.

\end{proof}

\begin{proof}[Proof of Theorem \ref{Thm:existence_posterior_gibbs}]

Lemma \ref{Lem:implique_theoreme} implies that $\forall \bs{\mu}^{(0)} \in (0,+\infty)^r$,

$$ P_{\bs{y}}(\bs{\mu}^{(0)}, d \bs{\mu}) \geqslant \frac{1}{r} \sum_{i=1}^r m_{i,\bs{y}} (\mu_i) d \mu_i \;  \delta_{\manqueI^{(0)}} (d \manqueI), $$

and thus $\forall \bs{\mu}^{(0)} \in (0,+\infty)^r$, $\forall n \geqslant r$,

$$ P_{\bs{y}}^n(\bs{\mu}^{(0)}, d \bs{\mu}) \geqslant \frac{1}{r^r} \prod_{i=1}^r m_{i,\bs{y}} (\mu_i) d \mu_i. $$

Defining $f_{\bs{y}}(\bs{\mu}) := r^{-r} \prod_{i=1}^r m_{i,\bs{y}} (\mu_i)$, $f_{\bs{y}}$ is a measurable positive function. Therefore $f_{\bs{y}}$ is the density with respect to the Lebesgue measure of a positive measure with mass $\epsilon_{\bs{y}}>0$. So $\epsilon_{\bs{y}}^{-1} f_{\bs{y}}$ is a probability density with respect to the Lebesgue measure and the Markov kernel $P_{\bs{y}}$ thus satisfies the uniform $(n,\epsilon_{\bs{y}})$ Doeblin condition:

\begin{equation}
 \forall \bs{\mu}^{(0)} \in (0,+\infty)^r \qquad P_{\bs{y}}^n(\bs{\mu}^{(0)}, d \bs{\mu}) \geqslant  \epsilon_{\bs{y}} \left( \frac{1}{\epsilon_{\bs{y}}} f_{\bs{y}} (\bs{\mu}) \right) d \bs{\mu} .
 \end{equation}

This implies that $P_{\bs{y}}$ is uniformly ergodic: it has a unique invariant probability distribution $\pi_G(\cdot | \bs{y})$ and $ \lim_{n \to \infty} \sup_{\bs{\mu}^{(0)} \in (0,+\infty)^r } \| P_{\bs{y}}^n (\bs{\mu}^{(0)},\cdot) - \pi_G(\cdot | \bs{y}) \|_{TV} = 0$, where $\| \cdot \|_{TV}$ is the total variation norm. By definition, $\pi_G(\cdot | \bs{y})$ is the Gibbs compromise between the incompatible posterior conditionals $\pi_i(\mu_i | \bs{y}, \manqueI)$.

\end{proof}

\end{appendices}

\pagebreak
\bibliography{biblio}
\bibliographystyle{plainnat}
\end{document}